\documentclass[12pt]{article}
\usepackage{amsmath,amssymb,amscd,amsthm}
\usepackage{mathtools}
\mathtoolsset{showonlyrefs}
\usepackage{color}
\usepackage[dvipdfmx,colorlinks,setpagesize=false]{hyperref}

\newtheorem{thm}{Theorem}[section]
\newtheorem{prop}[thm]{Proposition}
\newtheorem{lem}[thm]{Lemma}
\newtheorem{cor}[thm]{Corollary}

\theoremstyle{definition}
\newtheorem{defn}[thm]{Definition}
\newtheorem{exa}[thm]{Example}
\newtheorem{rmk}[thm]{Remark}
\newtheorem{para}[thm]{}

\newcounter{myitem}[thm]
\renewcommand{\themyitem}{\thethm.\arabic{myitem}}

\newenvironment{mylist}{
\setcounter{myitem}{\value{equation}}
\begin{list}{}{%
 \setlength{\topsep}{5pt}%
 \setlength{\parsep}{5pt}%
 \setlength{\itemsep}{0pt}%
 \setlength{\leftmargin}{50pt}%
 \setlength{\labelwidth}{50pt}%
}}
{\end{list}}

\newcommand{\itemno}{
\refstepcounter{myitem}
\item[{\rm (\themyitem)}]
\setcounter{equation}{\value{myitem}}}

\numberwithin{equation}{thm}

\newcommand{\RomI}{\uppercase\expandafter{\romannumeral 1}}
\newcommand{\RomII}{\uppercase\expandafter{\romannumeral 2}}
\newcommand{\RomIII}{\uppercase\expandafter{\romannumeral 3}}

\newcommand{\bC}{\mathbb C}

\newcommand{\bQ}{\mathbb Q}
\newcommand{\bZ}{\mathbb Z}
\newcommand{\bnnZ}{\mathbb Z_{\ge 0}}
\newcommand{\bpZ}{\mathbb Z_{> 0}}
\newcommand{\bH}{\mathbb H}
\newcommand{\bR}{\mathbb R}
\newcommand{\bpR}{\mathbb R_{> 0}}
\newcommand{\bV}{\mathbb V}
\newcommand{\ki}{1, 2, \dots, k}
\newcommand{\li}{1, 2, \dots, l}
\newcommand{\ski}{\{\ki\}}

\newcommand{\cA}{\mathcal A}

\newcommand{\cF}{\mathcal F}
\newcommand{\cL}{\mathcal L}
\newcommand{\cM}{\mathcal M}

\newcommand{\cO}{\mathcal O}

\newcommand{\cV}{\mathcal V}

\newcommand{\ca}{complex analytic }

\newcommand{\coh}{{\rm H}}

\newcommand{\gp}{^{\rm gp}}

\newcommand{\snc}{simple normal crossing }
\newcommand{\tB}{\widetilde{B}}

\DeclareMathOperator{\cok}{Coker}
\DeclareMathOperator{\dlog}{dlog}
\DeclareMathOperator{\e}{\mathbf e}
\DeclareMathOperator{\gr}{Gr}
\DeclareMathOperator{\id}{id}
\DeclareMathOperator{\image}{Image}
\DeclareMathOperator{\kernel}{Ker}
\DeclareMathOperator{\kos}{Kos}

\DeclareMathOperator{\pd}{\Delta^{\!\ast}}
\DeclareMathOperator{\pr}{pr}

\DeclareMathOperator{\rec}{rec}

\DeclareMathOperator{\res}{Res}

\DeclareMathOperator{\supp}{Supp}

\begin{document}

\title{Limits of Hodge structures in several variables, \RomII\footnotemark[0]}
\author{Taro Fujisawa \\
 \\
Tokyo Denki University \\
e-mail: fujisawa@mail.dendai.ac.jp}
%\address{Tokyo Denki University, Tokyo, Japan}
%\email{fujisawa@mail.dendi.ac.jp}
\footnotetext[0]
{\hspace{-18pt}2010 {\itshape Mathematics Subject Classification}.
Primary 14D07; Secondary 14C30, 32G20. \\
{\itshape Key words and phrases}.
 variation of Hodge structures, monodromy weight filtrations.}
%\subjclass[2010]{Primary 14D07; Secondary 14C30, 32G20.}
%\keywords{variation of Hodge structures, monodromy weight filtrations}

\maketitle

\begin{abstract}
The aim of this article is to study
degeneration of the variations of Hodge structure
associated to a proper semistable morphism
from a K\"ahler manifold.
We prove that the weight filtrations
constructed in \cite{FujisawaLHSSV}
coincide with the monodromy weight filtrations
on the relative log de Rham cohomology groups
of the proper semistable morphism.
Moreover, we show that
the limiting mixed Hodge structures
form admissible variations of mixed Hodge structure.
\end{abstract}

\section{Introduction}

\begin{para}
\label{question in the general form}
Let $Y$ be a complex manifold
and $E$ a reduced \snc divisor on $Y$.
A morphism $f: X \longrightarrow Y$ from a complex manifold $X$ to $Y$
is said to be semistable (along $E$)
if $D=f^{\ast}E$ is a {\itshape reduced} \snc divisor on $X$
and if $f$ is log smooth in the sense of Kato
\cite{KazuyaKato}
(for the precise definition,
see Definition \ref{definition for semistable morphism}).
Under the assumption for $X$ being a K\"ahler manifold,
a proper semistable morphism $f: X \longrightarrow Y$ along $E$
induces variations of Hodge structure
$R^qf_{\ast}\bQ_X|_{Y \setminus E}$
for all $q$,
because $f$ is smooth over $Y \setminus E$.
In this article,
as well as \cite{FujisawaLHSSV}, \cite{FujisawaDWSS},
we study the degeneration of these variations of Hodge structure
from the algebro-geometric viewpoint.
Because the question is of local nature,
we treat the case of $Y=\Delta^k \times S$
and $E=\{t_1t_2 \cdots t_k=0\}$,
where $\Delta^k$ is the $k$-dimensional polydisc
with the coordinate functions $t_1, t_2, \dots, t_k$
and $S$ is a complex manifold.
\end{para}

\begin{para}
\label{review of [S], [CKS] and so on}
For the case of $Y=\Delta^k$,
the degeneration of an abstract polarized variation of Hodge structure
on $Y \setminus E=(\pd)^k$
is intensively studied
by Cattani, Kaplan, Kashiwara, Kawai, Schmid and others,
apart from the morphism $f$.
Here we briefly review some of their results
for the case of unipotent monodromy.

Let $(\bV, (\cV, F))$ be
a polarizable variation of $\bR$-Hodge structure
of weight $m$
on $Y \setminus E=(\pd)^k$.
The universal covering
$\pi: \bH^k \longrightarrow Y \setminus E$
is given by
\begin{equation}
\pi^{\ast}t_i=\exp(2\pi\sqrt{-1}s_i), \quad i=\ki
\end{equation}
where $\bH^k$ is a product of the upper half plane $\bH$
with the coordinate $(s_1, s_2, \dots, s_k)$.
Then
$V_{\bR}=\Gamma(\bH^k, \pi^{-1}\bV)$
is a finite dimensional $\bR$-vector spaces.
The covering transformation $s_i \mapsto s_i+1$
induces the monodromy automorphism $T_i$ on $V_{\bR}$
for $i=1,2, \dots, k$.

In the following,
we assume that the local system $\bV$ is of unipotent monodromy,
i.e. $T_i$ is unipotent for all $i$.
The canonical extension of $\cV=\cO_{Y^{\ast}} \otimes \bV$
in the sense of Deligne \cite{DeligneED}
is denoted by $\widetilde{\cV}$.
For a subset $I \subset \ski$,
we set
\begin{equation}
\begin{split}
&Y[I]=\bigcap_{i \in I}\{t_i=0\} \\
&Y[I]^{\ast}=Y[I] \setminus (Y[I] \cap \bigcup_{j \notin I}\{t_j=0\})
\end{split}
\end{equation}
as in \ref{notation for Y and E} below.

By Theorem (4.12) and Theorem (6.16) in\cite{Schmid}
and by Theorem (3.3) in \cite{Cattani-Kaplan},
we have the following
(see also \cite[\S 3]{Steenbrink-Zucker}, \cite[\S 5]{Fujino-Fujisawa}):
\begin{mylist}
\itemno
\label{nilpotent orbit theorem}
The filtration $F$
extends to the canonical extension $\widetilde{\cV}$
with the property that $\gr_F^p\widetilde{\cV}$ is locally free
of finite rank for all $p$.
\itemno
\label{conditions for monodromy weight filtrations}
For a subset $K \subset \ski$,
there exists a unique finite increasing filtration $W(K)$ on $V_{\bR}$
such that the nilpotent endomorphism
$(\sum_{i \in K}c_i \log T_i)^l$
induces an isomorphism
\begin{equation}
\gr_l^{W(K)}V_{\bR}
\overset{\simeq}{\longrightarrow}
\gr_{-l}^{W(K)}V_{\bR}
\end{equation}
for any $l \in \bnnZ$ and for any $(c_i)_{i \in K}$
with $c_i \in \bpR$ for all $i \in K$.
\itemno
\label{conditions for relative monodromy weight filtrations}
For two subsets $K \subset J \subset \ski$,
the endomorphism $(\sum_{i \in J \setminus K}c_i \log T_i)^l$
induces an isomorphism
\begin{equation}
\gr_{l+a}^{W(J)}\gr_a^{W(K)}V_{\bR}
\overset{\simeq}{\longrightarrow}
\gr_{-l+a}^{W(J)}\gr_a^{W(K)}V_{\bR}
\end{equation}
for any $l \in \bnnZ$, for any $a \in \bZ$
and for any $(c_i)_{i \in J \setminus K}$
with $c_i \in \bpR$ for all $i \in J \setminus K$.
\itemno
\label{vmhs on Y[I]}
For a subset $I \subset \ski$,
the data
\begin{equation}
(\cO_{Y[I]} \otimes_{\cO_Y} \widetilde{\cV}, W(I)[m], F)|_{Y[I]^{\ast}}
\end{equation}
underlies an admissible variation of $\bR$-mixed Hodge structures
on $Y[I]^{\ast}$.
\end{mylist}
Here we remark that
$Y[I]^{\ast}=(\pd)^{\overline{I}}$
is the product of $k-|I|$ punctured discs
with the coordinate functions $t_i$
for $i \in \overline{I}=\ski \setminus I$,
and that the $\bR$-local system on $Y[I]^{\ast}$
associated to the variation of $\bR$-mixed Hodge structure
in \eqref{vmhs on Y[I]}
is given by the finite dimensional $\bR$-vector space $V_{\bR}$
equipped with the automorphisms
$(T_i)_{i \in \overline{I}}$.
\end{para}

\begin{para}
The goal of this article
is to carry out the de Rham theoretic realization
of \eqref{nilpotent orbit theorem}--\eqref{vmhs on Y[I]}
for the variation of $\bQ$-Hodge structures
induced by a proper semistable morphism
$f: X \longrightarrow Y$ from a K\"ahler manifold $X$
as in \ref{question in the general form}.

For the case of $k=1$,
Steenbrink proved this analogy in \cite{SteenbrinkLHS}
and \cite{SteenbrinkMHSVC}
(see also El Zein \cite{ElZeinCE},
Guillen-Navarro Aznar \cite{Guillen-NavarroAznarCI},
Saito \cite{SaitoMorihikoMHP},
Usui \cite{UsuiMTTS}).
For the case of $k \ge 2$,
parts of Steenbrink's results were generalized
in \cite{FujisawaLHSSV}.

Here, focusing on the case of $k \ge 2$,
we review previous results
and explain the relation between the results in this article
and the author's previous results
in \cite{FujisawaLHSSV} and \cite{FujisawaDWSS}.
For simplicity, we assume $Y=\Delta^k$
and $E=\{t_1t_2 \cdots t_k=0\}$ for a while.
Moreover, we simplify the presentation
at the cost of technical accuracy
in the remaining of this introduction.
\end{para}

\begin{para}
Let $X$ be a K\"ahler manifold and
$f: X \longrightarrow Y=\Delta^k$
a proper semistable morphism.
The divisor defined by the function $f^{\ast}t_i$ on $X$
is denoted by $D_i$ for $i=1,2, \dots, k$.
Then $D=f^{\ast}E=D_1+D_2+ \dots +D_k$.
We set
\begin{equation}
\bV=R^qf_{\ast}\bQ|_{Y \setminus E}, \quad
\cV=R^qf_{\ast}\Omega_{X/Y}|_{Y \setminus E}
\end{equation}
for a fixed $q \in \bZ$.
The filtration on $\cV$ induced by the stupid filtration on $\Omega_{X/Y}$
is denoted by $F$.
Then the data
$(\bV, (\cV, F))$
is a variation of $\bQ$-Hodge structure of weight $q$
on $Y^{\ast}=Y \setminus E=(\pd)^k$.
In \cite{UsuiRVCSV}, Usui stated that
\begin{equation}
\widetilde{\cV}=R^qf_{\ast}\Omega_{X/Y}(\log D)
\end{equation}
is the canonical extension of $\cV$,
which implies that the local system $\bV$ is of unipotent monodromy
(see Remark \ref{Usui's result} below).
Since $\widetilde{\cV}$ is locally free,
the canonical morphism
\begin{equation}
\label{base change iso:eq}
\cO_{Y[I]} \otimes \widetilde{\cV}
=\cO_{Y[I]} \otimes R^qf_{\ast}\Omega_{X/Y}(\log D)
\longrightarrow
R^qf_{\ast}(f^{-1}\cO_{Y[I]} \otimes_{f^{-1}\cO_Y} \Omega_{X/Y}(\log D))
\end{equation}
is an isomorphism for all $I \subset \ski$.
In particular,
we have the isomorphism
\begin{equation}
\label{base change iso over 0:eq}
\widetilde{\cV}(0)=\bC(0) \otimes \widetilde{\cV}
\simeq
R^qf_{\ast}(f^{-1}\bC(0) \otimes_{f^{-1}\cO_Y} \Omega_{X/Y}(\log D))
\end{equation}
for the case of $I=\ski$.
The stupid filtration on $\Omega_{X/Y}(\log D)$ induces
a finite decreasing filtration $F$ on $\widetilde{\cV}$
and on $\widetilde{\cV}(0)$.
In \cite{FujisawaLHSSV},
it was proved that
there exists a finite increasing filtration $L$ on
$\widetilde{\cV}(0)$
such that
$(\widetilde{\cV}(0), L, F)$ underlies a $\bQ$-mixed Hodge structure.
As a by-product,
it is also proved
that the filtration $F$ satisfies the property
\eqref{nilpotent orbit theorem},
i.e. $\gr_F^p\widetilde{\cV}$ is locally free of finite rank
for all $p$.
The difficulty comes from the fact that
it is impossible to define the filtration $L$
on $f^{-1}\bC(0) \otimes_{f^{-1}\cO_Y} \Omega_{X/Y}(\log D)$ directly.
In order to define the filtration $L$,
a complex
$sB(f)=sB_{\ski}(f)$
and a quasi-isomorphism
\begin{equation}
\label{qis for sB:eq}
f^{-1}\bC(0) \otimes_{f^{-1}\cO_Y} \Omega_{X/Y}(\log D)
\longrightarrow
sB(f) ,
\end{equation}
were constructed in \cite{FujisawaLHSSV}.
(Here we adopt the notation $sB(f)$
as in Definition \ref{definition of B_I(f)}
instead of $sB_X(D_1, \dots, D_k)$ used in \cite{FujisawaLHSSV}.)
By \eqref{base change iso over 0:eq} above,
we have the isomorphism 
\begin{equation}
\widetilde{\cV}(0) \simeq \coh^q(X_0, sB(f))
\end{equation}
where $X_0=f^{-1}(0)$ is the fiber of $f$
over the point $\{0\} \in \Delta^k$.
Then the filtration $L$ above was defined on $sB(f)$
in \cite{FujisawaLHSSV},
which induces the desired properties.
In fact, a filtration $L(K)$ on $sB(f)$
is defined for any subset $K \subset \ski$
(see Definition \ref{definition of L(K) on BI}),
and the filtration $L$ above is nothing but $L(\ski)$.
By interpreting Theorem (4.1) in \cite{FujisawaDWSS}
in terms of $(sB(f), L(K))$,
we obtain the $E_2$-degeneration of
the spectral sequence associated to the filtered complex $(sB(f), L(K))$
for any $K \subset \ski$.
\end{para}

\begin{para}
\label{point-wise vs local in intro}
The main results in \cite{FujisawaLHSSV} and \cite{FujisawaDWSS}
explained above are ``point-wise''
in the sense that
these results concern the complex $sB(f)$,
which computes $\widetilde{\cV}(0)$ as mentioned above.
In this article,
``local'' structure of $\widetilde{\cV}$
is to be investigated.
In other words, not only $\widetilde{\cV}(0)$ but also
$\cO_{Y[I]} \otimes \widetilde{\cV}$ are considered
for all $I \subset \ski$.
As in the case of $sB(f)$,
a complex $sB_I(f)$ and a quasi-isomorphism
\begin{equation}
f^{-1}\cO_{Y[I]} \otimes_{f^{-1}\cO_Y} \Omega_{X/Y}(\log D)
\longrightarrow
sB_I(f)
\end{equation}
for every $I \subset \ski$
are to be constructed
(see Definition \ref{definition of B_I(f)}
and Definition \ref{definition of theta} for the precise definition).
Then an increasing filtration $L(K)$ on $sB_I(f)$
is constructed for every $K \subset I$.
Combining with the isomorphism
\eqref{base change iso:eq},
we obtain the filtration $L(K)$
on $\cO_{Y[I]} \otimes \widetilde{\cV}$ for $K \subset I$.
(The compatibility of the filtration $L(K)$
with respect to the restriction is checked
in Proposition \ref{compatibility of the restriction and L(K)}.)
Then the main result of \cite{FujisawaLHSSV} recalled above
is generalized in Theorem \ref{theorem on VMHS for I},
which states that
\begin{equation}
\label{vmhs on YI* in intro:eq}
(\cO_{Y[I]} \otimes \widetilde{\cV}, L(I), F)|_{Y[I]^{\ast}}
\end{equation} 
underlies a graded polarizable variation of $\bQ$-mixed Hodge structures.
On the other hand,
the $E_2$-degeneration above for $(sB(f), L(K))$
is generalized in Theorem \ref{theorem for E2-degeneracy},
which states that the spectral sequence associated to
the filtered complex $(sB_I(f), L(K))$ degenerates at $E_2$-terms
for any $K \subset I$.
\end{para}

\begin{para}
\label{para:3}
As pointed out in Example 1.5 in \cite{Fujino-Fujisawa},
the local freeness of
\begin{equation}
\label{eq:9}
\gr_F^p\gr_m^{L(I)}(\cO_{Y[I]} \otimes \cV)
\end{equation}
is an indispensable condition
for the semipositivity theorem.
Theorem \ref{local freeness} below states such local freeness
in a more generalized form.

Theorem \ref{theorem on the monodromy weight filtrations},
the main result of this article,
states that the family of filtrations $\{L(J)\}_{J \subset I}$
on $R^n(f_I)_{\ast}sB_I(f)$ satisfies the properties
\eqref{conditions for monodromy weight filtrations}
and \eqref{conditions for relative monodromy weight filtrations}
for any $I \subset \ski$.
More precisely,
it states that
the morphism $N_{(J \setminus K)|I}(f;c)^l$
(for the definition,
see \eqref{definition of NJI(f;c):eq}
in Definition \ref{definition of N} below)
induces an isomorphism
\begin{equation}
\label{conditions for relative monodromy weight filtrations in intro:eq}
\gr_{l+m}^{L(J)}\gr_m^{L(K)}R^n(f_I)_{\ast}sB_I(f)
\overset{\simeq}{\longrightarrow}
\gr_{-l+m}^{L(J)}\gr_m^{L(K)}R^n(f_I)_{\ast}sB_I(f)
\end{equation}
for all $K \subset J \subset I \subset \ski$,
for all $l, m,n \in \bZ$ with $l \ge 0$
and for all $c=(c_i)_{i \in J \setminus K}$
with $c_i \in \bpR$ for all $i \in J \setminus K$
as in \eqref{conditions for relative monodromy weight filtrations}.
For the case of $K=\emptyset$,
the isomorphism above stands for
\eqref{conditions for monodromy weight filtrations}
because $L(\emptyset)$ is the trivial filtration.

The proof of this theorem is by induction on the dimension of $Y$.
Since the $E_1$-terms of the spectral sequence
associated to the filtered complex
$(R(f_I)_{\ast}sB_I(f), L(K))$
are described in terms of the data
obtained from proper semistable morphism
over a lower dimensional polydisc,
the induction hypothesis and
the $E_2$-degeneracy of this spectral sequence
implies the isomorphism
\eqref{conditions for relative monodromy weight filtrations in intro:eq}
for $K \not= \emptyset$.
Then the uniqueness of the relative monodromy weight filtration
implies that $L(J)$ coincides with the filtration $W(J)$
in \eqref{conditions for monodromy weight filtrations}
and \eqref{conditions for relative monodromy weight filtrations} as desired.
Then the admissibility
of the variation of $\bQ$-mixed Hodge structure
\eqref{vmhs on YI* in intro:eq} on $Y[I]^{\ast}$
follows from this theorem.
\end{para}

\begin{para}
By considering $Y=\Delta^k \times S$
instead of $Y=\Delta^k$,
we have technical advantage as follows.
For a subset $I \subset \ski$,
we often encounter the situation that
an object in question is to be considered over
$Y[I]^{\ast}=(\pd)^{\overline{I}} \times S$,
where $(\pd)^{\overline{I}}$ is the product of $k-|I|$ punctured discs
with the coordinate functions $t_i$
for $i \in \overline{I}=\ski \setminus I$
as in \ref{review of [S], [CKS] and so on}.
In such a case,
we may replace $Y$ by its open subset
$\Delta^I \times (\pd)^{\overline{I}} \times S$
and can start with $\Delta^I \times S'$
by setting $S'=(\pd)^{\overline{I}} \times S$,
where $\Delta^I$ is the $|I|$-dimensional polydisc
with the coordinate functions $t_i$ for $i \in I$.
Thus we may assume $I=\ski$ without loss of generality
in the case where the object in question
is considered over $Y[I]^{\ast}$.
\end{para}

\begin{para}
In \cite{GreenGriffithsDTLMHS},
Green and Griffiths
raised the question that the filtration $L$ in \cite{FujisawaLHSSV}
coincides with the monodromy weight filtration.
The main result of this article is the affirmative answer
to their question.

It is expected that
the results in this article
give the alternative proof of the fact
that $R^qf_{\ast}\Omega_{X/Y}(\log D)$ underlies
a logarithmic Hodge structure of weight $q$
in the sense of Kato-Usui
(cf. \cite[Definition 2.6.5]{KatoUsuiAMS}).
However, we will not discuss this question in this article.
\end{para}

\begin{para}
This article is organized as follows.
Section \ref{preliminaries}
presents some preliminaries.
For the later use,
we prove several lemmas concerning
compatibility of taking the cohomology
and the operations $\gr$ and $\otimes^L$
for a complex of sheaves.
In Section \ref{semistable morpshims},
we give the precise definition of the semistable morphism.
Moreover, we fix the notation
which is constantly used in this article.
In Section \ref{section for complex B},
we construct the complex $sB_I(f)$
and study its properties in detail.
Section \ref{gauss-manin}
is devoted to the computation of the Gauss-Manin connection
in terms of the complex $sB_I(f)$.
In Section \ref{section for the rational structure},
we construct rational structures
for $sB_I(f)$,
which play important roles
in the following sections.
Section \ref{section for the proper case}
deals with the case
where the semistable morphism $f$ is proper.
By using the rational structures
constructed in the last section,
we prove that
the higher direct image sheaves of $sB_I(f)$
are locally free of finite rank.
In Section \ref{section for the Kahler case},
we show that the data
\eqref{vmhs on YI* in intro:eq}
is a graded polarizable variation of $\bQ$-mixed Hodge structure
on $Y[I]^{\ast}$.
Moreover,
the $E_2$-degeneracy of the spectral sequence
and the local freeness of \eqref{eq:9}
mentioned in \ref{point-wise vs local in intro}
and \ref{para:3}
are proved.
In the final section, Section \ref{section for the main results},
we prove that our weight filtrations coincide
with the monodromy weight filtrations
and that the variation of mixed Hodge structure
\eqref{vmhs on YI* in intro:eq} on $Y[I]^{\ast}$ above
is admissible.
\end{para}

\subsection*{Notation}

\begin{para}
The cardinality of a finite set $A$
is denoted by $|A|$.
\end{para}

\begin{para}
Once we fix a set $A$, the complement of a subset $B \subset A$
is denoted by $\overline{B}$,
that is, $\overline{B}=A \setminus B$.
\end{para}

\begin{para}
\label{notation for ZI}
For a finite set $I$,
\begin{equation*}
\bZ^I=\bigoplus_{i \in I}\bZ e_i
\end{equation*}
is the free $\bZ$-module of rank $|I|$
generated by $\{e_i\}_{i \in I}$.
We set $e_I=\sum_{i \in I}e_i \in \bZ^I$.
For $q=\sum_{i \in I}q_ie_i \in \bZ^I$,
we set $|q|=\sum_{i \in I}q_i$.
We also use the convention
\begin{align}
&\bnnZ^I
=\{q=\sum_{i \in I}q_ie_i \in \bZ^I; q_i \ge 0 \text{ for all $i$}\} \\
&\bpZ^I
=\{q=\sum_{i \in I}q_ie_i \in \bZ^I; q_i > 0 \text{ for all $i$}\} .
\end{align}
For a subset $J$ of $I$,
we canonically have
\begin{equation*}
\bZ^I = \bZ^J \oplus \bZ^{I \setminus J}
\end{equation*}
which induces the canonical inclusion
$\bZ^J \longrightarrow \bZ^I$
and the canonical projection
$\bZ^I \longrightarrow \bZ^J$.
For the case of $I=\ski$
we use $\bZ^k$ instead of $\bZ^I$.
Similarly, we write $e$ instead of $e_I$.
As usual, we write $q=(q_1,q_2, \dots, q_k)$
for $q=\sum_{i=1}^{k}q_ie_i \in \bZ^k$.
We use the notation $\bQ^I$, $\bR^I$, $\bpR^I$, $\bC^I$ etc.
in the same way as above.
\end{para}

\begin{para}
Let $\Lambda$ be a set.
We denote the set of all the subsets of $\Lambda$
is denoted by $S(\Lambda)$.
Moreover, we set
\begin{equation*}
S_n(\Lambda)=\{\Gamma \in S(\Lambda); |\Gamma|=n\}
\end{equation*}
for $n \in \bZ$.
If a decomposition into disjoint union
\begin{equation*}
\Lambda=\coprod_{i=1}^k\Lambda_i
\end{equation*}
is given,
we set
\begin{align*}
&S^q(\Lambda)
=\{\Gamma \in S(\Lambda); |\Gamma \cap \Lambda_i|=q_i
\text{ for any $i \in J$}\} \\
&S^{\ge q}(\Lambda)
=\{\Gamma \in S(\Lambda); |\Gamma \cap \Lambda_i| \ge q_i
\text{ for any $i \in J$}\} \\
&S^{\le q}(\Lambda)
=\{\Gamma \in S(\Lambda); |\Gamma \cap \Lambda_i| \le q_i
\text{ for any $i \in J$}\}
\end{align*}
for $q \in \bZ^J$ and for $J \subset \ski$.
Moreover, we set
\begin{equation*}
S_n^q(\Lambda)=S_n(\Lambda) \cap S^q(\Lambda), \quad
S_n^{\ge q}(\Lambda)=S_n(\Lambda) \cap S^{\ge q}(\Lambda), \quad
S_n^{\le q}(\Lambda)=S_n(\Lambda) \cap S^{\le q}(\Lambda)
\end{equation*}
for $n \in \bZ$.
\end{para}

\begin{para}
\label{para:2}
Let $X$ be a topological space,
$Y \subset X$ a closed subspace,
$\iota: Y \hookrightarrow X$ the inclusion
and $\cO_Y$ a sheaf of rings on $Y$.
The direct image $\iota_{\ast}$ gives us
an equivalence of the following two abelian categories:
the category of $\cO_Y$-modules on $Y$
and the category of $\iota_{\ast}\cO_Y$-modules on $X$
whose support is contained in $Y$.
The quasi-inverse is given by $\iota^{-1}$.
Since the functors $\iota_{\ast}$ and $\iota^{-1}$ are exact functors,
this equivalence can be extended to the filtered objects.
Then the functors $\iota_{\ast}$ and $\iota^{-1}$
commute with the functor $\gr$ canonically.
Moreover the equivalence above also can be extended
to the (filtered) derived categories.
Taking into account of these equivalences,
we usually omit the symbols $\iota_{\ast}$ and $\iota^{-1}$
for $\cO_Y$-modules on $Y$
and for $\iota_{\ast}\cO_Y$-modules on $X$
whose support is contained in $Y$.
Hence a sheaf of $\iota_{\ast}\cO_Y$-module on $X$
whose support is contained in $Y$
is usually considered as
a sheaf of $\cO_Y$-module on $Y$
and vice versa.
We use the same convention for complexes.
%For a commutative diagram
%\begin{equation}
%\begin{CD}
%Y @>{\iota}>> X \\
%@V{g}VV @VV{f}V \\
%Y' @>>{\iota'}> X'
%\end{CD}
%\end{equation}
%where $X'$, $Y'$ and $\iota: Y' \longrightarrow X$
%is as $X$, $Y$ and $\iota$,
%and where $g: Y \longrightarrow Y'$ a morphism of ringed spaces,
%then we have
%$Rf_{\ast}i_{\ast}=\iota'_{\ast}Rg_{\ast}$.
%In other words, 
%$Rf_{\ast}=Rg_{\ast}$
%for $\cO_Y$-modules on $Y$, complex of $\cO_Y$-modules on $Y$ etc.,
%if we follow the convention above.
\end{para}

\begin{para}
For the tensor product of modules, sheaves, complexes etc.,
we omit the base ring and use the symbol $\otimes$ simply,
if there is no danger of confusion.
Similarly we use $\otimes^L$ for the derived tensor product.
\end{para}

\begin{para}
Let $X$ be a \ca space and $x \in X$ a point of $X$.
The residue field $\cO_{X,x}/\mathfrak{m}_x$ ($\simeq \bC$)
is denoted by $\bC(x)$.
The tensor product
$\bC(x) \otimes_{\cO_{X,x}} \cF_x$
is simply denoted by $\bC(x) \otimes \cF$
for an $\cO_X$-module $\cF$.
We use the same notation for a complex of $\cO_X$-modules.
Similarly, we simply use the symbol $\bC(x) \otimes^L$
for the derived tensor product.
\end{para}

\section{Preliminaries}
\label{preliminaries}

\begin{para}
In this section $(X, \cO_X)$ denotes a commutative and unitary ringed space,
i.e. $\cO_X$ is a sheaf of commutative and unitary rings on $X$.
We usually use the symbol $X$ instead of $(X, \cO_X)$
if there is no danger of confusion.
\end{para}

\begin{defn}
\label{defn:1}
Let $\cL$ and $\cM$ be $\cO_X$-modules on $X$
and $F$ a finite decreasing filtration on $\cM$.
Then we define a finite decreasing filtration $F$ on $\cL \otimes \cM$
by setting
\begin{equation}
F^p(\cL \otimes \cM)=\image(\cL \otimes F^p\cM \longrightarrow \cL \otimes \cM)
\end{equation} 
for all $p$.
\end{defn}

\begin{rmk}
We have the canonical surjective morphisms
\begin{align}
&\cL \otimes F^p\cM
\longrightarrow
F^p(\cL \otimes \cM)
\label{canonical surjection for tensor product and F:eq} \\
&\cL \otimes \gr_F^p\cM
\longrightarrow
\gr_F^p(\cL \otimes \cM)
\label{canonical surjection for tensor product and gr:eq}
\end{align}
by definition.
\end{rmk}

\begin{rmk}
If $\gr_F^p\cM$ is $\cO_X$-flat for all $p$,
then the canonical morphism
\eqref{canonical surjection for tensor product and F:eq}
is an isomorphism for all $p$.
Therefore the canonical morphism
\eqref{canonical surjection for tensor product and gr:eq}
is an isomorphism for all $p$.
In particular,
the canonical morphism
\eqref{canonical surjection for tensor product and gr:eq}
is an isomorphism for all $p$
if $\gr_F^p\cM$ is a locally free $\cO_X$-module of finite rank
for all $p$.
\end{rmk}

\begin{rmk}
For the case that $\cO_X$ is the constant sheaf $\bQ$,
the canonical morphisms
\eqref{canonical surjection for tensor product and F:eq}
and
\eqref{canonical surjection for tensor product and gr:eq}
are isomorphisms for all $p$.
\end{rmk}

\begin{para}
Let $\cL, \cM$ be $\cO_X$-modules
and $F, G$ finite decreasing filtrations on $\cM$.
Then $F, G$ induces
the filtrations $F, G$ on $\cL \otimes \cM$ respectively
as in Definition \ref{defn:1}.
Then the filtration $G$ on $\gr_F^p(\cL \otimes \cM)$
is induced for all $p$.
On the other hand,
$G$ induces the filtration $G$ on $\gr_F^p\cM$,
which induces the filtration $G$ on $\cL \otimes \gr_F^p\cM$
for all $p$.
Similarly, the filtration $F$ is induced
on $\gr_G^q\cM$ and on $\cL \otimes \gr_G^q\cM$ for all $q$.
It is easy to see
that the morphism
\eqref{canonical surjection for tensor product and gr:eq}
preserves the filtration $G$ on the both sides.
\end{para}

\begin{lem}
\label{lem:2}
In the situation above, we assume that
$\gr_G^q\gr_F^p\cM$
is a locally free $\cO_X$-module of finite rank
for all $p,q$.
Then the filtrations $G$ on $\cL \otimes \gr_F^p\cM$
and on $\gr_F^p(\cL \otimes \cM)$ coincide
under the canonical isomorphism
\eqref{canonical surjection for tensor product and gr:eq}.
In particular,
the canonical morphism
\eqref{canonical surjection for tensor product and gr:eq}
induces the isomorphism
\begin{equation}
\gr_G^q(\cL \otimes \gr_F^p\cM)
\longrightarrow
\gr_G^q\gr_F^p(\cL \otimes \cM)
\end{equation}
for all $p,q$.
\end{lem}
\begin{proof}
Because the question is of local nature,
we may assume that there exists a splitting
\begin{equation}
\gr_G^q\gr_F^p\cM
\longrightarrow
F^p\cM \cap G^q\cM
\end{equation}
of the canonical surjection
\begin{equation}
F^p\cM \cap G^q\cM
\longrightarrow
\gr_G^q\gr_F^p\cM
\end{equation}
for all $p,q$.
The image of this splitting is denoted by $\cM^{p,q}$ for a while.
By using the fact that there exists the canonical isomorphism
$\gr_F^p\gr_G^q\cM \simeq \gr_G^q\gr_F^p\cM$ for all $p,q$,
we obtain the direct sum decomposition
\begin{equation}
\label{eq:23}
\cM=
\bigoplus_{p,q}\cM^{p,q}
\end{equation}
such that
\begin{equation}
\label{eq:24}
F^p\cM=\bigoplus_{p' \ge p}\cM^{p',q}, \quad
G^q\cM=\bigoplus_{q' \ge q}\cM^{p,q'}
\end{equation}
for all $p,q$.
Then the conclusion is clear.
\end{proof}

\begin{rmk}
\label{rmk:1}
For the case of $\cO_X=\bQ$,
the conclusion of the lemma above
holds true without
any assumption.
Namely, we have the following:
Let $\cL, \cM$ be $\bQ$-sheaves on a topological space $X$,
and $F, G$ finite decreasing filtrations on $\cM$.
Then the the filtrations $G$ on $\cL \otimes_{\bQ} \gr_F^p\cM$
and on $\gr_F^p(\cL \otimes_{\bQ} \cM)$ are identified
under the canonical isomorphism
\eqref{canonical surjection for tensor product and gr:eq}.
(cf. \cite[(1.4.3)]{DeligneII}.)
%The morphism $\cM \oplus \cM \longrightarrow \cM$
%sending $(x,y)$ to $x-y$ induces the commutative diagram
%\begin{equation}
%\begin{CD}
%0 @>>> F^p\cM \cap G^q\cM
%  @>>> F^p\cM \oplus G^q\cM
%  @>>> \cM \\
%@. @. @VVV @| \\
%@. @.  \cM \oplus \cM
%  @>>> \cM
%\end{CD}
%\end{equation}
%in which the top row is exact.
%Tensoring $\cL$ over $\bQ$ to this commutative diagram,
%we obtain a commutative diagram
%\begin{equation}
%\begin{CD}
%0 @>>> \cL \otimes_{\bQ} (F^p\cM \cap G^q\cM)
%  @>>> (\cL \otimes_{\bQ} F^p\cM) \oplus (\cL \otimes_{\bQ} G^q\cM)
%  @>>> \cL \otimes_{\bQ} \cM \\
%@. @. @VVV @| \\
%0 @>>> F^p(\cL \otimes_{\bQ} \cM) \cap G^q(\cL \otimes_{\bQ} \cM)
%  @>>> F^p(\cL \otimes_{\bQ} \cM) \oplus G^q(\cL \otimes_{\bQ} \cM)
%  @>>> \cL \otimes_{\bQ} \cM \\
%@. @. @VVV @. \\
%@. @. 0
%\end{CD}
%\end{equation}
%with exact rows.
%Thus we have a surjection
%\begin{equation}
%\cL \otimes_{\bQ} (F^p\cM \cap G^q\cM)
%\longrightarrow
%F^p(\cL \otimes_{\bQ} \cM) \cap G^q(\cL \otimes_{\bQ} \cM)
%\end{equation}
%for all $p,q$.
%Therefore we have the commutative diagram
%\begin{equation}
%\begin{CD}
%\cL \otimes (F^{p-1}\cM \cap G^q\cM)
%@>>>
%F^{p-1}(\cL \otimes_{\bQ} \cM) \cap G^q(\cL \otimes_{\bQ} \cM)
%@>>> 0 \\
%@VVV @VVV \\
%\cL \otimes_{\bQ} (F^p\cM \cap G^q\cM)
%@>>>
%F^p(\cL \otimes_{\bQ} \cM) \cap G^q(\cL \otimes_{\bQ} \cM)
%@>>> 0 \\
%@VVV @VVV \\
%\cL \otimes_{\bQ} (G^q\gr_F^p\cM)
%@.
%G^q\gr_F^p(\cL \otimes_{\bQ} \cM) \\
%@VVV @VVV \\
%0 @. 0
%\end{CD}
%\end{equation}
%with exact rows and columns.
%Thus we obtain the conclusion.
\end{rmk}

\begin{defn}
Let $\cL, \cM$ be $\cO_X$-modules
and $F_1, F_2, \dots, F_l$ finite decreasing filtrations on $\cM$.
For any $i$ with $1 \le i \le l-1$,
the filtration $F_j$ for $i+1 \le j \le l$ on
\begin{equation}
\gr_{F_i}^{p_i}\gr_{F_{i-1}}^{p_{i-1}}
\cdots \gr_{F_2}^{p_2}\gr_{F_1}^{p_1}\cM
\end{equation}
is defined inductively on $i$.
This filtration $F_j$
induces the filtration $F_j$ on
\begin{equation}
\cL \otimes
\gr_{F_i}^{p_i}\gr_{F_{i-1}}^{p_{i-1}}
\cdots \gr_{F_2}^{p_2}\gr_{F_1}^{p_1}\cM
\end{equation}
as in Definition \ref{defn:1}.
We have the canonical morphism
\eqref{canonical surjection for tensor product and gr:eq}
\begin{equation}
\label{eq:1}
\cL \otimes
\gr_{F_i}^{p_i}\gr_{F_{i-1}}^{p_{i-1}}
\cdots \gr_{F_2}^{p_2}\gr_{F_1}^{p_1}\cM
\longrightarrow
\gr_{F_i}^{p_i}
(\cL \otimes
\gr_{F_{i-1}}^{p_{i-1}}
\cdots \gr_{F_2}^{p_2}\gr_{F_1}^{p_1}\cM)
\end{equation}
for $F_i$
on $\gr_{F_{i-1}}^{p_{i-1}}
\cdots \gr_{F_2}^{p_2}\gr_{F_1}^{p_1}\cM$,
which preserves the filtration $F_{i+1}, F_{i+2}, \dots, F_l$
on the both sides.
Thus the canonical morphism
\begin{equation}
\label{the canonical morphism for tensor product and gr's:eq}
\cL \otimes
\gr_{F_i}^{p_i}\gr_{F_{i-1}}^{p_{i-1}}
\cdots \gr_{F_2}^{p_2} \gr_{F_1}^{p_1}\cM
\longrightarrow
\gr_{F_i}^{p_i}\gr_{F_{i-1}}^{p_{i-1}}
\cdots \gr_{F_2}^{p_2} \gr_{F_1}^{p_1}
(\cL \otimes \cM)
\end{equation}
is obtained for any $i$ with $1 \le i \le l$
as the composite
\begin{equation}
\begin{split}
\cL \otimes
\gr_{F_i}^{p_i}\gr_{F_{i-1}}^{p_{i-1}}\gr_{F_{i-2}}^{p_{i-2}}
\cdots \gr_{F_2}^{p_2}\gr_{F_1}^{p_1}\cM
&\longrightarrow
\gr_{F_i}^{p_i}(\cL \otimes \gr_{F_{i-1}}^{p_{i-1}} \gr_{F_{i-2}}^{p_{i-2}}
\cdots \gr_{F_2}^{p_2}\gr_{F_1}^{p_1}\cM) \\
&\longrightarrow
\gr_{F_i}^{p_i}\gr_{F_{i-1}}^{p_{i-1}}(\cL \otimes \gr_{F_{i-2}}^{p_{i-2}}
\cdots \gr_{F_2}^{p_2}\gr_{F_1}^{p_1}\cM) \\
&\cdots \\
&\longrightarrow
\gr_{F_i}^{p_i}\gr_{F_{i-1}}^{p_{i-1}}
\cdots \gr_{F_2}^{p_2}
(\cL \otimes \gr_{F_1}^{p_1}\cM) \\
&\longrightarrow
\gr_{F_i}^{p_i}\gr_{F_{i-1}}^{p_{i-1}}
\cdots \gr_{F_2}^{p_2}\gr_{F_1}^{p_1}(\cL \otimes \cM),
\end{split}
\end{equation}
where each step is the morphism
induced by the canonical morphism \eqref{eq:1}.
The canonical morphism
\eqref{the canonical morphism for tensor product and gr's:eq}
preserves the filtration
$F_{i+1}, F_{i+2}, \dots, F_l$
on the both sides.
\end{defn}

\begin{lem}
\label{lem:1}
Let $\cM$ be an $\cO_X$-module
and $F_1, F_2, \dots, F_l$ finite decreasing filtrations on $\cM$.
If the $\cO_X$-module
\begin{equation}
\gr_{F_l}^{p_l}\gr_{F_{l-1}}^{p_{l-1}}
\cdots \gr_{F_2}^{p_2}\gr_{F_1}^{p_1}\cM
\end{equation}
is locally free of finite rank
for all $p_1, p_2, \dots, p_l$,
then
\begin{equation}
\gr_{F_j}^p\gr_{F_i}^{p_i} \gr_{F_{i-1}}^{p_{i-1}}
\cdots \gr_{F_2}^{p_2}\gr_{F_1}^{p_1}\cM
\end{equation}
is locally free of finite rank for $1 \le i < j \le l$
and for all $p, p_1, p_2, \dots, p_i$.
\end{lem}
\begin{proof}
For the case of $i=l-1$, we have nothing to prove.
For $i < j$,
if we assume that the $\cO_X$-module
\begin{equation}
\gr_{F_j}^p\gr_{F_i}^{p_i} \gr_{F_{i-1}}^{p_{i-1}}
\cdots \gr_{F_2}^{p_2}\gr_{F_1}^{p_1}\cM
\simeq
\gr_{F_i}^{p_i}\gr_{F_j}^p\gr_{F_{i-1}}^{p_{i-1}}
\cdots \gr_{F_2}^{p_2}\gr_{F_1}^{p_1}\cM
\end{equation}
is locally free of finite rank
for all $p_i$,
then the $\cO_X$-module
\begin{equation}
\gr_{F_j}^p\gr_{F_{i-1}}^{p_{i-1}}
\cdots \gr_{F_2}^{p_2}\gr_{F_1}^{p_1}\cM
\end{equation}
is locally free of finite rank.
Thus we obtain the conclusion by the descending induction on $i$.
\end{proof}

\begin{lem}
\label{local freeness and the commutattivity of gr}
Let $\cL$ and $\cM$ be $\cO_X$-modules
and $F_1, F_2, \dots, F_l$ finite decreasing filtrations on $\cM$.
If the $\cO_X$-module
$\gr_{F_l}^{p_l}\gr_{F_{l-1}}^{p_{l-1}}
\cdots \gr_{F_2}^{p_2}\gr_{F_1}^{p_1}\cM$
is locally free of finite rank
for all $p_1,p_2, \dots, p_l$,
then the canonical morphism
\eqref{the canonical morphism for tensor product and gr's:eq}
is an isomorphism
for all $p_1, p_2, \dots, p_i$,
under which the filtrations $F_{i+1}, F_{i+2}, \dots F_l$
are identified on the both sides.
\end{lem}
\begin{proof}
We proceed by induction on $i$.
By Lemma \ref{lem:2} and Lemma \ref{lem:1},
the canonical morphism
\begin{equation}
\cL \otimes \gr_{F_1}^{p_1}\cM
\longrightarrow
\gr_{F_1}^{p_1}(\cL \otimes \cM)
\end{equation}
is an isomorphism,
under which the filtrations $F_2, F_3, \dots, F_l$
are identified on the both sides.
Thus the case of $i=1$ is proved.
By Lemma \ref{lem:2} and Lemma \ref{lem:1} again,
the canonical morphism
\begin{equation}
\label{eq:2}
\cL \otimes
\gr_{F_i}^{p_i}\gr_{F_{i-1}}^{p_{i-1}}
\cdots \gr_{F_2}^{p_2}\gr_{F_1}^{p_1}\cM
\longrightarrow
\gr_{F_i}^{p_i}
(\cL \otimes
\gr_{F_{i-1}}^{p_{i-1}}
\cdots \gr_{F_2}^{p_2}\gr_{F_1}^{p_1}\cM)
\end{equation}
is an isomorphism,
under which the filtrations $F_{i+1}, F_{i+2}, \dots, F_l$
are identified on the both sides.
By the induction hypothesis,
the canonical morphism
\begin{equation}
\label{eq:4}
\cL \otimes
\gr_{F_{i-1}}^{p_{i-1}}
\cdots \gr_{F_2}^{p_2}\gr_{F_1}^{p_1}\cM
\longrightarrow
\gr_{F_{i-1}}^{p_{i-1}}
\cdots \gr_{F_2}^{p_2}\gr_{F_1}^{p_1}(\cL \otimes \cM)
\end{equation}
is an isomorphism
under which the filtrations
$F_i, F_{i+1}, F_{i+2}, \dots, F_l$
are identified on the both sides.
Therefore the isomorphism \eqref{eq:4}
induces the isomorphism
\begin{equation}
\label{eq:3}
\gr_{F_i}^{p_i}
(\cL \otimes
\gr_{F_{i-1}}^{p_{i-1}}
\cdots \gr_{F_2}^{p_2}\gr_{F_1}^{p_1}\cM)
\longrightarrow
\gr_{F_i}^{p_i}\gr_{F_{i-1}}^{p_{i-1}}
\cdots \gr_{F_2}^{p_2}\gr_{F_1}^{p_1}(\cL \otimes \cM)
\end{equation}
under which the filtrations $F_{i+1}, F_{i+2}, \dots, F_l$
are identified on the both sides.
Composing the isomorphisms
\eqref{eq:2} and \eqref{eq:3},
we obtain the conclusion.
\end{proof}

\begin{rmk}
\label{rmk:2}
For the case of $\cO_X=\bQ$,
the conclusion of the lemma above
also holds true
without any assumption.
Namely, we have the following:
Let $\cL, \cM$ be $\bQ$-sheaves on $X$
and $F_1, F_2, \dots, F_l$ finite decreasing filtrations on $\cM$.
Then the canonical morphism
\eqref{the canonical morphism for tensor product and gr's:eq}
is an isomorphism for all $p_1, p_2, \dots, p_i$,
under which the filtrations $F_{i+1}, F_{i+2}, \dots, F_l$
are identified on the both sides.
\end{rmk}

\begin{defn}
Let $K$ be a complex of $\cO_X$-modules on $X$
and $F$ a finite decreasing filtration on $K$.
We say that $(K,F)$ (or $F$ on $K$) is strict
if the differential
$d:K^p \longrightarrow K^{p+1}$
is strictly compatible with $F$ for all $p$.
\end{defn}

\begin{lem}
\label{lemma on successively strict filtrations}
Let $K$ be a complex of $\cO_X$-modules
and $F_1, F_2, \dots, F_l$ finite decreasing filtrations on $K$.
We assume that
$(\gr_{F_i}^{p_i}\gr_{F_{i-1}}^{p_{i-1}} \cdots \gr_{F_1}^{p_1}K, F_j)$
is strict for $0 \le i < j \le l$.
$($For $i=0$, we mean that $(K,F_j)$ is strict for all $j$.$)$
Then there exists an isomorphism
\begin{equation}
\coh^a(\gr_{F_i}^{p_i}\gr_{F_{i-1}}^{p_{i-1}} \cdots \gr_{F_1}^{p_1}K)
\simeq
\gr_{F_i}^{p_i}\gr_{F_{i-1}}^{p_{i-1}} \cdots \gr_{F_1}^{p_1}\coh^a(K)
\end{equation}
for any $i$ with $1 \le i \le l$
and for all $a,p_1,p_2, \dots, p_i$,
under which the filtration $F_j$
is identified on the both sides
for any $j$ with $i< j \le l$.
\end{lem}
\begin{proof}
From the assumption that $(K,F_1)$ is strict,
we have the isomorphism
\begin{equation}
\label{eq:8}
\coh^a(\gr_{F_1}^{p_1}K)
\simeq
\gr_{F_1}^{p_1}\coh^a(K)
\end{equation}
for all $a,p_1$,
under which the filtration $F_j$
is identified on the both sides
by the lemma on two filtrations
\cite[Proposition (7.2.5)]{DeligneIII}
and by the strictness of $F_j$ on $\gr_{F_1}^{p_1}K$
for $j=2,3, \dots, l$.
Thus we have the isomorphism
\begin{equation}
\coh^a(\gr_{F_2}^{p_2}\gr_{F_1}^{p_1}K)
\simeq
\gr_{F_2}^{p_2}\coh^a(\gr_{F_1}^{p_1}K)
%\simeq
%\gr_{F_2}^{p_2}\gr_{F_1}^{p_1}\coh^a(K)
\end{equation}
for all $p_2$
by the strictness of $F_2$ on $\gr_{F_1}^{p_1}K$.
Lemma on two filtrations implies that
the filtration $F_j$ on the left hand side
is identified with $F_j$ on the right hand side
by the strictness of $F_j$ on $\gr_{F_2}^{p_2}\gr_{F_1}^{p_1}K$
for $j=3,4, \dots, l$.
Combining with the isomorphism
\eqref{eq:8},
we obtain the conclusion for $i=2$.
By applying this procedure repeatedly,
we obtain the conclusion for $i=1,2, \dots, l$.
\end{proof}

\begin{defn}
Let $K$ be a complex of $\cO_X$-modules on $X$
and $F$ a finite decreasing filtration on $K$.
We say that $(K,F)$ (or $F$ on $K$)
is strongly strict
if $(K,F)$ is strict
and if $\coh^a(\gr_F^pK)$ is a locally free $\cO_X$-module
of finite rank for all $a,p$.
\end{defn}

\begin{rmk}
If $(K,F)$ is strongly strict,
then we have the following:
\begin{mylist}
\itemno
The sequence
\begin{equation*}
\begin{CD}
0 @>>> \coh^a(F^{p-1}K)
  @>>> \coh^a(F^pK)
  @>>> \coh^a(\gr_F^pK) @>>> 0
\end{CD}
\end{equation*}
is exact for all $a,p$.
In particular,
the canonical morphism
\begin{equation*}
\coh^a(F^pK) \longrightarrow F^p\coh^a(K)
\end{equation*}
is an isomorphism for all $a,p$.
\itemno
The $\cO_X$-module $\coh^a(F^pK)$
is locally free of finite rank for all $a,p$.
\end{mylist}
\end{rmk}

\begin{rmk}
The notion of strong strictness is well-defined
in the filtered derived category
in the following sense.
If $(K_1,F)$ and $(K_2,F)$
are isomorphic in the filtered derived category,
then $(K_1,F)$ is strongly strict
if and only if $(K_2,F)$ is strongly strict.
\end{rmk}

\begin{lem}
\label{lemma on strong strictness and filtered tensor product}
Let $(K,F)$ be a filtered perfect complex of $\cO_X$-modules
and $\cF$ an $\cO_X$-module.
The filtered derived tensor product is denoted by $(\cF \otimes^L K,F)$,
where $\cF$ is considered as a filtered object by the trivial filtration
$F^0\cF=\cF$, $F^1\cF=0$.
If $(K,F)$ is strongly strict,
then $(\cF \otimes^L K,F)$ is strict
and the canonical morphisms
\begin{align}
&\cF \otimes \coh^a(F^pK)
\longrightarrow
\coh^a(\cF \otimes^L F^pK)
\label{canonical morphism for F:eq} \\
&\cF \otimes \coh^a(\gr_F^pK)
\longrightarrow
\coh^a(\cF \otimes^L \gr_F^pK)
\label{canonical morphism for grF:eq}
\end{align}
are isomorphisms for all $a,p$.
In particular,
$(\cA \otimes^L K, F)$ is strongly strict
as a filtered complex of $\cA$-modules
for a commutative $\cO_X$-algebra $\cA$.
\end{lem}
\begin{proof}
The canonical morphism
\eqref{canonical morphism for grF:eq} is an isomorphism
because $\coh^a(\gr_F^pK)$ is locally free for all $a,p$.
Then we have the commutative diagram with exact rows
\begin{equation}
\begin{CD}
0 @>>> \cF \otimes \coh^a(F^{p+1}K)
  @>>> \cF \otimes \coh^a(F^pK)
  @>>> \cF \otimes \coh^a(\gr_F^pK)
  @>>> 0 \\
@. @VVV @VVV @VV{\simeq}V @. \\
@. \coh^a(\cF \otimes^L F^{p+1}K)
  @>>> \coh^a(\cF \otimes^L F^pK)
  @>>> \coh^a(\cF \otimes^L \gr_F^pK)
\end{CD}
\end{equation}
because $(K, F)$ is strongly strict.
Then the easy diagram chasing shows
that the second row in the diagram above
fits in the exact sequence
\begin{equation}
\begin{CD}
0 @>>> \coh^a(\cF \otimes^L F^{p-1}K)
  @>>> \coh^a(\cF \otimes^L F^pK)
  @>>> \coh^a(\cF \otimes^L \gr_F^pK)
  @>>> 0
\end{CD}
\end{equation}
and the canonical morphism
\eqref{canonical morphism for F:eq} is an isomorphism for all $a,p$
by descending induction on $p$.
The injectivity of the morphism
\begin{equation}
\coh^a(\cF \otimes^L F^{p+1}K)
\longrightarrow
\coh^a(\cF \otimes^L F^pK)
\end{equation}
for all $p$
implies that $(\cF \otimes^L K, F)$ is strict.
\end{proof}

For the case where $X$ is a complex manifold,
we have the following results.

\begin{lem}
\label{strong strictness criteria by pointwise strictness}
Let $X$ be a complex manifold
and $(K,F)$ a filtered perfect complex
such that $\coh^a(K)$ is locally free of finite rank for all $a$.
Then $(K,F)$ is strongly strict
if and only if $(\bC(x) \otimes^L K,F)$ is strict for every point $x \in X$.
\end{lem}

This lemma is a direct consequence
of Lemma 3.4 (iv) in \cite{Fujino-Fujisawa}.
However, the proof of Lemma 3.4 (iv) in \cite{Fujino-Fujisawa}
is somewhat misleading.
So, we restate the result and complete the proof.

\begin{lem}[Lemma 3.4 (iv) of \cite{Fujino-Fujisawa}]
Let $(K, F)$ be a filtered perfect complex on a complex manifold $X$.
Assume that the function
$X \ni x \mapsto \dim \coh^q(\bC(x) \otimes^L K)$
is locally constant.
If the morphisms
\begin{equation}
\label{differentials on C(x) otimes K:eq}
\begin{split}
&d(x):
(\bC(x) \otimes^L K)^{q-1}
\longrightarrow
(\bC(x) \otimes^L K)^q, \\
&d(x):
(\bC(x) \otimes^L K)^q
\longrightarrow
(\bC(x) \otimes^L K)^{q+1}
\end{split}
\end{equation}
are strictly compatible with the filtration
$F(\bC(x) \otimes^L K)$ for all $x \in X$,
then $\coh^q(\gr_F^pK)$ and $\coh^q(F^pK)$ are locally free of finite rank,
the canonical morphisms
\begin{equation}
\label{canonical morphisms for grF and F:eq}
\begin{split}
&\bC(x) \otimes \coh^q(\gr_F^pK)
\longrightarrow
\coh^q(\bC(x) \otimes^L \gr_F^pK) \\
&\bC(x) \otimes \coh^q(F^pK)
\longrightarrow
\coh^q(\bC(x) \otimes^L F^pK)
\end{split}
\end{equation}
are isomorphisms for all $p$ and for all $x \in X$,
and the differentials
\begin{equation}
\label{differentials on K:eq}
\begin{split}
&d: K^{q-1} \longrightarrow K^q \\
&d: K^q \longrightarrow K^{q+1}
\end{split}
\end{equation}
are strictly compatible with the filtration $F$.
\end{lem}
\begin{proof}
Note that we have the canonical isomorphisms
\begin{equation}
\begin{split}
&\bC(x) \otimes^L F^pK
\simeq
F^p(\bC(x) \otimes^L K) \\
&\bC(x) \otimes^L \gr_F^pK
\simeq
\gr_F^p(\bC(x) \otimes^L K)
\end{split}
\end{equation}
by the definition of filtered derived tensor product.
We obtain the exact sequences
\begin{equation}
\label{exact sequence for K otimes C(x):eq}
\begin{CD}
0 @>>> \coh^q(\bC(x) \otimes^L F^{p+1}K)
  @>>> \coh^q(\bC(x) \otimes^L F^pK) \\
@. @>>> \coh^q(\bC(x) \otimes^L \gr_F^pK)
   @>>> 0
\end{CD}
\end{equation}
from the assumption that
the differentials \eqref{differentials on C(x) otimes K:eq}
are strictly compatible with the filtration $F$.
Thus we have the equality
\begin{equation}
\label{equality for the dimension of the cohomology:eq}
\dim \coh^q(\bC(x) \otimes^L F^pK)
=
\sum_{p' \ge p}\dim \coh^q(\bC(x) \otimes^L \gr_F^{p'}K)
\end{equation}
for all $p$ and for all $x \in X$.
Taking $p$ sufficiently small,
we obtain
\begin{equation}
\dim \coh^q(\bC(x) \otimes^L K)
=
\sum_p\dim \coh^q(\bC(x) \otimes^L \gr_F^pK) ,
\end{equation}
which implies that the function
$X \ni x \mapsto \dim \coh^q(\bC(x) \otimes^L \gr_F^pK)$
is locally constant
as in the proof of Lemma 3.4 (iv) in \cite{Fujino-Fujisawa}.
Then the function
$X \ni x \mapsto \dim \coh^q(\bC(x) \otimes^L F^pK)$
is locally constant from the equality
\eqref{equality for the dimension of the cohomology:eq}.
Therefore $\coh^q(F^pK)$ and $\coh^q(\gr_F^pK)$ are locally free
and the canonical morphisms
\eqref{canonical morphisms for grF and F:eq}
are isomorphisms for all $p$
by Lemma 3.4 (iii) in \cite{Fujino-Fujisawa}.
From the exact sequence
\eqref{exact sequence for K otimes C(x):eq},
we have the exact sequence
\begin{equation}
\begin{CD}
0 @>>> \bC(x) \otimes \coh^q(F^{p+1}K)
  @>>> \bC(x) \otimes \coh^q(F^pK) \\
@. @>>> \bC(x) \otimes \coh^q(\gr_F^pK)
   @>>> 0 ,
\end{CD}
\end{equation}
from which we obtain the exact sequence
\begin{equation}
\begin{CD}
0 @>>> \coh^q(F^{p+1}K)
  @>>> \coh^q(F^pK)
  @>>> \coh^q(\gr_F^pK)
  @>>> 0
\end{CD}
\end{equation}
by using the local freeness of $\coh^q(F^{p+1}K)$ for all $p$.
Therefore the differentials
\eqref{differentials on K:eq}
are strictly compatible with $F$ as desired.
\end{proof}

\section{Semistable morphism}
\label{semistable morpshims}

\begin{defn}
\label{definition for semistable morphism}
Let $Y$ be a complex manifold and $E$ a reduced \snc divisor on $Y$.
A morphism of complex manifolds
$f: X \longrightarrow Y$
is said to be semistable along $E$
if the following two conditions are satisfied:
\begin{mylist}
\itemno
\label{condition for D being reduced}
The divisor $D=f^{\ast}E$
is a {\slshape reduced} \snc divisor on $X$.
\itemno
\label{second condition for a semistable morphism}
The morphism
\begin{equation*}
f^{\ast}\Omega_Y^1(\log E) \longrightarrow \Omega_X^1(\log D)
\end{equation*}
is injective
and the cokernel
\begin{equation*}
\Omega_{X/Y}^1(\log D)
=\Omega_X^1(\log D)\Bigl/f^{\ast}\Omega_Y^1(\log E)
\end{equation*}
is locally free.
\end{mylist}
Once we fix the divisor $E$,
the words ``along $E$'' is omitted.
Moreover, we use the notation
\begin{equation*}
\omega^p_Y=\Omega^p_Y(\log E), \quad
\omega^p_X=\Omega^p_X(\log D), \quad
\omega^p_{X/Y}=\Omega^p_{X/Y}(\log D)
\end{equation*}
if there is no danger of confusion.
We set
$Y^{\ast}=Y \setminus E$
and $X^{\ast}=X \setminus D=f^{-1}(Y^{\ast})$.
\end{defn}

\begin{exa}
Let $\Delta$ be the unit disc in $\bC$
with the coordinate function $t$,
and $\Delta^n$ the polydisc in $\bC^n$ with the coordinate functions
$x_1, x_2, \dots, x_n$.
A morphism $f: \Delta^n \longrightarrow \Delta$ given by
\begin{equation}
\label{semistable over Delta}
f^{\ast}t=x_1x_2 \cdots x_n
\end{equation}
is semistable along $\{0\}$.

Let $f_i:\Delta^{n_i} \longrightarrow \Delta$ ($i=1,2,\dots, k$)
be morphisms given by the form \eqref{semistable over Delta}
and $g: T \longrightarrow S$ a smooth morphism of complex manifolds.
Set $Y=\Delta^k \times S$ and $E=\{t_1t_2 \cdots t_k=0\}$
where $t_1, t_2, \dots, t_k$ is the coordinate functions of $\Delta^k$.
Then the morphism
\begin{equation}
\label{local model of semistable morphism:eq}
f_1 \times f_2 \times \dots \times f_k \times g:
\Delta^{n_1} \times \Delta^{n_2} \times \dots \times \Delta^{n_k} \times T
\longrightarrow
\Delta^k \times S
\end{equation}
is also semistable along $E$.
\end{exa}

\begin{lem}
\label{local description of semistable morphism}
Locally on $X$ and $Y$,
any semistable morphism $f: X \longrightarrow Y$
is isomorphic to the morphism of the form
\eqref{local model of semistable morphism:eq}.
In particular, any semistable morphism is equi-dimensional and flat.
\end{lem}
\begin{proof}
Similar to the proof of Lemma (6.5) in \cite{FujisawaLHSSV}.
\end{proof}

\begin{defn}
\label{definition of the filtration G}
Let $f:X \longrightarrow Y$ be a semistable morphism along $E$.
A finite decreasing filtration $G$ on $\omega_X$
is defined by
\begin{equation*}
G^p\omega_X^n
=\image(f^{-1}\omega^p_Y \otimes \omega^{n-p}_X
\overset{\wedge}{\longrightarrow} \omega^n_X)
\end{equation*}
for all $n,p$ as in \cite{Katz-Oda}.
Then we have the canonical isomorphism
\begin{equation}
\label{isomorphism for grGOmega}
f^{-1}\omega^p_Y \otimes \omega_{X/Y}[-p]
\overset{\simeq}{\longrightarrow}
\gr_G^p\omega_X
\end{equation}
for every $p$.
Therefore the spectral sequence
$E_r^{a,b}(Rf_{\ast}\omega_X,G)$
satisfies the property
\begin{equation*}
\begin{split}
E_1^{a,b}(Rf_{\ast}\omega_X,G)
&\simeq
R^bf_{\ast}(f^{-1}\omega^a_Y \otimes \omega_{X/Y}) \\
&\simeq
\omega^a_Y \otimes R^bf_{\ast}\omega_{X/Y}
\end{split}
\end{equation*}
for all $a,b$
because $\omega^a_Y$ is a locally free $\cO_Y$-module of finite rank.
Then the morphism of $E_1$-terms
induces the morphism
\begin{equation}
\label{the Gauss-Manin connection:eq}
R^nf_{\ast}\omega_{X/Y}
\longrightarrow
\omega^1_Y
\otimes
R^nf_{\ast}\omega_{X/Y}
\end{equation}
which is called the Gauss-Manin connection
and denoted by $\nabla$.
The residue of
the Gauss-Manin connection \eqref{the Gauss-Manin connection:eq}
along the divisor $E_i$
is denoted by
\begin{equation}
\label{the residue of nabla:eq}
\res_{E_i}(\nabla):
\cO_{E_i}
\otimes
R^nf_{\ast}\omega_{X/Y}
\longrightarrow
\cO_{E_i}
\otimes
R^nf_{\ast}\omega_{X/Y}
\end{equation}
for every $i=1,2,\dots,k$.
\end{defn}

\begin{rmk}
As in \cite{Katz-Oda},
$\nabla$ is an integrable logarithmic connection,
that is,
$\nabla$ satisfies the Leibniz rule
and $\nabla^2=0$.
Moreover, it satisfies the Griffiths transversality
\begin{equation}
\label{Griffiths transversality:eq}
\nabla(F^pR^nf_{\ast}\omega_{X/Y})
\subset
\omega^1_Y
\otimes
F^{p-1}R^nf_{\ast}\omega_{X/Y}
\end{equation}
where $F$ denotes the filtration
induced by the stupid filtration $F$ on $\omega_{X/Y}$.

On the other hand,
the restriction of $\nabla$ on $Y^{\ast}=Y \setminus E$
is identified with
\begin{equation*}
d \otimes \id:
\cO_{Y^{\ast}} \otimes_{\bC} R^nf_{\ast}\bC_{X^{\ast}}
\longrightarrow
\Omega^1_{Y^{\ast}} \otimes_{\bC} R^nf_{\ast}\bC_{X^{\ast}}
\end{equation*}
via the isomorphisms
\begin{equation*}
R^nf_{\ast}\omega_{X/Y}|_{Y^{\ast}}
=
R^nf_{\ast}\Omega_{X^{\ast}/Y^{\ast}}
\simeq
R^nf_{\ast}f^{-1}\cO_{Y^{\ast}}
\simeq
\cO_{Y^{\ast}} \otimes_{\bC} R^bf_{\ast}\bC_{X^{\ast}}
\end{equation*}
as in \cite{Katz-Oda}, \cite[4.3]{Fujino-Fujisawa}.
\end{rmk}

\begin{para}
\label{notation for Y and E}
Let $Y$ be a complex manifold,
$E$ a reduced \snc divisor on $Y$
and
\begin{equation}
E=\sum_{i=1}^{k}E_i
\end{equation}
the irreducible decomposition of $E$.
We set
\begin{equation}
E_I=\sum_{i \in I}E_i, \quad
Y[I]=\bigcap_{i \in I}E_i
\end{equation}
for $I \subset \ski$.
(We set $E_{\emptyset}=0, Y[\emptyset]=Y$.)
We write $Y_0=Y[\ski]$ for short.
Since $E$ is a reduced \snc divisor,
$Y[I]$ is a closed submanifold of $Y$ for any $I$.
By setting $\overline{I}=\ski \setminus I$,
we have the \snc divisor $E_{\overline{I}}$ on $Y$,
whose restriction to $Y[I]$, denoted by $E_{\overline{I}} \cap Y[I]$,
is a reduced \snc divisor on $Y[I]$.
We use the notation
\begin{equation}
\omega_{Y[I]}^p=\Omega_{Y[I]}^p(\log E_{\overline{I}} \cap Y[I])
\end{equation}
for all $p$.
We set
$Y[I]^{\ast}=Y[I] \setminus (E_{\overline{I}} \cap Y[I])$
for $I \subset \ski$.

We have a sequence of closed submanifolds
\begin{equation*}
Y_0=Y[\ski] \subset Y[I] \subset Y[J] \subset Y[\emptyset]=Y
\end{equation*}
for $J \subset I \subset \ski$.
The closed immersion $Y[I] \hookrightarrow Y[J]$
is denoted by $\iota_{J|I}$.
For the case of $J=\emptyset$,
the closed immersion
$\iota_{\emptyset|I}: Y[I] \hookrightarrow Y$
is denoted by $\iota_{|I}$ for short.
For the case of $I=\ski$,
the closed immersion $\iota_{J|\ski}: Y_0=Y[\ski] \hookrightarrow Y[J]$
is simply denoted by $\iota_{J|}$.
We use the notation $\iota$
instead of $\iota_{\emptyset|\ski}: Y_0 \hookrightarrow Y$ for short.
\end{para}

\begin{para}
\label{setting for a semistable morphism}
Let $f:X \longrightarrow Y$ be a semistable morphism along $E$.
We set $D=f^{\ast}E$ and $D_i=f^{\ast}E_i$ for $i=\ki$.
Then $D$ and $D_i$'s are {\itshape reduced} \snc divisors on $X$
with $D=\sum_{i=1}^{k}D_i$.
We set
\begin{equation*}
D_I=\sum_{i \in I}D_i, \quad
X[I]=\bigcap_{i \in I}D_i
\end{equation*}
for $I \subset \ski$
($D_{\emptyset}=0, X[\emptyset]=X$ as above).
We write $X_0=X[\ski]$ as before.
For $J \subset I \subset \ski$,
we have a sequence of closed subspaces
\begin{equation}
X_0=X[\ski] \subset X[I] \subset X[J] \subset X=X[\emptyset]
\end{equation}
as before.
The closed immersion
$X[I] \hookrightarrow X[I]$
is also denoted by $\iota_{J|I}$
if there is no danger of confusion.
We also use the notation
$\iota_{|I}: X[I] \hookrightarrow X$,
$\iota_{J|}: X_0 \hookrightarrow X[J]$
and $\iota: X_0 \hookrightarrow X$ as above.
We have the cartesian square
\begin{equation}
\label{cartesian square for X[I]:eq}
\begin{CD}
X[I] @>{\iota_{|I}}>> X \\
@V{f_I}VV @VV{f}V \\
Y[I] @>>{\iota_{|I}}> Y
\end{CD}
\end{equation}
in which the left vertical arrow is denoted by $f_I$.
We use the notation $f_{\ski}=f_0$,
which fits in the cartesian square
\begin{equation}
\begin{CD}
X_0 @>{\iota}>> X \\
@V{f_0}VV @VV{f}V \\
Y_0 @>>{\iota}> Y
\end{CD}
\end{equation}
by definition.
We set
\begin{equation*}
X[I]^{\ast}
=X[I] \setminus (D_{\overline{I}} \cap X[I])
=f_I^{-1}(Y[I]^{\ast})
\end{equation*}
for $I \subset \ski$.

Let
$D=\sum_{\lambda \in \Lambda}D_{\lambda}$
be the irreducible decomposition of $D$.
We set
\begin{equation*}
X[\Gamma]=\bigcap_{\lambda \in \Gamma}D_{\lambda}
\end{equation*}
for $\Gamma \subset \Lambda$,
which is a closed submanifold of $X$
because $D$ is a reduced \snc divisor on $X$.
The closed immersion
$X[\Gamma] \hookrightarrow X$ is denoted by $\iota_{\Gamma}$.

From the equality $D=\sum_{i=1}^{k}D_i$ and the fact that $D$ is reduced,
the index set $\Lambda$ decomposes into a disjoint union
\begin{equation*}
\Lambda = \coprod_{i=1}^k\Lambda_i
\end{equation*}
such that
\begin{equation*}
D_i=\sum_{\lambda \in \Lambda_i}D_{\lambda}
\end{equation*}
for $i=\ki$.
We set
\begin{equation*}
\Lambda_I=\coprod_{i \in I}\Lambda_i \subset \Lambda
\end{equation*}
for $I \subset \ski$.
If $\Gamma \cap \Lambda_i \not= \emptyset$ for all $i \in I$,
then $X[\Gamma] \subset X[I]$.
%In this case,
%the closed immersion $X[\Gamma] \hookrightarrow X[I]$
%is denoted by $\iota_{I|\Gamma}$.
%For the case of $I=\ski$,
%we set $\iota_{0|\Gamma}=\iota_{\ski|\Gamma}$
Thus we have a commutative diagram
\begin{equation}
\label{square for Gamma:eq}
\begin{CD}
X[\Gamma] @>{\iota_{\Gamma}}>> X \\
@V{f_{\Gamma}}VV @VV{f}V \\
Y[I] @>>{\iota_{|I}}> Y
\end{CD}
\end{equation}
in which the left vertical arrow is denoted by $f_{\Gamma}$,
once we fix $I \subset \ski$.
For $\Gamma \subset \Lambda_I$,
the restriction $D_{\overline{I}} \cap X[\Gamma]$
of the divisor $D_{\overline{I}}$ to $X[\Gamma]$
is a \snc divisor on $X[\Gamma]$.
If $\Gamma \cap \Lambda_i \not=\emptyset$ for all $i \in I$,
the morphism $f_{\Gamma}:X[\Gamma] \longrightarrow Y[I]$ above
is semistable along $E_{\overline{I}} \cap Y[I]$
by Lemma \ref{local description of semistable morphism}.

In the remainder of this article,
we assume that $\Lambda$ is a finite set
and fix a total order $<$ on $\Lambda$
with $\Lambda_1 < \Lambda_2 < \cdots < \Lambda_k$.
This assumption does not affect to our main results
because they concern the case where $f$ is proper.
\end{para}

\section{Construction of $sB_I(f)$}
\label{section for complex B}

\begin{para}
\label{local situation}
From now on,
we treat the case of $Y=\Delta^k \times S$
equipped with a reduced \snc divisor $E=\{t_1t_2 \cdots t_k=0\}$,
where $t_1, t_2, \dots, t_k$ are the coordinate functions of $\Delta^k$.
By setting $E_i=\{t_i=0\}$ for $i=\ki$,
we have
$E=\sum_{i=1}^kE_i$.
We use the notation in \ref{notation for Y and E}.
We have $S=\{0\} \times S=Y[\ski]=Y_0$ by definition.

There exists the canonical projection
\begin{equation}
\pi_I: Y \longrightarrow Y[I]
\end{equation}
for every $I \subset \ski$.
We trivially have $\pi_I\iota_{|I}=\id$.
When we consider $Y[I]$ as a complex manifold {\itshape below} $Y$
by the projection $\pi_I$ above,
we use the notation $Y_I$ instead of $Y[I]$.
The \snc divisor $E_{\overline{I}} \cap Y[I]$ on $Y[I]$
defines a \snc divisor on $Y_I$,
which is denoted by $E_{\overline{I}} \cap Y_I$.
We always consider
that $Y_I$ is equipped with the divisor $E_{\overline{I}} \cap Y_I$,
unless the otherwise is mentioned.
Thus we use the notation
\begin{equation}
\omega^p_{Y_I}=\Omega^p_{Y_I}(\log E_{\overline{I}} \cap Y_I)
\end{equation}
as in Definition \ref{definition for semistable morphism}
and \ref{notation for Y and E}.
We set
$Y_I^{\ast}=Y_I \setminus (E_{\overline{I}} \cap Y_I)$
for $I \subset \ski$.
We trivially have $Y[I]^{\ast}=Y_I^{\ast}$ as complex manifolds.

For $J \subset I \subset \ski$,
we have a sequence of the projections
\begin{equation*}
Y=Y_{\emptyset}
\longrightarrow Y_J
\longrightarrow Y_I
\longrightarrow Y_{\ski}=S
\end{equation*}
by definition.
We use $\pi:Y \longrightarrow S$ instead of $\pi_{\ski}$.

For $J \subset I \subset \ski$,
the coordinate functions $t_1, t_2, \dots, t_k$
induce the direct sum decomposition
\begin{equation}
\omega^1_{Y[J]}
\simeq
\omega^1_{Y[J]/Y_I} \oplus \iota_{|J}^{\ast}\pi_I^{\ast}\omega^1_{Y_I}
\end{equation}
which induces the projection
\begin{equation}
\omega^p_{Y[J]}
\longrightarrow
\iota_{|J}^{\ast}\pi_I^{\ast}\omega^p_{Y_I}
\end{equation}
for every $p$.
Then we have the morphisms
\begin{equation}
\iota_{J|I}^{\ast}\omega^p_{Y[J]}
\longrightarrow
\omega^p_{Y[I]}
\end{equation}
and
\begin{equation}
\pr_{J|I}:
\iota_{J|I}^{-1}\omega^p_{Y[J]}
\longrightarrow
\omega^p_{Y[I]}
\label{the projection between Omega from Y[J] to Y[I]:eq}
\end{equation}
for every $p$.
\end{para}

\begin{rmk}
For a semistable morphism $f:X \longrightarrow Y$ along $E$,
$\pi_If:X \longrightarrow Y_I$
and $\pi_If\iota_{\Gamma}:X[\Gamma] \longrightarrow Y_I$
are semistable along $E_{\overline{I}} \cap Y_I$
for $I \subset \ski$ and for $\Gamma \subset \Lambda_I$.
We can see these facts
by Lemma \ref{local description of semistable morphism}.
Moreover, we have
$\pi_If\iota_{\Gamma}=f_{\Gamma}$ via the identification
$Y_I=Y[I]$.
\end{rmk}

Now, we construct a complex $sB_I(f)$
which is a replacement for
$f_I^{-1}\cO_{Y[I]} \otimes^L \iota_{|I}^{-1}\omega_{X/Y}$
for $I \subset \ski$.

The finite increasing filtration on
$\omega_X=\Omega_X(\log D)$ by the order of poles along $D_i$
is denoted by $W(D_i)$ for $i=\ki$.
Similarly, $W(D_K)$ denotes the finite increasing filtration on $\omega_X$
by the order of poles along $D_K=\sum_{i \in K}D_i$
for $K \subset \ski$.

\begin{defn}
\label{definition of B_I(f)}
Let $Y$ and $E$ be as above
and $f:X \longrightarrow Y$ a semistable morphism along $E$.
For a subset $I \subset \ski$,
a complex $sB_I(f)$ is defined as follows:
\begin{itemize}
\item
For $p \in \bnnZ, q \in \bnnZ^I$, we set
\begin{equation*}
B_I(f)^{p,q}
=\omega_{X/Y_I}^{p+|q|+|I|}\biggl/
\sum_{i \in I}W(D_i)_{q_i}
\end{equation*}
where $W(D_i)$ denotes the filtration on $\omega_{X/Y_I}$
induced by $W(D_i)$ on $\omega_X$.
\end{itemize}
We can easily see the inclusion
\begin{equation}
\label{support of B:eq}(
\supp(B_I(f)^{p,q}) \subset X[I]
\end{equation}
for any $p,q$.
Moreover, the $\cO_X$-module $B_I(f)^{p,q}$
carries the $(\iota_{|I})_{\ast}\cO_{X[I]}$-module structure
via the surjection $\cO_X \longrightarrow (\iota_{|I})_{\ast}\cO_{X[I]}$.
\begin{itemize}
\item
We set
\begin{equation*}
d_0=(-1)^{|I|}d:B_I(f)^{p,q} \longrightarrow B_I(f)^{p+1,q}
\end{equation*}
where $d$ on the right hand side of the equality
is the morphism induced from $d$
on $\omega_{X/Y_I}$.
\end{itemize}
Here we note that
the morphism $d_0$ is not a morphism of $\cO_X$-modules
but a morphism of $(\iota_{|I})_{\ast}f_I^{-1}\cO_{Y[I]}$-modules
via the canonical morphism
$f_I^{-1}\cO_{Y[I]} \longrightarrow \cO_{X[I]}$.
\begin{itemize}
\item
For $i \in I$,
we set
\begin{equation*}
d_i=(-1)^{|I|}\dlog t_i \wedge:B_I(f)^{p,q} \longrightarrow B_I(f)^{p,q+e_i}
\end{equation*}
where $\dlog t_i \wedge$ denotes the morphism defined by
\begin{equation}
\omega
\mapsto
\dlog t_i \wedge \omega=\frac{dt_i}{t_i} \wedge \omega
\end{equation}
as in \cite{FujisawaPLMHS}.
\end{itemize}
By setting
\begin{equation*}
sB_I(f)^n
=\bigoplus_{p+|q|=n}B_I(f)^{p,q}
=\bigoplus_{q \in \bnnZ^I}
\biggl(\omega_{X/Y_I}^{n+|I|}\biggl/
\sum_{i \in I}W(D_i)_{q_i} \biggr)
\end{equation*}
with the differential
$d=d_0+\sum_{i \in I}d_i$,
we obtain the complex of $(\iota_{|I})_{\ast}f_I^{-1}\cO_{Y[I]}$-modules
$sB_I(f)$ on $X$.
We use the symbol $sB(f)$ for the case of $I=\ski$.
Because of \eqref{support of B:eq},
the complex $sB_I(f)$ is considered as
a complex of $f_I^{-1}\cO_{Y[I]}$-modules on $X[I]$
as remarked in \ref{para:2}.

We define a finite decreasing filtration $F$ on $sB_I(f)$ by
\begin{equation}
\label{definition of F on sBI(f):eq}
F^psB_I(f)^n
=\bigoplus_{\substack{p'+|q|=n \\
                      p' \ge p}}
B_I(f)^{p',q}
=\bigoplus_{\substack{q \in \bnnZ^I \\
                      |q| \le n-p}}
\biggl(\omega_{X/Y_I}^{n+|I|}\biggl/
\sum_{i \in I}W(D_i)_{q_i} \biggr)
\end{equation}
for all $n,p$.

By definition,
$sB_{\emptyset}(f)$ equipped with $F$
is nothing but $\omega_{X/Y}$
equipped with the stupid filtration $F$.
\end{defn}

\begin{rmk}
In the case where $S$ is a point
and $I=\ski$,
the complex $sB(f)$ is same as
$sB_X(D_1, \dots, D_k)$ in \cite[(3.12)]{FujisawaLHSSV}
except the sign of the differentials.
\end{rmk}

\begin{rmk}
\label{remark for grF}
Because we have $d_0(F^psB_I(f)^n) \subset F^{p+1}sB_I(f)^{n+1}$
for all $n,p$,
the complex
$\gr_F^psB_I(f)$ turns out to be a complex of $\cO_{X[I]}$-modules
for all $p$.
We trivially have
\begin{equation}
\label{description of grF:eq}
\gr_F^psB_I(f)^n
=\bigoplus_{\substack{q \in \bnnZ^I \\
                      |q|=n-p}}
B_I(f)^{p,q}
=\bigoplus_{\substack{q \in \bnnZ^I \\
                      |q|=n-p}}
\biggl(\omega^{n+|I|}_{X/Y_I}\biggl/
\sum_{i \in I}W(D_i)_{q_i}\biggr)
\end{equation}
for all $n, p$.
\end{rmk}

\begin{defn}
\label{definition of L(K) on BI}
In the same situation as in Definition \ref{definition of B_I(f)},
a finite increasing filtration $L(K)$ on $B_I(f)^{p,q}$ for $K \subset I$
is defined by
\begin{equation}
L(K)_mB_I(f)^{p,q}
=W(D_K)_{m+2|q_K|+|K|}\omega_{X/Y_I}^{p+|q|+|I|}\biggl/
\sum_{i \in I}W(D_i)_{q_i}
\label{definition of L(K) on BI:eq}
\end{equation}
where $q_K$ is the image of $q$
by the projection $\bZ^I \longrightarrow \bZ^K$
as in \ref{notation for ZI}.
We can easily see that $L(K)$ is preserved by every $d_i$.
Thus we obtain a finite increasing filtration $L(K)$ on $sB_I(f)$.
We write $L=L(\ski)$ on $sB(f)$ for short.
By definition, $L(\emptyset)$ is the trivial filtration,
that is,
$L(\emptyset)_{-1}sB_I(f)=0$
and $L(\emptyset)_0sB_I(f)=sB_I(f)$.
\end{defn}

\begin{lem}
\label{lemma on grL(K)sB}
For any $K \subset I \subset \ski$,
the residue morphism for $W(D_K)$
induces the isomorphism of complexes
\begin{equation}
\label{grL(K)sB:eq}
\gr_m^{L(K)}sB_I(f)
\simeq
\bigoplus_{q \in \bnnZ^K}
\bigoplus_{\Gamma \in S_{m+2|q|+|K|}^{\ge q+e_K}(\Lambda_K)}
sB_{I \setminus K}(f_{\Gamma})[-m-2|q|],
\end{equation}
where $f_{\Gamma}:X[\Gamma] \longrightarrow Y[K]$
is the morphism
in the commutative diagram \eqref{square for Gamma:eq}.
Note that $\Gamma \cap \Lambda_i \not= \emptyset$ for all $i \in K$
if $\Gamma \in S_{m+2|q|+|K|}^{\ge q+e_K}(\Lambda_K)$.
Under the isomorphism \eqref{grL(K)sB:eq} above,
we have
\begin{equation*}
\begin{split}
&L(J)\gr_m^{L(K)}sB_I(f) \\
&\qquad
\simeq
\bigoplus
L(J \setminus K \cap J)
[\,|\Gamma \cap \Lambda_{K \cap J}|-2|q_{K \cap J}|-|K \cap J|\,]\,
sB_{I \setminus K}(f_{\Gamma})[-m-2|q|]
\end{split}
\end{equation*}
for any $J \subset I$ and
\begin{equation*}
F\gr_m^{L(K)}sB_I(f)
\simeq
\bigoplus F[-m-|q|]\,
sB_{I \setminus K}(f_{\Gamma})[-m-2|q|] ,
\end{equation*}
where the direct sums above are taken over the same index set
\begin{equation}
\{(q,\Gamma) \mid
q \in \bnnZ^K,
\Gamma \in S_{m+2|q|+|K|}^{\ge q+e_K}(\Lambda_K)\}
\end{equation}
as in \eqref{grL(K)sB:eq}.
In particular, we have
\begin{equation*}
L(J)\gr_m^{L(K)}sB_I(f)
\simeq
\bigoplus
L(J \setminus K)[m]
sB_{I \setminus K}(f_{\Gamma})[-m-2|q|]
\end{equation*}
for $J \supset K$
and
\begin{equation*}
L(J)\gr_m^{L(K)}sB_I(f)
\simeq
\bigoplus
L(J)
sB_{I \setminus K}(f_{\Gamma})[-m-2|q|]
\end{equation*}
for $J$ with $J \cap K=\emptyset$.
\end{lem}
\begin{proof}
We obtain the conclusion
by Lemma \ref{local description of semistable morphism}
and by the local computation
as in \cite[Proposition 3.6]{DeligneED}.
Note that we need the total order on $\Lambda$
in \ref{setting for a semistable morphism} here.
\end{proof}

\begin{rmk}
In the lemma above,
we use the residue morphism
defined as in \cite[\RomII.3.7]{DeligneED}.
It is different by the sign
from the residue morphism used in \cite{FujisawaLHSSV}.
\end{rmk}

%\begin{para}
%\label{para:1}
%Since $sB_I(f)$ is a complex
%of $(\iota_{|I})_{\ast}f_I^{-1}\cO_{Y[I]}$-modules,
%$\iota_{|I}^{-1}sB_I(f)$
%is a complex of $f_I^{-1}\cO_{Y[I]}$-modules on $X[I]$.
%The filtraions $F$ and $L(K)$ for $K \subset I$
%on $\iota_{|I}^{-1}sB_I(f)$
%are defined by
%\begin{equation}
%F^p\iota_{|I}^{-1}sB_I(f)
%=\iota_{|I}^{-1}F^psB_I(f), \quad
%L(K)_m\iota_{|I}^{-1}sB_I(f)
%=\iota_{|I}^{-1}L(K)_msB_I(f)
%\end{equation}
%for all $m,p$.
%We trivailly have
%\begin{equation}
%\label{eq:6}
%\begin{split}
%&\iota_{|I}^{-1}\gr_F^psB_I(f)
%\simeq
%\gr_F^p\iota_{|I}^{-1}sB_I(f) \\
%&\iota_{|I}^{-1}\gr_m^{L(K)}sB_I(f)
%\simeq
%\gr_m^{L(K)}\iota_{|I}^{-1}sB_I(f) \\
%&\iota_{|I}^{-1}\gr_F^p\gr_m^{L(K)}sB_I(f)
%\simeq
%\gr_F^p\gr_m^{L(K)}\iota_{|I}^{-1}sB_I(f)
%\end{split}
%\end{equation}
%for all $m,p$
%(cf. \cite[(1.4.3)]{DeligneII}).
%Then the complex
%$(\iota_{|I})_{\ast}\iota_{|I}^{-1}sB_I(f)$
%carries the filtrations $F$ and $L(K)$ for $K \subset I$
%as usual.
%By the inclusion \eqref{support of B:eq},
%the canonical morphism
%\begin{equation}
%\label{eq:5}
%sB_I(f)
%\longrightarrow
%(\iota_{|I})_{\ast}\iota_{|I}^{-1}sB_I(f)
%\end{equation}
%is an isomorphism of complexes
%of $(\iota_{|I})_{\ast}f_I^{-1}\cO_{Y[I]}$-modules,
%under which the filtrations $F$ and $L(K)$ for $K \subset L$
%are identified on the both sides.
%\end{para}

\begin{lem}
\label{flatness lemma}
For $K \subset I \subset \ski$,
the $\cO_{X[I]}$-sheaf
$\gr_F^p\gr_m^{L(K)}sB_I(f)^n$
is $f_I^{-1}\cO_{Y[I]}$-flat
for all $m,n,p$.
\end{lem}
\begin{proof}
It is sufficient to consider the stalk at
a point $x \in X[I]$.
First, we treat the case of $K=I$.
Let $x$ be a point of $X[I]$ and $y=f(x)=f_{\Gamma}(x) \in Y[I]$.
We have an isomorphism
\begin{equation*}
\gr_F^p\gr_m^{L(I)}sB_I(f)_x^n
\simeq
\bigoplus
\omega^{n-m-2|q|}_{X[\Gamma]/Y[I],x}
\end{equation*}
by Lemma \ref{lemma on grL(K)sB},
where the direct sum on the right hand side
is taken over the index set
\begin{equation*}
\{(q,\Gamma) \mid q \in \bnnZ^I \text{ with } |q|=n-p,
\Gamma \in S_{m+2|q|+|I|}^{\ge q+e_I}(\Lambda_I), x \in X[\Gamma]\}.
\end{equation*}
Therefore $\gr_F^p\gr_m^{L(I)}sB_I(f)_x^n$
are $\cO_{Y[I],y}$-flat
for all $m,n,p$
because the semistable morphism $f_{\Gamma}:X[\Gamma] \longrightarrow Y[I]$
is flat by Lemma \ref{local description of semistable morphism}.
In particular,
\begin{equation*}
\gr_F^psB_I(f)_x^n,\quad
\gr_m^{L(I)}sB_I(f)_x^n, \quad
sB_I(f)_x^n
\end{equation*}
are $\cO_{Y[I],y}$-flat for all $m,n,p$.
Hence we obtain the conclusion for general $K \subset I$
by Lemma \ref{lemma on grL(K)sB} again.
\end{proof}

\begin{defn}
\label{definition of theta}
For $J \subset I \subset \ski$,
a morphism of $\cO_X$-modules
\begin{equation*}
\omega_X^{n+|J|}
\longrightarrow
\omega_X^{n+|I|}
\end{equation*}
is defined by sending $\omega \in \omega_X^{n+|J|}$ to
\begin{equation}
\dlog t_{i_1} \wedge \dlog t_{i_2}
\wedge \dots \wedge \dlog t_{i_l} \wedge \omega
\in \omega_X^{n+|I|}
\label{morphism on Omega for J|I:eq}
\end{equation}
where $I \setminus J=\{i_1,i_2,\dots, i_l\}$
with $1 \le i_1 < i_2 < \dots < i_l \le k$.
We can easily check that this morphism
induces a morphism of $\iota_{J|I}^{-1}\cO_{X[J]}$-modules
\begin{equation}
\theta_{J|I}(f):
\iota_{J|I}^{-1}B_J(f)^{p,q}
\longrightarrow
B_I(f)^{p,q}
\end{equation}
for $p \in \bnnZ$ and for $q \in \bnnZ^J \subset \bnnZ^I$.
Moreover, the equalities
\begin{align*}
&d_0\theta_{J|I}(f)=\theta_{J|I}(f)d_0 \\
&d_i\theta_{J|I}(f)=\theta_{J|I}(f)d_i \quad \text{for $i \in J$} \\
&d_i\theta_{J|I}(f)=0 \quad \text{for $i \in I \setminus J$}
\end{align*}
are easily seen.
Thus a morphism of complexes
\begin{equation}
\label{the morphism thetaJ|I:eq}
\theta_{J|I}(f):
\iota_{J|I}^{-1}sB_J(f) \longrightarrow sB_I(f)
\end{equation}
is obtained.
It is easy to see that
$\theta_{J|I}(f)$ preserves the filtrations $F$ and $L(K)$ for $K \subset J$.
We use the notation
$\theta_{J|}(f)=\theta_{J|\ski}(f)$,
$\theta_{|I}(f)=\theta_{\emptyset|I}(f)$
and $\theta(f)=\theta_{\emptyset|\ski}(f)$ for short.
\end{defn}

\begin{lem}
\label{transitivity of theta}
For $K \subset J \subset I \subset \ski$,
the diagram
\begin{equation*}
\begin{CD}
\iota_{K|I}^{-1}sB_K(f)
@>{\iota_{K|J}^{-1}\theta_{K|J}(f)}>>
\iota_{J|I}^{-1}sB_J(f) \\
@| @VV{\theta_{J|I}(f)}V \\
\iota_{K|I}^{-1}sB_K(f) @>>{\theta_{K|I}(f)}> sB_I(f)
\end{CD}
\end{equation*}
commutes up to sign.
\end{lem}
\begin{proof}
Easy.
\end{proof}

\begin{lem}
\label{lemma on grLsBJ to grLsBI}
For $K \subset J \subset I \subset \ski$,
the morphism $\theta_{J|I}(f)$
induces a quasi-isomorphism
\begin{equation}
f_I^{-1}\cO_{Y[I]}
\otimes
\iota_{J|I}^{-1}\gr_F^p\gr_m^{L(K)}sB_J(f)
\overset{\simeq}{\longrightarrow}
\gr_F^p\gr_m^{L(K)}sB_I(f)
\end{equation}
for all $m, p$.
\end{lem}
\begin{proof}
Once we obtain the conclusion for $K=\emptyset$,
we obtain the conclusion for general $K$
by using Lemma \ref{lemma on grL(K)sB}.
So, we assume that $K=\emptyset$.
Namely,
we will prove that the morphism
\begin{equation}
\label{qis theta:eq}
f_I^{-1}\cO_{Y[I]}
\otimes
\iota_{J|I}^{-1}\gr_F^psB_J(f)
\longrightarrow
\gr_F^psB_I(f)
\end{equation}
induced by $\theta_{J|I}(f)$
is a quasi-isomorphism.

First, we treat the case
where $I=\ski$ and $S$ is a point.
If $J=\emptyset$,
the conclusion is already proved in \cite[Corollary (5.14)]{FujisawaLHSSV}.
Then the conclusion for $J=\emptyset$ above
and Lemma \ref{lemma on grL(K)sB}
imply that the morphism
$\theta_{J|}(f)$ induces a filtered quasi-isomorphism
\begin{equation}
f_0^{-1}\cO_{Y_0}
\otimes
\iota_{J|}^{-1}\gr_F^p\gr_m^{L(J)}sB_J(f)
\longrightarrow
\gr_F^p\gr_m^{L(J)}sB(f)
\end{equation}
for all $m,p$ and for any $J$.
Thus we conclude that
the morphism \eqref{qis theta:eq}
is a quasi-isomorphism for $I=\ski$ and for $Y=\Delta^k$
by using the $f_J^{-1}\cO_{Y[J]}$-flatness of
$\gr_F^p\gr_m^{L(J)}sB_J(f)^n
\simeq \gr_m^{L(J)}\gr_F^psB_J(f)^n$
in Lemma \ref{flatness lemma}.

Next, we treat the general case.
By definition,
we have $Y=\Delta^I \times Y_I$,
where $\Delta^I$ denotes the polydisc of dimension $|I|$
with the coordinate functions $(t_i)_{i \in I}$.
Since the problem is of local nature,
we may assume that
\begin{equation*}
X=\Delta^n \times Z, \qquad
f=g \times h,
\end{equation*}
where
$g: \Delta^n \longrightarrow \Delta^I$
is a product of the morphisms
of the form \eqref{semistable over Delta}
and where $h: Z \longrightarrow Y_I$ is a semistable morphism
along $E_{\overline{I}} \cap Y_I$.
We denote the projections
$X \longrightarrow \Delta^n$ and $X \longrightarrow Z$
by $\pr_1$ and $\pr_2$
and set $D_Z=h^{\ast}(E_{\overline{I}} \cap Y_I)$
for a while.
From the equality \eqref{description of grF:eq},
we have the identifications
\begin{equation}
\label{decomposition of  sBI(f):eq}
\begin{split}
&\gr_F^psB_I(f)
\simeq
\bigoplus_{q \ge 0}\pr_1^{\ast}\gr_F^{p-q}sB(g)[-q]
\otimes_{\cO_X} \pr_2^{\ast}\omega_{Z/Y_I}^q \\
&\gr_F^psB_J(f)
\simeq
\bigoplus_{q \ge 0}
\pr_1^{\ast}\gr_F^{p-q}sB_J(g)[-q]
\otimes_{\cO_X} \pr_2^{\ast}\omega_{Z/Y_I}^q ,
\end{split}
\end{equation}
under which the morphism $\gr_F^p\theta_{J|I}(f)$
is identified with the direct sum
of the morphisms $\pr_1^{\ast}\gr_F^{p-q}\theta_{J|}(g) \otimes \id$
over all $q \ge 0$.
From what we have proved above,
$\pr_1^{\ast}\gr_F^{p-q}\theta_{J|}(g) \otimes \id$
induces a quasi-isomorphism
because $\pr_1:X \longrightarrow \Delta^n$ is flat
and $\pr_2^{\ast}\omega_{Z/Y_I}^q$ is a locally free $\cO_X$-module.
Thus we conclude that
$\theta_{J|I}(f)$ induces a quasi-isomorphism \eqref{qis theta:eq}.
\end{proof}

As in the proof above,
the $f_J^{-1}\cO_{Y[J]}$-flatness of
$\gr_F^p\gr_m^{L(J)}sB_J(f)^n$
implies the following:

\begin{cor}
The morphism $\theta_{J|I}(f)$ induces
a quasi-isomorphism
\begin{equation}
f_I^{-1}\cO_{Y[I]}
\otimes
\iota_{J|I}^{-1}sB_J(f)
\longrightarrow
sB_i(f)
\end{equation}
for any $J \subset I \subset \ski$.
\end{cor}

\begin{para}
\label{notation for a base change for S}
Let $S'$ be a complex manifold
and $g:S' \longrightarrow S$ a morphism of complex manifolds.
We set
$Y'=\Delta^k \times S'$,
$g=\id \times g: Y'=\Delta^k \times S' \longrightarrow Y=\Delta^k \times S$,
$E'_i=(\id \times g)^{\ast}E_i$
and $E'=(\id \times g)^{\ast}E$.
By the cartesian square
\begin{equation}
\begin{CD}
X' @>{g'}>> X \\
@V{f'}VV @VV{f}V \\
Y' @>>{\id \times g}> Y
\end{CD}
\end{equation}
we define $f': X' \longrightarrow Y'$,
which is semistable along $E'$
by Lemma \ref{local description of semistable morphism}.
We set $D'=f'^{\ast}E'=g'^{\ast}D$
and $D'_i=f'^{\ast}E'_i=g'^{\ast}D_i$ for $i=1,2, \dots, k$.
We use the notation such as $Y'[I], Y'_I, X'[I], X'[\Gamma]$
as in the case of $f: X \longrightarrow Y$.
From the cartesian squares
\begin{equation}
\begin{CD}
X' @>{g'}>> X \\
@V{f'}VV @VV{f}V \\
Y' @>>{\id \times g}> Y \\
@V{\pi'_I}VV @VV{\pi_I}V \\
Y'_I @>>{\id \times g}> Y_I
\end{CD}
\end{equation}
we have the canonical morphisms
\begin{equation}
g'^{-1}\omega^p_{X/Y_I}
\longrightarrow
\omega^p_{X'/Y'_I}
\end{equation}
for every $I \subset \ski$ and for all $p$,
which preserve the filtrations $W(D_i)$ and $W(D'_i)$
on the both sides for $i \in I$.
Thus we obtain the morphism of complexes
\begin{equation}
\label{base change morphism for sB:eq}
%f'^{-1}\cO_{Y'[I]} \otimes_{g'^{-1}f^{-1}\cO_{Y[I]}} g'^{-1}sB_I(f)
%\longrightarrow
%sB_I(f')
g'^{-1}sB_I(f)
\longrightarrow
sB_I(f')
\end{equation}
for every $I \subset \ski$,
which preserves the filtrations $F$ and $L(J)$ for all $J \subset I$.
\end{para}

\section{The Gauss-Manin connection}
\label{gauss-manin}

In this section,
we construct the Gauss-Manin connection
on $R^n(f_I)_{\ast}sB_I(f)$
as in the case on $R^nf_{\ast}\omega_{X/Y}$
in \eqref{the Gauss-Manin connection:eq}.
For this purpose,
we construct a complex $s\tB_I(f)$
which is a replacement of
$f_I^{-1}\cO_{Y[I]}
\otimes \iota_{|I}^{-1}\omega_X$
for $I \subset \ski$
and follow the arguments
in Definition \ref{definition of the filtration G}.
Then we describe its residues
by a similar way to \cite{SteenbrinkLHS}.

\begin{defn}
For a subset $I \subset \ski$,
we set
\begin{equation}
\tB_I(f)^{p,q}=
\omega^{p+|q|+|I|}_X\biggl/\sum_{i \in I}W(D_i)_{q_i}
\end{equation}
for $p \in \bnnZ$ and $q \in \bnnZ^I$,
with the morphisms
\begin{equation}
\begin{split}
&d_0=(-1)^{|I|}d:
\tB_I(f)^{p,q} \longrightarrow \tB_I(f)^{p+1,q} \\
&d_i=(-1)^{|I|}\dlog t_i \wedge:
\tB_I(f)^{p,q} \longrightarrow \tB_I(f)^{p,q+e_i}
\end{split}
\end{equation}
as in Definition \ref{definition of B_I(f)}.
Then we obtain the complex $s\tB_I(f)$
by setting
\begin{equation*}
s\tB_I(f)^n
=\bigoplus_{p+|q|=n}\tB_I(f)^{p,q}
=\bigoplus_{q \in \bnnZ^I}
\biggl(\omega^{n+|I|}_X\biggl/\sum_{i \in I}W(D_i)_{q_i} \biggr)
\end{equation*}
with the differential
$d=d_0+\sum_{i \in I}d_i$
as in Definition \ref{definition of B_I(f)}.
Then we have
\begin{equation}
\label{eq:7}
\supp(\tB_I(f)^{p,q}) \subset X[I]
\end{equation}
for any $p,q$ as before.
Although $\tB_I(f)^{p,q}$ is an $\cO_{X[I]}$-module for all $p,q$,
we consider $s\tB_I(f)$ as a complex of $\bC$-sheaves
because the morphism $d_0$ is not a morphism of $\cO_{X[I]}$-modules.
Because of \eqref{eq:7},
the complex $s\tB_I(f)$ is considered
as a complex of $\bC$-sheaves on $X[I]$.

The filtration $F$ on $s\tB_I(f)$ is defined by
\begin{equation}
F^ps\tB_I(f)^n
=\bigoplus_{\substack{p'+|q|=n \\
                      p' \ge p}}
\tB_I(f)^{p',q}
=\bigoplus_{\substack{q \in \bnnZ^I\\
                      |q| \le n-p}}
\biggl(\omega^{n+|I|}_X\biggl/\sum_{i \in I}W(D_i)_{q_i} \biggr)
\end{equation}
as in \eqref{definition of F on sBI(f):eq}.
For $K \subset I$,
the filtration $L(K)$ on $\tB_I(f)$ is defined by
\begin{equation}
L(K)_m\tB_I(f)^{p,q}
=W(D_K)_{m+2|q_K|+|K|}\omega_X^{p+|q|+|I|}\biggl/
\sum_{i \in I}W(D_i)_{q_i}
\end{equation}
as in \eqref{definition of L(K) on BI:eq}.
%As in \eqref{eq:5} and Remark \ref{rmk:2},
%the canonical morphism
%\begin{equation}
%\tB_I(f) \longrightarrow (\iota_{|I})_{\ast}\iota_{|I}^{-1}\tB_I(f)
%\end{equation}
%is an isomorphism of complexes of $\bC$-sheaves on $X$,
%under which the filtrations $F$ and $L(K)$ for $K \subset I$
%are identified on the both sides.
The canonical projection
$\omega^n_X \longrightarrow \omega^n_{X/Y_I}$ induces
a morphism of complexes
\begin{equation}
s\tB_I(f) \longrightarrow sB_I(f) ,
\end{equation}
which is denoted by $\pr_I$ if there is no danger of confusion.
It is easy to see that $\pr_I$ preserves the filtrations $F$ and $L(K)$.
\end{defn}

\begin{defn}
For a subset $I \subset \ski$,
a finite decreasing filtration $G(I)$ on $\omega_X$ is defined by
\begin{equation}
G(I)^p\omega^n_X
=\image((\pi_I f)^{-1}\omega^p_{Y_I}
\otimes_{(\pi_I f)^{-1}\cO_{Y_I}} \omega^{n-p}_X
\overset{\wedge}{\longrightarrow} \omega^n_X)
\end{equation}
for all $n,p$.
By definition,
the filtration $G(\emptyset)$ is nothing but the filtration $G$
defined in Definition \ref{definition of the filtration G}.
Because we have
\begin{equation}
\dlog t_i \wedge G(I)^p\omega_X \subset G(I)^p\omega_X
\end{equation}
for every $i \in I$ and for every $p$,
the filtration $G$ on $\tB_I(f)$ is defined by
\begin{equation}
G^r\tB_I(f)^{p,q}
=G(I)^r\omega^{p+|q|+|I|}_X \biggl/
\sum_{i \in I}W(D_i)_{q_i}
\end{equation}
for all $p,q,r$.
Thus we obtain a finite decreasing filtration $G$ on $s\tB_I(f)$.
\end{defn}

%\begin{para}
%As in \ref{para:1},
%the complex of $\bC$-sheaves $\iota_{|I}^{-1}s\tB_I(f)$ on $X[I]$
%carries the filtrations $F$, $L(K)$ for $K \subset I$ and $G$.
%We have the canonical isomorphisms similar to
%\eqref{eq:6}.
%\end{para}

\begin{lem}
\label{lemma on grGstBI}
We have the isomorphism
\begin{equation}
f_I^{-1}\omega^p_{Y[I]}
\otimes
sB_I(f)[-p]
\overset{\simeq}{\longrightarrow}
\gr_G^ps\tB_I(f)
\label{identification for grG:eq}
\end{equation}
for every $p$,
under which
\begin{align}
&f_I^{-1}\omega^p_{Y[I]}
\otimes
F^{r-p}sB_I(f)[-p]
\simeq
F^r\gr_G^ps\tB_I(f) \\
&f_I^{-1}\omega^p_{Y[I]}
\otimes
L(K)_msB_I(f)
\simeq
L(K)_m\gr_G^ps\tB_I(f)
\end{align}
for $K \subset I \subset \ski$
and for all $m,r$.
\end{lem}
\begin{proof}
Easy from $\iota_{|I}^{-1}(\pi_If)^{-1}=(\pi_If\iota_{|I})^{-1}=f_I^{-1}$.
\end{proof}

\begin{defn}
The morphism of $E_1$-terms
\begin{equation}
E_1^{p,n}(R(f_I)_{\ast}s\tB_I(f),G)
\longrightarrow
E_1^{p+1,n}(R(f_I)_{\ast}s\tB_I(f),G)
\end{equation}
induces a morphism of $\bC$-sheaves
\begin{equation}
\label{nabla(I):eq}
\nabla(I):
\omega^p_{Y[I]} \otimes R^n(f_I)_{\ast}sB_I(f)
\longrightarrow
\omega^{p+1}_{Y[I]} \otimes R^n(f_I)_{\ast}sB_I(f)
\end{equation}
via the identification
\begin{equation}
E_1^{p,n}(R(f_I)_{\ast}s\tB_I(f),G)
\simeq
\omega^p_{Y[I]} \otimes R^n(f_I)_{\ast}sB_I(f)
\end{equation}
which is given by Lemma \ref{lemma on grGstBI}
and by the projection formula
together with the local freeness of $\omega_{Y[I]}^p$.
We have
\begin{equation}
\begin{split}
&\nabla(I)(\omega^p_{Y[I]}
\otimes
L(K)_mR^n(f_I)_{\ast}sB_I(f))
\subset
\omega^{p+1}_{Y[I]}
\otimes
L(K)_mR^n(f_I)_{\ast}sB_I(f) \\
&\nabla(I)(\omega^p_{Y[I]}
\otimes
F^rR^n(f_I)_{\ast}sB_I(f))
\subset
\omega^{p+1}_{Y[I]}
\otimes
F^{r-1}R^n(f_I)_{\ast}sB_I(f)
\end{split}
\end{equation}
for all $m,r$ by Lemma \ref{lemma on grGstBI} again.
\end{defn}

\begin{rmk}
For the case of $I=\emptyset$,
the filtered complex $(s\tB_{\emptyset}(f),G)$
coincides with $(\omega_X,G)$
in Definition \ref{definition of the filtration G}.
Therefore we have $\nabla(\emptyset)=\nabla$.
\end{rmk}

\begin{prop}
\label{wedge product for omegaY and B}
For $I \subset \ski$, we have
\begin{equation}
(-1)^{|I|}\nabla(I)(\omega \otimes v)
=d\omega \otimes v+(-1)^p\omega \wedge (-1)^{|I|}\nabla(I)(v)
\end{equation}
for $\omega \in \omega^p_{Y[I]}$
and $v \in R^n(f_I)_{\ast}sB_I(f)$.
\end{prop}
\begin{proof}
We consider the complex
$f_I^{-1}\omega_{Y[I]} \otimes_{\bC} s\tB_I(f)$
equipped with the filtration $G$ defined by
\begin{equation}
G^r(f_I^{-1}\omega_{Y[I]} \otimes_{\bC} s\tB_I(f))
=\sum_{a+b=r}
f_I^{-1}G^a\omega_{Y[I]} \otimes_{\bC} G^bs\tB_I(f) ,
\end{equation}
where the filtration $G$ on $\omega_{Y[I]}$ denotes the stupid filtration.
Then we have the isomorphism
\begin{equation}
\label{direct sum decomposition of omegaY[I] otimes stB:eq}
\begin{split}
E_1^{p,n}(R(f_I)_{\ast}f_I^{-1}\omega_{Y[I]}
&\otimes_{\bC}
s\tB_I(f), G) \\
&\simeq
\bigoplus_{a+b=p}
\omega^a_{Y[I]} \otimes_{\bC} \omega^b_{Y[I]}
\otimes_{\cO_{Y[I]}}R^n(f_I)_{\ast}sB_I(f)
\end{split}
\end{equation}
for every $n,p$.
On the other hand, a morphism
\begin{equation}
(\pi_If)^{-1}\omega_{Y_I}^p
\otimes_{\bC}
\omega_X^q
\longrightarrow
\omega_X^{p+q}
\end{equation}
is defined by sending $\omega \otimes \eta$
to $(-1)^{p|I|}\omega \wedge \eta$.
Direct computation shows
that this morphism gives a morphism of complexes
\begin{equation}
\label{product on omegaY[I] and stB:eq}
\iota_{|I}^{-1}(\pi_If)^{-1}\omega_{Y_I}
\otimes_{\bC}
s\tB_I(f)
=
f_I^{-1}\omega_{Y[I]} \otimes_{\bC} s\tB_I(f)
\longrightarrow
s\tB_I(f),
\end{equation}
which preserves the filtrations $G$ on the both sides.
Thus we obtain the morphism of $E_1$-terms
\begin{equation}
E_1^{p,n}(R(f_I)_{\ast}
(f_I^{-1}\omega_{Y[I]} \otimes_{\bC} s\tB_I(f)), G)
\longrightarrow
E_1^{p,n}(R(f_I)_{\ast}s\tB_I(f),G)
\end{equation}
which is compatible with the morphism of $E_1$-terms.
The compatibility of this morphism on the direct summand
$\omega^p_{Y[I]} \otimes_{\bC} R^n(f_I)_{\ast}sB_I(f)$
of \eqref{direct sum decomposition of omegaY[I] otimes stB:eq}
implies the equality
\begin{equation}
(-1)^{p|I|}\nabla(I)(\omega \otimes v)
=(-1)^{(p+1)|I|}d\omega \otimes v+(-1)^{p|I|+p}\omega \wedge \nabla(I)(v)
\end{equation}
for $\omega \in \omega^p_{Y[I]}$
and $v \in R^n(f_I)_{\ast}sB_I(f)$.
\end{proof}

\begin{defn}
The filtration $G$ on $s\tB_I(f)$
induces the morphism
\begin{equation}
\begin{split}
\gamma_I(f):
f_I^{-1}\omega^p_{Y[I]}
\otimes
sB_I(f)[-p]
\longrightarrow
f_I^{-1}\omega^{p+1}_{Y[I]}
\otimes
sB_I(f)[-p]
\end{split}
\end{equation}
in the derived category
for every $p$.
Then we have
\begin{equation}
\nabla(I)=R^{p+n}(f_I)_{\ast}(\gamma_I(f))
\end{equation}
by definition.
\end{defn}

\begin{defn}
For $J \subset I \subset \ski$,
the morphism
defined by \eqref{morphism on Omega for J|I:eq}
induces a morphism of complexes
\begin{equation}
\widetilde{\theta}_{J|I}(f):
\iota_{J|I}^{-1}s\tB_J(f)
\longrightarrow
s\tB_I(f)
\end{equation}
as in Definition \ref{definition of theta}.
\end{defn}

\begin{lem}
The morphism $\widetilde{\theta}_{J|I}(f)$
preserves the filtration $G$ on the both sides.
Under the isomorphism \eqref{identification for grG:eq},
the morphism
\begin{equation}
\gr_G^p\widetilde{\theta}_{J|I}(f):
\gr_G^p\iota_{J|I}^{-1}\tB_J(f)
\longrightarrow
\gr_G^p\tB_I(f)
\end{equation}
is identified with the morphism
\begin{equation}
\begin{split}
(-1)^{p(|I|-|J|)}f_I^{-1}\pr_{J|I}
&\otimes
\theta_{J|I}(f)[-p] \\
&:f_I^{-1}\iota_{J|I}^{-1}\omega^p_{Y[J]}
\otimes
\iota_{J|I}^{-1}sB_J(f)[-p]
\longrightarrow
f_I^{-1}\omega^p_{Y[I]}
\otimes
sB_I(f)[-p]
\end{split}
\end{equation}
for every $p$,
where $\pr_{J|I}$ is the morphism
\eqref{the projection between Omega from Y[J] to Y[I]:eq}.
\end{lem}
\begin{proof}
Easy by definition.
\end{proof}

\begin{cor}
\label{cor on compatibility of nabla(J) and nabla(I)}
The diagram
\begin{equation}
\begin{CD}
\iota_{J|I}^{-1}\omega^p_{Y[J]}
\otimes
\iota_{J|I}^{-1}R^n(f_J)_{\ast}sB_J(f)
@>{(-1)^{|J|}\iota_{J|I}^{-1}\nabla(J)}>>
\iota_{J|I}^{-1}\omega^{p+1}_{Y[J]}
\otimes
\iota_{J|I}^{-1}R^n(f_J)_{\ast}sB_J(f) \\
@V{\pr_{J|I} \otimes \theta_{J|I}(f)}VV
@VV{\pr_{J|I} \otimes \theta_{J|I}(f)}V \\
\omega^p_{Y[I]}
\otimes
R^n(f_I)_{\ast}sB_I(f)
@>>{(-1)^{|I|}\nabla(I)}>
\omega^{p+1}_{Y[I]}
\otimes
R^n(f_I)_{\ast}sB_I(f)
\end{CD}
\end{equation}
is commutative.
\end{cor}

\begin{defn}
\label{definition of nu}
For $j \in I$,
the canonical projection
\begin{equation}
\begin{split}
B_I(f)^{p,q}
=\omega^{p+|q|+|I|}_{X/Y_I}&\left/\sum_{i \in I}W(D_i)_{q_i} \right. \\
&\longrightarrow
B_I(f)^{p-1,q+e_j}
=\omega^{p+|q|+|I|}_{X/Y_I} \left/
\sum_{i \in I, i \not= j}W(D_i)_{q_i}+W(D_j)_{q_j+1} \right.
\end{split}
\end{equation}
induces a morphism of complexes
\begin{equation}
\nu_{j|I}(f): sB_I(f) \longrightarrow sB_I(f)
\end{equation}
for $j \in I$.
On the other hand,
we define a morphism of complexes
$\mu_{J|I}(f)$ by
\begin{equation}
\mu_{J|I}(f)=\sum_{j \in I \setminus J}\dlog t_j \otimes \nu_{j|I}(f):
sB_I(f)
\longrightarrow
(f_I)^{-1}\iota_{J|I}^{-1}\omega_{Y[J]/Y_I}^1
\otimes
sB_I(f)
\end{equation}
for $J \subsetneq I$.
Similarly,
the projection
\begin{align}
\tB_I(f)^{p,q}
=\omega^{p+|q|+|I|}_X&\left/\sum_{i \in I}W(D_i)_{q_i} \right. \\
&\longrightarrow
\tB_I(f)^{p-1,q+e_j}
=\omega^{p+|q|+|I|}_X\left/
\sum_{i \in I, i \not= j}W(D_i)_{q_i}+W(D_j)_{q_j+1} \right.
\end{align}
induces a morphism of complexes
\begin{equation}
\widetilde{\nu}_{j|I}(f): s\tB_I(f) \longrightarrow s\tB_I(f)
\end{equation}
for $j \in I$.
\end{defn}

\begin{lem}
\label{lemma for nu and L}
For $K \subset I \subset \ski$,
we have
\begin{equation*}
\begin{split}
&\nu_{j|I}(f)(L(K)_msB_I(f)) \subset L(K)_msB_I(f), \\
&\nu_{j|I}(f)(F^psB_I(f)) \subset F^{p-1}sB_I(f)
\end{split}
\end{equation*}
for all $m,p$ and for all $j \in I$.
If $j \in K$ in addition,
we have
\begin{equation*}
\nu_{j|I}(f)(L(K)_msB_I(f)) \subset L(K)_{m-2}sB_I(f)
\end{equation*}
for every $m$.
Therefore we have
\begin{equation}
\mu_{J|I}(f)(L(K)_msB_I(f))
\subset
(f_I)^{-1}\omega_{Y[J]/Y_I}^1
\otimes
L(K)_{m-2}sB_I(f)
\end{equation}
for every $m$ if $K \cup J=I$.
\end{lem}
\begin{proof}
Easy.
\end{proof}

\begin{lem}
The morphism $\widetilde{\nu}_{j|I}(f)$ preserves the filtration $G$
on $s\tB_I(f)$ for every $j \in I$.
Moreover
$\gr_G^p\widetilde{\nu}_{j|I}(f)$
coincides with
$\id \otimes \nu_{j|I}(f)[-p]$
for every $j \in I$
under the identification
\eqref{identification for grG:eq}.
\end{lem}
\begin{proof}
Easy by definition.
\end{proof}

\begin{cor}
\label{corollary for commutativity of nu and gamma}
The diagram
\begin{equation}
\begin{CD}
sB_I(f)
@>{\gamma_I(f)}>>
f_I^{-1}\omega^1_{Y[I]}
\otimes
sB_I(f) \\
@V{\nu_{j|I}(f)}VV
@VV{\id \otimes \nu_{j|I}(f)}V \\
sB_I(f)
@>>{\gamma_I(f)}>
f_I^{-1}\omega^1_{Y[I]}
\otimes
sB_I(f)
\end{CD}
\end{equation}
is commutative in the derived category.
\end{cor}

\begin{lem}
\label{lemma for nu and residue of nabla}
For $J \subsetneq I \subset \ski$,
the diagram
\begin{equation*}
\begin{CD}
\iota_{J|I}^{-1}sB_J(f)
@>{\iota_{J|I}^{-1}\gamma_J(f)}>>
f_I^{-1}\iota_{J|I}^{-1}\omega_{Y[J]}^1
\otimes
\iota_{J|I}^{-1}sB_J(f) \\
@V{\theta_{J|I}(f)}VV
@VV{f_I^{-1}\iota_{J|I}^{-1}\pr \otimes \theta_{J|I}(f)}V \\
sB_I(f)
@>>{(-1)^{|J|+1}\mu_{J|I}(f)}>
f_I^{-1}\iota_{J|I}^{-1}\omega_{Y[J]/Y_I}^1
\otimes
sB_I(f)
\end{CD}
\end{equation*}
is commutative,
where $\pr: \omega_{Y[J]}^1 \longrightarrow \omega_{Y[J]/Y_I}^1$
denotes the canonical projection.
\end{lem}
\begin{proof}
We define a complex of
$\bC$-sheaves
$C_{J|I}(f)$ on $X[I]$ by
\begin{equation*}
C_{J|I}(f)^p=
(f_I^{-1}\iota_{J|I}^{-1}\omega^1_{Y[J]/Y_I}
\otimes
sB_I(f)^{p-1})
\oplus
sB_I(f)^p
\end{equation*}
with the differentials
\begin{equation*}
d(\omega \otimes x, y)
=(-\omega \otimes dx+(-1)^{|J|+1}\mu_{J|I}(f)(y), dy)
\end{equation*}
for $\omega \in f_I^{-1}\iota_{J|I}^{-1}\omega^1_{Y[J]/Y_I}$,
$x \in sB_I(f)^{p-1}$
and $y \in sB_I(f)^p$.
We can easily check the equality $d^2=0$
from the fact that $\mu_{J|I}(f)$ is a morphism of complexes.
Then a morphism of complexes
\begin{equation*}
C_{J|I}(f) \longrightarrow sB_I(f)
\end{equation*}
is defined by the projection
$C_{J|I}(f)^p \longrightarrow sB_I(f)^p$.
Moreover, the injection
\begin{equation}
f_I^{-1}\iota_{J|I}^{-1}\omega_{Y[J]/Y_I}^1
\otimes
sB_I(f)^{p-1}
\longrightarrow
C_{J|I}(f)^p
\end{equation}
induces a morphism of complexes
\begin{equation*}
f_I^{-1}\iota_{J|I}^{-1}\omega_{Y[J]/Y_I}^1
\otimes
sB_I(f)[-1]
\longrightarrow
C_{J|I}(f)
\end{equation*}
which fits in the exact sequence
\begin{equation*}
\begin{CD}
0 @>>> f_I^{-1}\iota_{J|I}^{-1}\omega_{Y[J]/Y_I}^1 \otimes sB_I(f)[-1]
  @>>> C_{J|I}(f)
  @>>> sB_I(f) @>>> 0
\end{CD}
\end{equation*}
by definition.

We denote by $I \setminus J=\{i_1,i_2,\dots,i_l\}$
with $1 \le i_1 < i_2 < \dots < i_l \le k$.
The morphism
\begin{equation*}
\omega_X^p \longrightarrow \omega_X^{p-1+l}
\end{equation*}
sending $\omega \in \omega_X^p$ to
\begin{equation*}
(-1)^{l-j}\dlog t_{i_1} \wedge \dots \wedge \dlog t_{i_{j-1}}
\wedge \dlog t_{i_{j+1}} \wedge \dots \wedge \dlog t_{i_l} \wedge \omega
\end{equation*}
induces a morphism
\begin{equation*}
\eta_j:
\iota_{J|I}^{-1}\tB_J(f)^p
\longrightarrow
B_I(f)^{p-1}
\end{equation*}
which satisfies the equality
\begin{equation}
\eta_j(G^2\iota_{J|I}^{-1}s\tB_J(f)^p)=0
\end{equation}
for $p \in \bZ$ and for $j=\li$.

Now a morphism
\begin{equation*}
\eta:
\iota_{J|I}^{-1}s\tB_J(f)^p\bigl/G^2\iota_{J|I}^{-1}s\tB_J(f)^p
\longrightarrow
C_{J|I}(f)^p
\end{equation*}
is defined by
\begin{equation*}
\eta(\omega)
=(\sum_{j=1}^{l}\dlog t_{i_j} \otimes \eta_j(\omega),
\pr_I \cdot \widetilde{\theta}_{J|I}(f)(\omega))
\end{equation*}
for $\omega \in \iota_{|I}^{-1}s\tB_J(f)^p$.
It is easy to check that $\eta$ defines a morphism of complexes.
Then the diagram
\begin{equation}
\begin{CD}
0 @. 0 \\
@VVV @VVV \\
f_I^{-1}\iota_{J|I}^{-1}\omega_{Y[J]}^1
\otimes
\iota_{J|I}^{-1}sB_J(f)[-1]
@>{f_I^{-1}\iota_{J|I}^{-1}\pr \otimes \theta_{J|I}(f)[-1]}>>
f_I^{-1}\iota_{J|I}^{-1}\omega_{Y[J]/Y_I}^1
\otimes
sB_I(f)[-1] \\
@VVV @VVV \\
\iota_{J|I}^{-1}s\tB_J(f)\bigl/G^2\iota_{J|I}^{-1}s\tB_J(f)
@>{\eta}>>
C_{J|I}(f) \\
@VVV @VVV \\
\iota_{J|I}^{-1}sB_J(f)
@>{\theta_{J|I}(f)}>>
sB_I(f) \\
@VVV @VVV \\
0 @. 0
\end{CD}
\end{equation}
is commutative with exact columns.
Thus we obtain the conclusion.
\end{proof}

\begin{defn}
\label{definition of N}
The morphism
\begin{equation}
R^n(f_I)_{\ast}(\nu_{j|I}(f)):
R^n(f_I)_{\ast}sB_I(f)
\longrightarrow
R^n(f_I)_{\ast}sB_I(f)
\end{equation}
induced by $\nu_{j|I}(f)$ for every $n$
is denoted by $N_{j|I}(f)$ for short.
We set
\begin{equation}
\label{definition of NJI(f;c):eq}
N_{J|I}(f;c)=\sum_{j \in J}c_jN_{j|I}(f):
R^n(f_I)_{\ast}sB_I(f)
\longrightarrow
R^n(f_I)_{\ast}sB_I(f)
\end{equation}
for $J \subset I$
and for $c=(c_j)_{j \in J} \in \bC^J$.
We simply denote $N_{J|I}(f)=N_{J|I}(f;e_J)$.
We use the convention
$N_I(f;c)=N_{I|I}(f;c)$,
$N_I(f)=N_I(f;e_I)$,
$N(f;c)=N_{\ski}(f;c)$
and $N(f)=N(f;e)$.
\end{defn}

\begin{lem}
For $K \subset I \subset \ski$,
we have
\begin{align}
&N_{j|I}(f)(F^pR^n(f_I)_{\ast}sB_I(f))
\subset F^{p-1}R^n(f_I)_{\ast}sB_I(f) \\
&N_{j|I}(f)(L(K)_mR^n(f_I)_{\ast}sB_I(f))
\subset L(K)_mR^n(f_I)_{\ast}sB_I(f)
\end{align}
for all $m,n,p$.
Moreover,
we have
\begin{equation}
N_{j|I}(f)(L(K)_mR^n(f_I)_{\ast}sB_I(f))
\subset L(K)_{m-2}R^n(f_I)_{\ast}sB_I(f)
\end{equation}
for all $m, n$
if $j \in K$.
Therefore
\begin{equation}
N_{J|I}(f;c)(L(K)_mR^n(f_I)_{\ast}sB_I(f))
\subset L(K)_{m-2}R^n(f_I)_{\ast}sB_I(f)
\end{equation}
for $J \subset K$, for $c \in \bC^J$ and for all $m,n,p$.
\end{lem}
\begin{proof}
Easy by Lemma \ref{lemma for nu and L}.
\end{proof}

\begin{lem}
The diagram
\begin{equation}
\begin{CD}
R^n(f_I)sB_I(f)
@>{\nabla(I)}>>
\omega_{Y[I]}^1
\otimes
R^n(f_I)sB_I(f) \\
@V{N_{j|I}(f)}VV @VV{\id \otimes N_{j|I}(f)}V \\
R^n(f_I)sB_I(f)
@>{\nabla(I)}>>
\omega_{Y[I]}^1
\otimes
R^n(f_I)sB_I(f)
\end{CD}
\end{equation}
is commutative for every $j \in I$.
Therefore the diagram
\begin{equation}
\begin{CD}
R^n(f_I)sB_I(f)
@>{\nabla(I)}>>
\omega_{Y[I]}^1
\otimes
R^n(f_I)sB_I(f) \\
@V{N_{J|I}(f;c)}VV @VV{\id \otimes N_{J|I}(f;c)}V \\
R^n(f_I)sB_I(f)
@>{\nabla(I)}>>
\omega_{Y[I]}^1
\otimes
R^n(f_I)sB_I(f)
\end{CD}
\end{equation}
is commutative for all $J \subset I \subset \ski$
and for all $c \in \bC^J$.
\end{lem}
\begin{proof}
Easy by Corollary \ref{corollary for commutativity of nu and gamma}.
\end{proof}

\begin{thm}
\label{thm on the compatibility of N and nu}
For the residue of $\nabla$
in \eqref{the residue of nabla:eq},
the diagram
\begin{equation}
\begin{CD}
\cO_{Y[J]}
\otimes
\iota_{|I}^{-1}R^nf_{\ast}\omega_{X/Y}
@>{\id \otimes \iota_{j|I}^{-1}\res_j(\nabla)}>>
\cO_{Y[J]}
\otimes
\iota_{|I}^{-1}R^nf_{\ast}\omega_{X/Y} \\
@VVV @VVV \\
R^n(f_I)_{\ast}sB_I(f)
@>>{-N_{j|I}(f)}>
R^n(f_I)_{\ast}sB_I(f)
\end{CD}
\end{equation}
is commutative for $j \in I$,
where the two vertical arrows are the morphisms
induced by $\theta_{|I}(f)$.
\end{thm}
\begin{proof}
By Lemma \ref{lemma for nu and residue of nabla}
for $J=\emptyset$,
we can easily obtain the conclusion.
\end{proof}

\begin{rmk}
\label{nu and base change for S}
Let $g: S' \longrightarrow S$ be a morphism of complex manifolds.
Then the construction of $\nu_{j|I}(f)$
is compatible with the base change by $g$.
More precisely, we have the commutative diagram
\begin{equation}
%\begin{CD}
%f'^{-1}\cO_{Y'[I]} \otimes_{g'^{-1}f^{-1}\cO_{Y[I]}} g'^{-1}sB_I(f)
%@>>>
%sB_I(f') \\
%@V{\id \otimes \nu_{j|I}(f)}VV @VV{\nu_{j|I}(f')}V \\
%f'^{-1}\cO_{Y'[I]} \otimes_{g'^{-1}f^{-1}\cO_{Y[I]}} g'^{-1}sB_I(f)
%@>>>
%sB_I(f')
%\end{CD}
\begin{CD}
g'^{-1}sB_I(f) @>>> sB_I(f') \\
@V{(g')^{-1}\nu_{j|I}(f)}VV @VV{\nu_{j|I}(f')}V \\
g'^{-1}sB_I(f) @>>> sB_I(f')
\end{CD}
\end{equation}
where the horizontal arrows are the morphism
\eqref{base change morphism for sB:eq}.
\end{rmk}

\section{Rational structures}
\label{section for the rational structure}

In this section,
we construct a rational structure on $sB_I(f)$
as in \cite{SteenbrinkLE},
which is inspired by the log geometry of Fontaine-Illusie
developed in \cite{KazuyaKato}.
By using these rational structures,
we prove the local freeness of
$R^n(f_I)_{\ast}\iota_{|I}^{-1}sB_I(f)$,
$\gr_m^{L(K)}R^n(f_I)_{\ast}\iota_{|I}^{-1}sB_I(f)$ etc.\ for
a proper semistable morphism $f$
in the next section.

\begin{para}
\label{situation for hodge theoretic properties}
Let $Y=\Delta^k \times S$ and $E$ be as in \ref{local situation}
and $f:X \longrightarrow Y$ a semistable morphism along $E$
as before.
We use the notation in \ref{notation for Y and E},
\ref{setting for a semistable morphism}
and \ref{local situation}.
\end{para}

\begin{para}
For $I \subset \ski$,
the open immersion
$X \setminus D_I \hookrightarrow X$
is denoted by $j_I$.
Then the sheaf of monoids
\begin{equation*}
M_X(D_I)=(j_I)_{\ast}\cO_{X \setminus D_I}^{\ast} \cap \cO_X
\end{equation*}
defines a log structure on $X$ (see \cite{KazuyaKato}).
We use the notation $M_X(D)$
for $I=\ski$ as before.
We have
\begin{equation*}
\cO_X^{\ast}=M_X(D_{\emptyset})
\subset M_X(D_J)
\subset M_X(D_I)
\subset M_X(D)
\end{equation*}
for $J \subset I \subset \ski$.
A morphism of abelian sheaves
\begin{equation*}
\mathbf{e}(f):\cO_X \longrightarrow M_X(D)\gp
\end{equation*}
is the composite of the exponential map
\begin{equation*}
\cO_X \ni g \mapsto \exp(2\pi\sqrt{-1} g) \in \cO_X^{\ast}
\end{equation*}
and the inclusion $\cO_X^{\ast} \hookrightarrow M_X(D)$.
We denote by
\begin{equation*}
\mathbf{e}(f)_{\bQ}:
\cO_X \longrightarrow M_X(D)\gp_{\bQ}=M_X(D)\gp \otimes_{\bZ} \bQ
\end{equation*}
the base extension of the morphism $\mathbf{e}(f)$.
For $I \subset \ski$,
a morphism of abelian sheaves
\begin{equation*}
\bZ^I \longrightarrow M_X(D)\gp
\end{equation*}
is defined by sending $e_i \in \bZ^I$
to the global section $t_i$ of $M_X(D)$ for $i \in I$.
Then a morphism of $\bQ$-sheaves
\begin{equation*}
\bQ^I \longrightarrow M_X(D)\gp_{\bQ}
\end{equation*}
is obtained by the base extension to $\bQ$.
The direct sum of the morphism above
and $\mathbf{e}(f)_{\bQ}$
gives us the morphism
\begin{equation*}
\bQ^I \oplus \cO_X \longrightarrow M_X(D)_{\bQ}\gp
\end{equation*}
denoted by $\varphi^I(f)$.
We write $\varphi(f)$ for $\varphi^{\ski}(f)$.
On the other hand,
$\varphi^{\emptyset}(f)$
is nothing but the morphism $\mathbf{e}(f)_{\bQ}$

A global section $1$ of $\cO_X$ can be considered
as a global section of $\bQ^I \oplus \cO_X$
by the canonical inclusion $\cO_X \subset \bQ^I \oplus \cO_X$.
Then the global section $1$
is contained in $\kernel(\varphi^I)$ for every $I \subset \ski$.
Hence the complex of $\bQ$-sheaves
\begin{equation*}
\kos(\varphi^I(f);\infty;1)
\end{equation*}
is defined in \cite[Definition 1.8]{FujisawaMHSLSD}.
For $J \subset \overline{I}=\ski \setminus I$,
we have
\begin{equation}
\image(\varphi^I(f))
\subset
M_X(D_{\overline{J}})_{\bQ}\gp
\end{equation}
because $I \subset \overline{J}=\ski \setminus J$.
Then an increasing filtration
$W(J)$ on $\kos(\varphi^I(f);\infty;1)$
is defined as $W(M_X(D_{\overline{J}})\gp_{\bQ})$
in \cite[Definition 1.8]{FujisawaMHSLSD}.
For the case of $J=\{i\}$ with $i \notin I$,
we write $W(i)$ instead of $W(\{i\})$.
\end{para}

\begin{defn}
For $I \subset I' \subset \ski$,
we set
\begin{equation*}
A_{I|I'}(f)^{p,q}
=\kos(\varphi^{I' \setminus I}(f);\infty;1)^{p+|q|+|I|}
\left/\sum_{i \in I}W(i)_{q_i} \right.
\end{equation*}
for $p \in \bnnZ, q \in \bnnZ^I$.
We set
\begin{align*}
&d_0=(-1)^{|I|}d:
A_{I|I'}(f)^{p,q} \longrightarrow A_{I|I'}(f)^{p+1,q} \\
&d_i=(-1)^{|I|}t_i\wedge:
A_{I|I'}(f)^{p,q} \longrightarrow A_{I|I'}(f)^{p,q+e_i}
\qquad \text{$i \in I$,}
\end{align*}
where $d$ on the right hand side of the first equality
denotes the differential of the complex
$\kos(\varphi^{I' \setminus I}(f);\infty;1)$
and where $t_i \wedge$
denotes the morphism of complexes
\begin{equation*}
\kos(\varphi^{I' \setminus I}(f);\infty;1)
\longrightarrow
\kos(\varphi^{I' \setminus I}(f);\infty;1)[1]
\end{equation*}
induced by the wedge product $t_i \wedge$
as in \cite[(1.11)]{FujisawaMHSLSD}.
By setting
\begin{equation}
sA_{I|I'}(f)^n=
\bigoplus_{p+|q|=n}A_{I|I'}(f)^{p,q}
\end{equation}
with the differential
$d=d_0+\sum_{i \in I}d_i$,
we obtain a complex $sA_{I|I'}(f)$.
By definition,
$sA_{\emptyset|I'}(f)$ is nothing but $\kos(\varphi^{I'}(f);\infty;1)$.
We write $sA_I(f)$
for $sA_{I|I}(f)$.
Moreover, $sA(f)$ denotes $sA_{\ski}(f)$.
We use $sA_{I|}(f)$ and $sA_{|I'}(f)$
instead of $sA_{I|\ski}(f)$ and $sA_{\emptyset|I'}(f)$ respectively.

An increasing filtration
$L(K)$ on $A_{I|I'}(f)^{p,q}$ for $K \subset I$
is defined by
\begin{equation*}
L(K)_mA_{I|I'}(f)^{p,q}
=W(K)_{m+2|q_K|+|K|}\kos(\varphi^{I' \setminus I};\infty;1)^{p+|q|+|I|}
\left/\sum_{i \in I}W(i)_{q_i} \right.
\end{equation*}
for every $m$.
It is easy to check that this defines an increasing filtration
$L(K)$ on $sA_{I|I'}(f)$.
We use the notation $L=L(\ski)$ on $sA(f)$.
\end{defn}

\begin{rmk}
We can easily see the inclusion
\begin{equation*}
\supp(A_{I|I'}(f)^{p,q}) \subset X[I]
\end{equation*}
for all $p,q$.
Therefore the complex
$sA_{I|I'}(f)$ is considered
as a complex of $\bQ$-sheaves on $X[I]$
as in \ref{para:2}.
\end{rmk}

\begin{lem}
\label{Lemma of grLsA}
We have an isomorphism of complexes
\begin{equation}
\label{equality for grLsA:eq}
\gr_m^{L(K)}sA_{I|I'}(f)
\simeq
\bigoplus_{q \in \bnnZ^K}
\bigoplus_{\Gamma \in S_{m+2|q|+|K|}^{\ge q+e_K}(\Lambda_K)}
\iota_{\Gamma}^{-1}
sA_{(I \setminus K)|(I' \setminus K)}(\pi_Kf)[-m-2|q|]
\end{equation}
for every integer $m$,
under which
the filtration
$L(J)\gr_m^{L(K)}sA_{I|I'}(f)$
coincides with the direct sum of
\begin{equation}
L(J \setminus K \cap J)
[\,|\Gamma \cap \Lambda_{K \cap J}|-2|q_{K \cap J}|-|K \cap J|\,]\,
\iota_{\Gamma}^{-1}sA_{(I \setminus K)|(I' \setminus K)}(\pi_Kf)[-m-2|q|]
\end{equation}
over the same index set
as in \eqref{equality for grLsA:eq}.
\end{lem}
\begin{proof}
Easy from Proposition 1.10 in \cite{FujisawaMHSLSD}.
Here we note that we need the total order on $\Lambda$
in \ref{setting for a semistable morphism}
as in Lemma \ref{lemma on grL(K)sB}.
\end{proof}

\begin{para}
Under the situation
in Lemma \ref{Lemma of grLsA},
we have the commutative diagram of the canonical morphisms
\begin{equation*}
\begin{CD}
\iota_{\Gamma}^{-1}(\bQ^{I' \setminus I} \oplus \cO_X)
@>{\iota_{\Gamma}^{-1}\varphi^{I' \setminus I}(\pi_K f)}>>
\iota_{\Gamma}^{-1}M_X(D_{\overline{K}})\gp_{\bQ} \\
@VVV @VVV \\
\bQ^{I' \setminus I} \oplus \cO_{X[\Gamma]}
@>>{\varphi^{I' \setminus I}(f_{\Gamma})}>
M_{X[\Gamma]}(D_{\overline{K}} \cap X[\Gamma])\gp_{\bQ}
\end{CD}
\end{equation*}
for $\Gamma \in S_{m+2|q|+|K|}^{\ge q+e_K}(\Lambda_K)$,
which induces morphisms of complexes
\begin{equation}
\label{the canonical morphism of kos for closed immersion:eq}
\iota_{\Gamma}^{-1}\kos(\varphi^{I' \setminus I}(\pi_K f);\infty;1)
\longrightarrow
\kos(\varphi^{I' \setminus I}(f_{\Gamma});\infty;1)
\end{equation}
and
\begin{equation}
\label{the canonical morphism of A for closed immersion:eq}
\iota_{\Gamma}^{-1}sA_{(I \setminus K)|(I' \setminus K)}(\pi_K f)
\longrightarrow
sA_{(I \setminus K)|(I' \setminus K)}(f_{\Gamma})
\end{equation}
for $K \subset I \subset I' \subset \ski$.
\end{para}

\begin{lem}
\label{lemma for closed immersion}
The morphism
\eqref{the canonical morphism of A for closed immersion:eq}
is a filtered quasi-isomorphism
with respect to $L(J)$
for $J \subset I \setminus K$.
\end{lem}
\begin{proof}
Because the problem is of local nature,
Lemma \ref{local description of semistable morphism}
enables us to assume the following:
\begin{itemize}
\item
$X=X[\Gamma] \times \Delta^{\Gamma}$
where $\Delta^{\Gamma}$ denotes the polydisc
with the coordinate functions $(x_{\lambda})_{\lambda \in \Gamma}$.
\item
$\pi_K f$ coincides with the composite of the projection
$X=X[\Gamma] \times \Delta^{\Gamma} \longrightarrow X[\Gamma]$
and $f_{\Gamma}:X[\Gamma] \longrightarrow Y[K]$
via the identification $Y[K]=Y_K$.
\end{itemize}
The morphism \eqref{the canonical morphism of kos for closed immersion:eq}
is a quasi-isomorphism
by Lemma (1.4) in \cite{SteenbrinkLE}
because
\begin{equation}
\kernel(\iota_{\Gamma}^{-1}\varphi^{I' \setminus I}(\pi_K f))
=\kernel(\varphi^{I' \setminus I}(f_{\Gamma}))
\end{equation}
and
\begin{equation}
\cok(\iota_{\Gamma}^{-1}\varphi^{I' \setminus I}(\pi_K f))
=\cok(\varphi^{I' \setminus I}(f_{\Gamma})) .
\end{equation}
Therefore the morphism
\eqref{the canonical morphism of A for closed immersion:eq}
is a quasi-isomorphism for the case of $K=I$.
Then the morphism
\eqref{the canonical morphism of A for closed immersion:eq}
is a filtered quasi-isomorphism with respect to $L(I \setminus K)$
by Lemma \ref{Lemma of grLsA}.
In particular, the morphism
\eqref{the canonical morphism of A for closed immersion:eq}
is a quasi-isomorphism.
Therefore we obtain the conclusion by Lemma \ref{Lemma of grLsA} again.
\end{proof}

\begin{para}
Let $g: S' \longrightarrow S$ be a morphism of complex manifolds.
We set $f': X' \longrightarrow Y'=\Delta^k \times S'$ etc.
as in \ref{notation for a base change for S}.
Then we have the commutative diagram
\begin{equation}
\begin{CD}
g'^{-1}(\bQ^{I' \setminus I} \oplus \cO_X)
@>{g'^{-1}\varphi^{I' \setminus I}(f)}>>
g'^{-1}M_X(D)\gp_{\bQ} \\
@VVV @VVV \\
\bQ^{I' \setminus I} \oplus \cO_{X'}
@>>{\varphi^{I' \setminus I}(f')}>
M_{X'}(D)\gp_{\bQ}
\end{CD}
\end{equation}
which induces a morphism of complexes
\begin{equation}
g'^{-1}\kos(\varphi^{I' \setminus I}(f);\infty;1)
\longrightarrow
\kos(\varphi^{I' \setminus I}(f');\infty;1)
\end{equation}
for $I \subset I' \subset \ski$.
Then we obtain the morphism of complexes
\begin{equation}
\label{base change morphism for sA:eq}
g'^{-1}sA_{I|I'}(f) \longrightarrow sA_{I|I'}(f')
\end{equation}
which preserves the filtration $L(K)$ for $K \subset I$.
\end{para}

\begin{lem}
\label{lemma for base change isomorphism for sA}
The morphism \eqref{base change morphism for sA:eq}
is a filtered quasi-isomorphism
with respect to the filtration $L(K)$ for $K \subset I$.
\end{lem}
\begin{proof}
Because we have
\begin{equation}
\kernel(g'^{-1}\varphi^{I' \setminus I}(f))
=\kernel(\varphi^{I' \setminus I}(f'))
\end{equation}
and
\begin{equation}
\cok(g'^{-1}\varphi^{I' \setminus I}(f))
=\cok(\varphi^{I' \setminus I}(f')) ,
\end{equation}
we can proceed the same way as in Lemma \ref{lemma for closed immersion}.
\end{proof}

\begin{defn}
For $I \subset \ski$,
the projection
\begin{equation*}
\bQ^I \oplus \cO_X \longrightarrow \cO_X
\end{equation*}
is denoted by $\tau^I(f)$ for a while.
Then a morphism
\begin{equation*}
\psi^I(f): \kos(\varphi^I(f);n)^p \longrightarrow \omega^p_X
\end{equation*}
is defined by setting
\begin{equation*}
\psi^I(f)(x_1^{[n_1]}x_2^{[n_2]} \cdots x_k^{[n_k]} \otimes y)
=\frac{(2\pi\sqrt{-1})^{-p}}{n_1!n_2! \cdots n_k!}
\tau^I(x_1)^{n_1}\tau^I(x_2)^{n_2} \cdots \tau^I(x_k)^{n_k}
(\bigwedge^p\dlog)(y)
\end{equation*}
for $x_1, x_2, \dots, x_k \in \bQ^I \oplus \cO_X$, $y \in \bigwedge^pM_X(D)\gp$
and $n_1, n_2, \dots, n_k \in \bnnZ$ with $n_1+n_2+\dots+n_k=n-p$
as in \cite[(2.4)]{FujisawaMHSLSD}.
It is easy to check that these morphisms for all $n$
induce a morphism
\begin{equation}
\psi^I(f): \kos(\varphi^I(f);\infty;1)^p \longrightarrow \omega^p_X
\end{equation}
denoted by the same letter $\psi^I(f)$.
We simply denote $\psi^{\emptyset}(f)$ by $\psi(f)$.
Note that the morphism $\psi^I(f)$
does not define a morphism of complexes
except for the case of $I=\emptyset$.
\end{defn}

\begin{lem}
\label{lemma for t^ and dlog t ^}
The diagram
\begin{equation*}
\begin{CD}
\kos(\varphi^I(f);\infty;1)^p @>{\psi^I(f)}>> \omega^p_X \\
@V{t \wedge}VV @VV{\dlog t \wedge}V \\
\kos(\varphi^I(f);\infty;1)^{p+1}
@>>{(2\pi\sqrt{-1})\psi^I(f)}>
\omega^{p+1}_X
\end{CD}
\end{equation*}
is commutative
for any global section $t$ of $M_X(D)$.
\end{lem}
\begin{proof}
Easy by definition.
\end{proof}

\begin{defn}
\label{definition of alpha}
For $I \subset I' \subset \ski$
and for $p \in \bnnZ, q \in \bnnZ^I$,
the composite of the morphism
\begin{equation}
\label{morphism psi^{I' setminus I}:eq}
(2\pi\sqrt{-1})^{|q|+|I|}\psi^{I' \setminus I}(f):
\kos(\varphi^{I' \setminus I}(f);\infty;1)^{p+|q|+|I|}
\longrightarrow
\omega_X^{p+|q|+|I|}
\end{equation}
and the projection
\begin{equation*}
\omega_X^{p+|q|+|I|}
\longrightarrow
\omega_{X/Y_I}^{p+|q|+|I|}
\end{equation*}
induces a morphism
\begin{equation}
\label{morphism from AII' to BI:eq}
A_{I|I'}(f)^{p,q} \longrightarrow B_I(f)^{p,q}
\end{equation}
because we have
$\psi^{I' \setminus I}(f)(W(i)_m) \subset W(D_i)_m$ for all $m$.
It is easy to check that
the morphism \eqref{morphism from AII' to BI:eq}
is compatible with
$d_0$.
Moreover Lemma \ref{lemma for t^ and dlog t ^}
implies that the morphism \eqref{morphism from AII' to BI:eq}
is compatible with $d_i$ for any $i \in I$.
Thus a morphism of complexes
\begin{equation}
\label{definition of alpha:eq}
\alpha_{I|I'}(f): sA_{I|I'}(f) \longrightarrow sB_I(f)
\end{equation}
is obtained.
Clearly, $\alpha_{I|I'}(f)$
preserves the filtration $L(K)$ for $K \subset I$.
We use $\alpha_I(f)=\alpha_{I|I}(f)$,
$\alpha_{I|}(f)=\alpha_{I|\ski}(f)$,
$\alpha_{|I'}(f)=\alpha_{\emptyset|I'}(f)$,
$\alpha_0(f)=\alpha_{\emptyset|\ski}(f)$
and $\alpha(f)=\alpha_{\ski}(f)=\alpha_{\ski|\ski}(f)$ for short.
\end{defn}

\begin{lem}
\label{lemma for alpha0}
The morphism $\alpha_0(f)$ induces a quasi-isomorphism
\begin{equation*}
\cO_{Y,f(x)} \otimes_{\bQ} \kos(\varphi(f);\infty;1)_x
\longrightarrow
\omega_{X/Y,x}
\end{equation*}
for any $x \in X_0$.
\end{lem}
\begin{proof}
Since the question is of local nature,
we may assume that $f$ is a morphism of the form
\eqref{local model of semistable morphism:eq}.
By using Lemma (1.4) in \cite{SteenbrinkLE},
we can easily reduce the problem
to the case of $k=1$.
Thus we obtain the conclusion
by Lemma (1.4) in \cite{SteenbrinkLE}
and by Proposition (1.13) in \cite{SteenbrinkLHS}.
\end{proof}

\begin{prop}
\label{proposition for alpha being fqis}
The morphism
$\alpha_{I|I'}(f)$
induces a filtered quasi-isomorphism
\begin{equation*}
(f_I^{-1}\cO_{Y[I]}
\otimes
sA_{I|I'}(f))\bigl|_{X[I']^{\ast}}
\longrightarrow
sB_I(f)\bigl|_{X[I']^{\ast}}
\end{equation*}
with respect to $L(K)$
for $K \subset I \subset I' \subset \ski$.
\end{prop}
\begin{proof}
We may assume $I'=\ski$,
because we consider the situation over $X[I']^{\ast}$.
Therefore it is sufficient to prove
that the morphism $\alpha_{I|}(f)$
induces a filtered quasi-isomorphism
\begin{equation*}
(f_I^{-1}\cO_{Y[I]} \otimes sA_{I|}(f))\bigl|_{X_0}
\longrightarrow
sB_I(f)\bigl|_{X_0}
\end{equation*}
with respect to $L(K)$ for $K \subset I$.

By Lemma \ref{lemma on grL(K)sB},
Lemma \ref{Lemma of grLsA} and
Lemma \ref{lemma for closed immersion}
we obtain a commutative diagram
\begin{equation}
\begin{CD}
\gr_m^{L(K)}sA_{I|}(f)
@>{\gr_m^{L(K)}\alpha_{I|}(f)}>>
\gr_m^{L(K)}sB_I(f) \\
@A{\simeq}AA @| \\
\bigoplus
\iota_{\Gamma}^{-1}
sA_{(I \setminus K)|}(\pi_Kf)[-m-2|q|]
@.
\gr_m^{L(K)}sB_I(f) \\
@V{\simeq}VV @AA{\simeq}A \\
\bigoplus
sA_{(I \setminus K)|}(f_{\Gamma})[-m-2|q|]
@>>>
\bigoplus
sB_{I \setminus K}(f_{\Gamma})[-m-2|q|],
\end{CD}\label{commutative diagram for grA and grB:eq}
\end{equation}
where the direct sums are taken over the index set
\begin{equation}
\label{index set for grL(K):eq}
\{(q, \Gamma); q \in \bnnZ^K, \Gamma \in S_{m+2|q|+|K|}^{\ge q+e_K}(\Lambda_K)\}
\end{equation}
as in \eqref{grL(K)sB:eq},
and where the top horizontal arrow
is the morphism $\gr_m^{L(K)}\alpha_{I|}(f)$
and the bottom horizontal arrow is the direct sum
of the morphisms
\begin{equation*}
(2\pi\sqrt{-1})^{-m-|q|}
\alpha_{(I \setminus K)|}(f_{\Gamma})[-m-2|q|]
\end{equation*}
over the index set \eqref{index set for grL(K):eq}.
Therefore Lemma \ref{lemma for alpha0} implies that
$\alpha_{I|}(f)$ induces a filtered quasi-isomorphism
with respect to $L(I)$.
In particular, $\alpha_{I|}(f)$ induces a quasi-isomorphism.
Then we can conclude that
$\alpha_{I|}(f)$ induces a filtered quasi-isomorphism
with respect to $L(K)$ for $K \subset I$
by the commutative diagram
\eqref{commutative diagram for grA and grB:eq} again.
\end{proof}

\begin{cor}
\label{corollary on grL's for sA and sB}
The morphism $\alpha_{I|I'}(f)$ induces
a bifiltered quasi-isomorphism
\begin{equation}
(f_I^{-1}\cO_{Y[I]}
\otimes
sA_{I|I'}(f),L(J), L(K))\bigl|_{X[I']^{\ast}}
\overset{\simeq}{\longrightarrow}
(sB_I(f), L(J), L(K))\bigl|_{X[I']^{\ast}}
\end{equation}
for $J, K \subset I \subset \ski$.
\end{cor}
\begin{proof}
By the commutative diagram \eqref{commutative diagram for grA and grB:eq}
we can easily obtain the conclusion.
\end{proof}

\begin{defn}
From the morphism
$\alpha_{I|I'}(f): sA_{I|I'}(f) \longrightarrow sB_I(f)$,
we obtain morphisms
\begin{equation*}
R(f_I)_{\ast}sA_{I|I'}(f)
\longrightarrow
R(f_I)_{\ast}sB_I(f)
\end{equation*}
and
\begin{equation*}
\cO_{Y[I]} \otimes R(f_I)_{\ast}sA_{I|I'}(f)
\longrightarrow
R(f_I)_{\ast}sB_I(f)
\end{equation*}
which we denote by the same letter $\alpha_{I|I'}(f)$
if there is no danger of confusion.
These morphisms preserves the filtration $L(K)$
for $K \subset I$.
\end{defn}

\begin{lem}
\label{lemma on base change for sB}
Under the situation in $\ref{notation for a base change for S}$,
the morphism \eqref{base change morphism for sB:eq}
induces a quasi-isomorphism
\begin{equation}
f'^{-1}\cO_{Y'[I]}
\otimes
g'^{-1}\gr_m^{L(K)}sB_I(f)
\longrightarrow
\gr_m^{L(K)}sB_I(f')
\end{equation}
for all $m \in \bZ$ and for all $K \subset I$.
\end{lem}
\begin{proof}
For any $I' \supset I$,
we have the commutative diagram
\begin{equation}
\small
\begin{CD}
(f_I'^{-1}\cO_{Y'[I]}
\otimes
g'^{-1}\gr_m^{L(K)}sA_{I|I'}(f))|_{X'[I']^{\ast}}
@>>>
(f_I'^{-1}\cO_{Y'[I]}
\otimes
\gr_m^{L(K)}sA_{I|I'}(f'))|_{X'[I']^{\ast}} \\
@VVV @VVV \\
(f_I'^{-1}\cO_{Y'[I]}
\otimes
g'^{-1}\gr_m^{L(K)}sB_I(f))|_{X'[I']^{\ast}}
@>>>
\gr_m^{L(K)}sB_I(f')|_{X'[I']^{\ast}}
\end{CD}
\end{equation}
where the top horizontal arrow
and the two vertical arrows are quasi-isomorphisms
by Lemma \ref{lemma for base change isomorphism for sA}
and Proposition \ref{proposition for alpha being fqis}.
Thus we obtain the conclusion.
\end{proof}

Now we compare the canonical connection $d \otimes \id$
on $\cO_{Y[I]} \otimes R^n(f_I)_{\ast}sA_I(f)$
and the morphism
$(-1)^{|I|}\nabla(I)$ on $R^n(f_I)_{\ast}sB_I(f)$
via the morphism $\alpha_I(f)$.

\begin{defn}
For $I \subset \ski$,
the morphism
\begin{equation}
(2\pi\sqrt{-1})^{|q|+|I|}\psi^{\emptyset}(f):
\kos(\e(f)_{\bQ},\infty,1)^{p+|q|+|I|}
\longrightarrow
\omega_X^{p+|q|+|I|}
\end{equation}
induces a morphism
$A_I(f)^{p,q} \longrightarrow \tB_I(f)^{p,q}$
for $p \in \bnnZ, q \in \bnnZ^I$.
We define a morphism
\begin{equation}
f_I^{-1}\omega^r_{Y[I]}
\otimes
A_I(f)^{p,q}
\longrightarrow
\tB_I(f)^{p+r,q}
\end{equation}
by sending $\omega \otimes \eta$
to $(-1)^{r|I|}\omega \wedge \alpha_I(f)(\eta)$
for $p \in \bnnZ, q \in \bnnZ^I, r \in \bnnZ$
as in the proof of Proposition \ref{wedge product for omegaY and B}.
It is easy to check that these morphisms
induce a morphism of complexes
\begin{equation}
f_I^{-1}\omega_{Y[I]}
\otimes
sA_I(f)
\longrightarrow
s\tB_I(f)
\end{equation}
which is denoted by $\widetilde{\alpha}_I(f)$.
A finite decreasing filtration $G$
on $f_I^{-1}\omega_{Y[I]} \otimes sA_I(f)$
is induced from the stupid filtration on $\omega_{Y[I]}$.
Then $\widetilde{\alpha}_I(f)$ preserves the filtrations $G$
on the both sides.
\end{defn}

Then the following lemma is an easy consequence:

\begin{lem}
\label{remark on the canonical and Gauss-Manin}
The diagram
\begin{equation*}
\begin{CD}
\omega^p_{Y[I]} \otimes R^n(f_I)_{\ast}sA_I(f)
@>{\id \otimes \alpha_I(f)}>>
\omega^p_{Y[I]} \otimes R^n(f_I)_{\ast}sB_I(f) \\
@V{d \otimes \id}VV @VV{(-1)^{|I|}\nabla(I)}V \\
\omega^{p+1}_{Y[I]} \otimes R^n(f_I)_{\ast}sA_I(f)
@>>{\id \otimes \alpha_I(f)}>
\omega^{p+1}_{Y[I]} \otimes R^n(f_I)_{\ast}sB_I(f)
\end{CD}
\end{equation*}
is commutative for all $p$.
\end{lem}

\begin{lem}
\label{lemma for alphaI to BI}
The restriction
\begin{equation}
\widetilde{\alpha}_I(f)\bigl|_{X[I]^{\ast}}:
(f_I^{-1}\omega_{Y[I]}
\otimes
sA_I(f))\bigl|_{X[I]^{\ast}}
\longrightarrow
s\tB_I(f)\bigl|_{X[I]^{\ast}}
\end{equation}
of the morphism
$\widetilde{\alpha}_I(f)$ on $X[I]^{\ast}$
is a filtered quasi-isomorphism
with respect to the filtrations $G$ on the both sides.
\end{lem}
\begin{proof}
We have isomorphisms of complexes
\begin{equation}
\begin{split}
&\gr_G^p(f_I^{-1}\omega_{Y[I]}
\otimes
sA_I(f))
\simeq
f_I^{-1}\omega^p_{Y[I]}
\otimes
sA_I(f)[-p] \\
&\gr_G^ps\tB_I(f)
\simeq
f_I^{-1}\omega^p_{Y[I]}
\otimes
sB_I(f)[-p]
\end{split}
\end{equation}
for every $p$.
By definition,
$\gr_G^p\widetilde{\alpha}_I(f)$
is identified with
$(-1)^{p|I|}\id \otimes \alpha_I(f)[-p]$
under these identifications.
Thus we obtain the conclusion
by Proposition \ref{proposition for alpha being fqis}.
\end{proof}

\section{Proper case}
\label{section for the proper case}

From now on,
we assume that the semistable morphism
$f:X \longrightarrow Y=\Delta^k \times S$
is proper.

\begin{lem}
\label{lemma for sA and sB}
For $K, J \subset I \subset I' \subset \ski$,
the morphism $\alpha_{I|I'}(f)$ in \eqref{definition of alpha:eq}
induces an isomorphism
\begin{equation}
(\cO_{Y[I]}
\otimes
R(f_I)_{\ast}sA_{I|I'}(f), L(J), L(K))
\bigl|_{Y[I']^{\ast}}
\overset{\simeq}{\longrightarrow}
(R(f_I)_{\ast}sB_I(f), L(J), L(K))\bigl|_{Y[I']^{\ast}}
\end{equation}
in the bifiltered derived category.
\end{lem}
\begin{proof}
It suffices to prove
that the morphism $\alpha_{I|I'}(f)$ induces
an isomorphism
\begin{equation}
\cO_{Y[I]}
\otimes
\gr_l^{L(J)}\gr_m^{L(K)}
R(f_I)_{\ast}sA_{I|I'}(f)\bigl|_{Y[I']^{\ast}}
\overset{\simeq}{\longrightarrow}
\gr_l^{L(J)}\gr_m^{L(K)}
R(f_I)_{\ast}sB_I(f)\bigl|_{Y[I']^{\ast}}
\end{equation}
in the derived category for all $l,m$.
Since $f$ is proper,
we obtain the conclusion
from Corollary \ref{corollary on grL's for sA and sB},
by applying the projection formula
(cf. \cite[Proposition 2.6.6]{Kashiwara-SchapiraSM}).
\end{proof}

\begin{lem}
\label{lemma on theta for direct image complexes}
For $K, J \subset I \subset \ski$,
the morphism $\theta_{J|I}(f)$ in \eqref{the morphism thetaJ|I:eq}
induces an isomorphism
%\begin{equation}
%(\cO_{Y[I]} \otimes^L_{\cO_{Y[J]}} R(f_J)_{\ast}sB_J(f), F, L(K)) \\
%\overset{\simeq}{\longrightarrow}
%(R(f_I)_{\ast}sB_I(f), F, L(K))
%\end{equation}
%in the bifiltered derived category,
%where $\otimes^L_{\cO_{Y[J]}}$ stands for
%the bifiltered derived tensor product over $\cO_{Y[J]}$.
\begin{equation}
\cO_{Y[I]}
\otimes^L
\iota_{J|I}^{-1}
\gr_F^p\gr_m^{L(K)}
R(f_J)_{\ast}sB_J(f)
\overset{\simeq}{\longrightarrow}
\gr_F^p\gr_m^{L(K)}
R(f_I)_{\ast}sB_I(f)
\end{equation}
in the derived category,
where $\otimes^L$
stands for the derived tensor product
over $\cO_{Y[J]}$.
\end{lem}
\begin{proof}
By Lemma \ref{flatness lemma}
and Lemma \ref{lemma on grLsBJ to grLsBI},
we have the isomorphisms
\begin{equation}
\begin{split}
f_I^{-1}\cO_{Y[I]}
\otimes^L
\iota_{J|}^{-1}\gr_F^p\gr_m^{L(K)}sB_J(f)
&\,\simeq\,
f_I^{-1}\cO_{Y[I]}
\otimes
\iota_{J|}^{-1}\gr_F^p\gr_m^{L(K)}sB_J(f) \\
&\overset{\simeq}{\longrightarrow}
\gr_F^p\gr_m^{L(K)}sB_I(f)
\end{split}
\end{equation}
for all $m, p$
in the derived category.
Then we obtain the conclusion by
the projection formula and the proper base change theorem
(cf. \cite[Proposition 2.6.7]{Kashiwara-SchapiraSM}).
\end{proof}

\begin{lem}
\label{filtered identification between ss of A and B}
For $K, J \subset I \subset I' \subset \ski$,
the morphism $\alpha_{I|I'}(f)$ induces an isomorphism
\begin{equation}
\label{rational structure of ss:eq}
\begin{split}
(\cO_{Y[I]}
\otimes
E_r^{a,b}(R(f_I)_{\ast}sA_{I|I'}&(f),L(K))
\bigl|_{Y[I']^{\ast}} \\
&\overset{\simeq}{\longrightarrow}
E_r^{a,b}(R(f_I)_{\ast}sB_I(f),L(K))
\bigl|_{Y[I']^{\ast}}
\end{split}
\end{equation}
for all $a,b,r$,
under which we have
\begin{equation}
\begin{split}
(\cO_{Y[I]}
\otimes
(L(J)_{\rec})_mE_r^{a,b}(R(f_I)_{\ast}&sA_{I|I'}(f)),L(K))
\bigl|_{Y[I']^{\ast}}\\
&\simeq
(L(J)_{\rec})_mE_r^{a,b}(R(f_I)_{\ast}sB_I(f),L(K))
\bigl|_{Y[I']^{\ast}}
\end{split}
\end{equation}
for every $m$.
In particular,
we have
\begin{equation}
\begin{split}
(\cO_{Y[I]}
\otimes
\gr_{m_l}^{L(J_l)_{\rec}}
\cdots
&\gr_{m_2}^{L(J_2)_{\rec}}
\gr_{m_1}^{L(J_1)_{\rec}}
E_r^{a,b}(R(f_I)_{\ast}sA_{I|I'}(f)),L(K))
\bigl|_{Y[I']^{\ast}} \\
&\simeq
\gr_{m_l}^{L(J_l)_{\rec}}
\cdots
\gr_{m_2}^{L(J_2)_{\rec}}
\gr_{m_1}^{L(J_1)_{\rec}}
E_r^{a,b}(R(f_I)_{\ast}sB_I(f),L(K))
\bigl|_{Y[I']^{\ast}}
\end{split}
\end{equation}
for $K, J_1,J_2, \dots, J_l \subset I$
and for all $a,b,m_1,m_2, \dots, m_l,r$.
\end{lem}
\begin{proof}
Easy from Lemma \ref{lemma for sA and sB}
and Remark \ref{rmk:2}.
\end{proof}

\begin{prop}
For $J \subset I \subset I' \subset \ski$,
the morphism $\alpha_{I|I'}(f)$ induces an isomorphism
\begin{equation}
\label{isom form sAI|I' to sBI:eq}
(\cO_{Y[I]}
\otimes
R^n(f_I)_{\ast}sA_{I|I'}(f))\bigl|_{Y[I']^{\ast}}
\simeq
R^n(f_I)_{\ast}sB_I(f)\bigl|_{Y[I']^{\ast}}
\end{equation}
for all $n$,
under which we have
\begin{equation}
(\cO_{Y[I]}
\otimes
L(J)_mR^n(f_I)_{\ast}sA_{I|I'}(f))\bigl|_{Y[I']^{\ast}}
\simeq
L(J)_mR^n(f_I)_{\ast}sB_I(f)\bigl|_{Y[I']^{\ast}}
\end{equation}
for all $m$.
In particular,
the morphism
$\alpha_{I|I'}(f)$
induces an isomorphism
\begin{equation}
\begin{split}
(\cO_{Y[I]}
\otimes
\gr_{m_l}^{L(J_l)}
&\cdots
\gr_{m_2}^{L(J_2)}
\gr_{m_1}^{L(J_1)}
R^n(f_I)_{\ast}sA_{I|I'}(f))\bigl|_{Y[I']^{\ast}} \\
&\simeq
\gr_{m_l}^{L(J_l)}
\cdots
\gr_{m_2}^{L(J_2)}
\gr_{m_1}^{L(J_1)}
R^n(f_I)_{\ast}sB_I(f)\bigl|_{Y[I']^{\ast}}
\end{split}
\end{equation}
for $J_1, J_2, \dots, J_l \subset I$
and for all $m_1,m_2, \dots, m_l, n$.
\end{prop}
\begin{proof}
Take $K=\emptyset$
in Lemma \ref{filtered identification between ss of A and B}.
\end{proof}

\begin{prop}
\label{local freeness of the spectral sequence}
For $K, J_1, J_2, \dots, J_l \subset I \subset \ski$,
the coherent $\cO_{Y[I]}$-module
\begin{equation}
\gr_{m_l}^{L(J_l)_{\rec}}
\cdots
\gr_{m_2}^{L(J_2)_{\rec}}
\gr_{m_1}^{L(J_1)_{\rec}}
E_r^{a,b}(R(f_I)_{\ast}sB_I(f),L(K))
\end{equation}
is locally free of finite rank
for all $a,b,m_1,m_2, \dots, m_l,r$.
In particular,
\begin{equation}
E_r^{a,b}(R(f_I)_{\ast}sB_I(f),L(K))
\end{equation}
is a locally free $\cO_{Y[I]}$-module of finite rank
for $K \subset I$ and for all $a,b,r$.
\end{prop}
\begin{proof}
For any $y \in Y[I]$,
there exists $I' \supset I$
such that $y \in Y[I']^{\ast}$.
From Lemma \ref{filtered identification between ss of A and B},
the stalk
\begin{equation}
\gr_{m_l}^{L(J_l)_{\rec}}
\cdots
\gr_{m_2}^{L(J_2)_{\rec}}
\gr_{m_1}^{L(J_1)_{\rec}}
E_r^{a,b}(R(f_I)_{\ast}sB_I(f),L(K))_y
\end{equation}
is a free $\cO_{Y[I],y}$-module
of finite rank.
Thus we obtain the conclusion.
\end{proof}

\begin{prop}
\label{proposition on the local freeness of grLgrLsBI}
For $J_1, J_2, \dots, J_l \subset I \subset \ski$,
the coherent $\cO_{Y[I]}$-module
\begin{equation}
\gr_{m_l}^{L(J_l)}
\cdots
\gr_{m_2}^{L(J_2)}
\gr_{m_1}^{L(J_1)}
R^n(f_I)_{\ast}sB_I(f)
\end{equation}
is locally free of finite rank
for all $m_1, m_2, \dots, m_l,n$.
In particular,
\begin{equation}
R^n(f_I)_{\ast}sB_I(f)
\end{equation}
is locally free of finite rank for all $n$.
\end{prop}
\begin{proof}
Take $K=\emptyset$
in Proposition \ref{local freeness of the spectral sequence}.
\end{proof}

\begin{thm}
\label{theorem on the canonical extension}
If the semistable morphism
$f:X \longrightarrow Y=\Delta^k \times S$ is proper,
then
$R^n(f_I)_{\ast}sB_I(f)$
is the canonical extension of
$(\cO_{Y[I]} \otimes R^n(f_I)_{\ast}sA_I(f))|_{Y[I]^{\ast}}$
via the isomorphisms
\eqref{isom form sAI|I' to sBI:eq}
for $I=I'$.
\end{thm}
\begin{proof}
Considering Proposition \ref{proposition on the local freeness of grLgrLsBI},
Lemma \ref{remark on the canonical and Gauss-Manin}
and Proposition \ref{wedge product for omegaY and B},
it suffices to check
that the residues of the connection $(-1)^{|I|}\nabla(I)$
are nilpotent.
By Corollary \ref{cor on compatibility of nabla(J) and nabla(I)},
these residues
coincide with the residues of $\nabla$
on $R^nf_{\ast}\omega_{X/Y}$.
Therefore
Theorem \ref{thm on the compatibility of N and nu}
shows that
the residues of $(-1)^{|I|}\nabla(I)$
along $E_j \cap Y[I]$ for $j \notin I$
coincide with $-N_{j|I}(f)$,
which is nilpotent by definition.
Thus
$R^n(f_I)_{\ast}sB_I(f)$
equipped with
$(-1)^{|I|}\nabla(I)$
is the canonical extension as desired.
\end{proof}

\begin{rmk}
\label{Usui's result}
For $I=\emptyset$,
the theorem above tells us that
$R^nf_{\ast}\omega_{X/Y}$ is
the canonical extension of its restriction on $Y^{\ast}$.
This has been already claimed
by Usui \cite[p.138]{UsuiRVCSV}.
\end{rmk}

\begin{lem}
\label{lemma for ss being the canonical extension}
For $K \subset I \subset \ski$,
the locally free $\cO_{Y[I]}$-module
\begin{equation}
E_r^{a,b}(R(f_I)_{\ast}sB_I(f),L(K))
\end{equation}
is the canonical extension of
\begin{equation}
\cO_{Y[I]}
\otimes
E_r^{a,b}(R(f_I)_{\ast}sA_I(f),L(K))\bigl|_{Y[I]^{\ast}}
\end{equation}
via the isomorphism \eqref{rational structure of ss:eq}
for $I'=I$.
\end{lem}
\begin{proof}
From the commutative diagram
\eqref{commutative diagram for grA and grB:eq},
Theorem \ref{theorem on the canonical extension}
implies the conclusion for $r=1$.
Then we obtain the conclusion for general $r$
because taking the canonical extension is an exact functor.
\end{proof}

\begin{rmk}
\label{remark on the canonical extension of ss}
For $K, J \subset I \subset \ski$,
the locally free $\cO_{Y[I]}$-module
\begin{equation}
(L(J)_{\rec})_mE_r^{a,b}(R(f_I)_{\ast}sB_I(f),L(K))
\end{equation}
is the canonical extension of its restriction on $Y[I]^{\ast}$.
\end{rmk}

\begin{lem}
\label{lemma for theta and the spectral sequences}
For $K \subset J \subset I \subset \ski$,
the morphism $\iota_{|I}^{-1}\theta_{J|I}(f)$
in \eqref{the morphism thetaJ|I:eq}
induces an isomorphism
\begin{equation}
\label{theta and ss:eq}
\cO_{Y[I]}
\otimes
\iota_{J|I}^{-1}E_r^{a,b}(R(f_J)_{\ast}sB_J(f),L(K))
\overset{\simeq}{\longrightarrow}
E_r^{a,b}(R(f_I)_{\ast}sB_I(f),L(K))
\end{equation}
for all $a,b,r$.
\end{lem}
\begin{proof}
By Lemma \ref{lemma on theta for direct image complexes},
the morphism $\iota_{|I}^{-1}\theta_{J|I}(f)$
induces an isomorphism
\begin{equation}
\cO_{Y[I]}
\otimes^L
\iota_{J|I}^{-1}
\gr_m^{L(K)}
R(f_J)_{\ast}sB_J(f)
\overset{\simeq}{\longrightarrow}
\gr_m^{L(K)}
R(f_I)_{\ast}sB_I(f)
\end{equation}
in the derived category for all $m$.
Then the local freeness of
$E_r^{a,b}(R(f_J)_{\ast}sB_J(f),L(K))$
for all $a,b,r$ implies the conclusion.
\end{proof}

\begin{prop}
\label{compatibility of the restriction and L(K)}
For $K \subset J \subset I \subset \ski$,
the morphism $\theta_{J|I}(f)$
induces isomorphisms
\begin{align}
&\cO_{Y[I]}
\otimes
\iota_{J|I}^{-1}
L(K)_mR^n(f_J)_{\ast}sB_J(f)
\overset{\simeq}{\longrightarrow}
L(K)_mR^n(f_I)_{\ast}sB_I(f)
\label{isom for L(K) induced by theta:eq} \\
&\cO_{Y[I]}
\otimes
\iota_{J|I}^{-1}
\gr_m^{L(K)}R^n(f_J)_{\ast}sB_J(f)
\overset{\simeq}{\longrightarrow}
\gr_m^{L(K)}R^n(f_I)_{\ast}sB_I(f)
\label{isom for grL(K) induced by theta:eq}
\end{align}
for all $m,n$.
In particular,
the morphism $\theta_{J|I}(f)$
induces an isomorphism
\begin{equation}
\cO_{Y[I]}
\otimes
\iota_{J|I}^{-1}
R^n(f_J)_{\ast}sB_J(f)
\overset{\simeq}{\longrightarrow}
R^n(f_I)_{\ast}sB_I(f)
\end{equation}
under which the filtration $L(K)$ for $K \subset J$
is identified on the both sides,
where the filtration $L(K)$ on the left hand side
is defined as in Definition $\ref{defn:1}$.
\end{prop}
\begin{proof}
Taking $r$ sufficiently large
in \eqref{theta and ss:eq},
we obtain the isomorphism
\eqref{isom for grL(K) induced by theta:eq}
for all $m,n$.
The local freeness of
$\gr_m^{L(K)}R^n(f_J)_{\ast}sB_J(f)$
implies \eqref{isom for L(K) induced by theta:eq}.
For the latter,
we note that the canonical morphism
\begin{equation}
\cO_{Y[I]}
\otimes
\iota_{J|I}^{-1}
L(K)_mR^n(f_J)_{\ast}sB_J(f)
\longrightarrow
L(K)_m(
\cO_{Y[I]}
\otimes
\iota_{J|I}^{-1}
R^n(f_J)_{\ast}sB_J(f))
\end{equation}
is an isomorphism,
because $\gr_m^{L(K)}R^n(f_J)_{\ast}sB_J(f)$
is a locally free $\cO_{Y[J]}$-module of finite rank for all $m$.
\end{proof}

\begin{cor}
\label{corollary on grL's and theta}
For $J_1,J_2, \dots, J_l \subset J \subset I \subset \ski$,
the morphism $\theta_{J|I}(f)$
induces an isomorphism
\begin{equation}
\label{the morphism induced by theta on grL's:eq}
\begin{split}
\cO_{Y[I]}
\otimes
\gr_{m_l}^{L(J_l)} \cdots
&\gr_{m_2}^{L(J_2)}\gr_{m_1}^{L(J_1)}
\iota_{J|I}^{-1}R^n(f_J)_{\ast}sB_J(f) \\
&\overset{\simeq}{\longrightarrow}
\gr_{m_l}^{L(J_l)} \cdots
\gr_{m_2}^{L(J_2)}\gr_{m_1}^{L(J_1)}
R^n(f_I)_{\ast}sB_I(f)
\end{split}
\end{equation}
for all $m_1,m_2, \dots, m_l,n$.
\end{cor}
\begin{proof}
From the local freeness of
$\gr_{m_l}^{L(J_l)} \cdots \gr_{m_2}^{L(J_2)}\gr_{m_1}^{L(J_1)}
R^n(f_J)_{\ast}sB_J(f)$
in Proposition \ref{proposition on the local freeness of grLgrLsBI},
we obtain the conclusion
by Lemma \ref{local freeness and the commutattivity of gr}.
\end{proof}

\section{Degeneration of Hodge structures}
\label{section for the Kahler case}

In this and next section,
we always assume that $X$ is a K\"ahler manifold
and that the morphism $f: X \longrightarrow Y=\Delta^k \times S$
is proper and semistable.

\begin{lem}
\label{lemma on VMHS for I}
For $J \subset I \subset \ski$,
the data
\begin{equation}
((R^n(f_I)_{\ast}L(J)_msA_I(f),L(I)[n]),
(R^n(f_I)_{\ast}L(J)_msB_I(f),L(I)[n],F),
\alpha_I(f))\bigl|_{Y[I]^{\ast}}
\end{equation}
is a graded polarizable variation of $\bQ$-mixed Hodge structure
on $Y[I]^{\ast}$ for all $m,n$.
\end{lem}
\begin{proof}
We may assume that $I=\ski$.
Then
\begin{equation}
Y[\ski]^{\ast}=Y[\ski]=\{0\} \times S=S
\end{equation}
by definition.
As in the proof of Proposition \ref{proposition for alpha being fqis},
we have the isomorphism
\begin{equation}
(\gr_{-a}^LL(J)_msB(f),F)
\simeq
\bigoplus
(\Omega_{X[\Gamma]/S}[a-2|q|],F[a-|q|])
\end{equation}
and the quasi-isomorphism
\begin{equation}
\gr_{-a}^LL(J)_msA(f)
\simeq
\bigoplus
\kos(\mathbf{e}(f_{\Gamma})_{\bQ};\infty;1)[a-2|q|]
\end{equation}
by Lemma \ref{lemma on grL(K)sB},
Lemma \ref{Lemma of grLsA} and
Lemma \ref{lemma for closed immersion},
where the direct sums are taken over the same index set
\begin{equation}
\{(q, \Gamma) ; {q \in \bnnZ^k}, 
\Gamma \in S_{-a+2|q|+k}^{\ge q+e}(\Lambda)
\text{ with }
|\Gamma \cap \Lambda_J| \le m+2|q_J|+|J|\} .
\end{equation}
Note that $\kos(\mathbf{e}(f_{\Gamma})_{\bQ};\infty;1)$
is the Koszul complex associated to the exponential map
\begin{equation}
\cO_{X[\Gamma]} \longrightarrow \cO^{\ast}_{X[\Gamma]} \otimes \bQ
\end{equation}
with the base extension to $\bQ$.
Moreover, there exists an quasi-isomorphism
$\bQ_{X[\Gamma]} \longrightarrow \kos(\mathbf{e}(f_{\Gamma})_{\bQ};\infty;1)$
(cf. \cite[Corollary 1.15]{FujisawaMHSLSD}).
Under these identifications
the morphism $\gr_{-a}^L\alpha(f)$
is identified with the direct sum of the morphisms
$(2\pi\sqrt{-1})^{a-|q|}\psi(f_{\Gamma})[a-2|q|]$.
Because
$f_{\Gamma}: X[\Gamma] \longrightarrow S$
is a proper smooth morphism from a K\"ahler manifold $X[\Gamma]$,
we obtain the following:
\begin{mylist}
\itemno
\label{VHS condition:eq}
The data
\begin{equation*}
(R^{a+b}f_{0\ast}\gr_{-a}^LL(J)_msA(f),
(R^{a+b}f_{0\ast}\gr_{-a}^LL(J)_msB(f),F),
R^{a+b}f_{0\ast}\gr_{-a}^L\alpha(f))
\end{equation*}
is a polarizable variation of $\bQ$-Hodge structure of weight $b$
on $S$.
\itemno
\label{strictness condition for E0:eq}
The filtered complex
$(Rf_{0\ast}\gr_{-a}^LL(J)_msB(f),F)$
is strongly strict for all $a$.
\end{mylist}
The morphism
\begin{equation}
d_0:
E_0^{a,b}(Rf_{0\ast}L(J)_msB(f),L)
\longrightarrow
E_0^{a,b+1}(Rf_{0\ast}L(J)_msB(f),L)
\end{equation}
is strictly compatible with respect to the filtration
$F=F_{\rec}=F_{d}=F_{d^{\ast}}$
by \eqref{strictness condition for E0:eq}.
Then we have $F_{\rec}=F_{d}=F_{d^{\ast}}$
on $E_1^{a,b}(Rf_{0\ast}L(J)_msB(f),L)$
and on $E_2^{a,b}(Rf_{0\ast}L(J)_msB(f),L)$
by the lemma on two filtration
\cite[Proposition (7.2.5)]{DeligneIII}.
Moreover
\begin{equation}
(E_1^{a,b}(Rf_{0\ast}L(J)_msA(f),L),
(E_1^{a,b}(Rf_{0\ast}L(J)_msB(f),L),F_{\rec}),
E_1^{a,b}(\alpha(f)))
\end{equation}
is a polarizable variation of $\bQ$-Hodge structure of weight $b$ on $S$
by \eqref{VHS condition:eq}.
Because the morphism
\begin{equation}
d_1:
E_1^{a,b}(Rf_{0\ast}L(J)_msB(f),L)
\longrightarrow
E_1^{a+1,b}(Rf_{0\ast}L(J)_msB(f),L)
\end{equation}
underlies a morphism of variations of $\bQ$-Hodge structure,
the morphism $d_1$ is strictly compatible with the filtration $F_{\rec}$ and
\begin{equation}
(E_2^{a,b}(Rf_{0\ast}L(J)_msA(f),L),
(E_2^{a,b}(Rf_{0\ast}L(J)_msB(f),L),F_{\rec}),
E_2^{a,b}(\alpha(f)))
\end{equation}
is a polarizable variation
of $\bQ$-Hodge structures of weight $b$ on $S$.
Because the morphism
\begin{equation}
d_2:
E_2^{a,b}(Rf_{0\ast}L(J)_msB(f),L)
\longrightarrow
E_2^{a+2,b-1}(Rf_{0\ast}L(J)_msB(f),L)
\end{equation}
underlies a morphism of variations of $\bQ$-Hodge structure again,
we obtain $d_2=0$.
Inductively, we obtain $d_r=0$ for $r \ge 2$.
Then we have
\begin{equation}
\begin{split}
&E_2^{a,b}(Rf_{0\ast}L(J)_msA(f),L)
\simeq
E_{\infty}^{a,b}(Rf_{0\ast}L(J)_msA(f),L) \\
&E_2^{a,b}(Rf_{0\ast}L(J)_msB(f),L)
\simeq
E_{\infty}^{a,b}(Rf_{0\ast}L(J)_msB(f),L)
\end{split}
\end{equation}
and $F=F_{\rec}=F_{d}=F_{d^{\ast}}$
on $E_{\infty}^{a,b}(Rf_{0\ast}L(J)_msB(f),L)$
for all $a,b$.
Thus the data
\begin{equation}
(\gr_{-a}^LR^{a+b}f_{0\ast}L(J)_msA(f),
(\gr_{-a}^LR^{a+b}f_{0\ast}L(J)_msB(f),F), \gr_{-a}^L\alpha_0(f))
\end{equation}
is a polarizable variation of Hodge structures of weight $b$ on $S$.
By Lemma \ref{remark on the canonical and Gauss-Manin}
the canonical connection $d \otimes \id$
on $\cO_{S} \otimes R^qf_{0\ast}sA(f)$ is identified
with the Gauss-Manin connection $(-1)^k\nabla(\ski)$,
which satisfies the Griffiths transversality.
Thus we obtain the conclusion.
\end{proof}

The lemma above implies the following theorem.

\begin{thm}
\label{theorem on VMHS for I}
Let $X$ be a K\"ahler manifold
and $f:X \longrightarrow Y=\Delta^k \times S$
a proper semistable morphism.
For any $I \subset \ski$,
the data
\begin{equation}
\label{VMHS for I:eq}
((R^n(f_I)_{\ast}sA_I(f),L(I)[n]),
(R^n(f_I)_{\ast}sB_I(f),L(I)[n],F), \alpha_I(f))\bigl|_{Y[I]^{\ast}}
\end{equation}
is a graded polarizable variation
of $\bQ$-mixed Hodge structure on $Y[I]^{\ast}$
for every $n \in \bZ$.
Moreover, $L(J)_m$
gives a subvariation of $\bQ$-mixed Hodge structure
for every $J \subset I$ and for every $m \in \bZ$.
\end{thm}

\begin{cor}
\label{MHS on E1-terms}
For $K \subset I \subset \ski$,
the data
\begin{equation}
\begin{split}
((E_1^{a,b}(R(f_I)_{\ast}sA_I(f),&L(K)), L(I \setminus K)[b]), \\
&(E_1^{a,b}(R(f_I)_{\ast}sB_I(f),L(K)),L(I \setminus K)[b],F),
E_1^{a,b}(\alpha_I(f)))
\bigl|_{Y[I]^{\ast}}
\end{split}
\end{equation}
is a variation of $\bQ$-mixed Hodge structure on $Y[I]^{\ast}$
for all $a, b$.
For $J \subset I$,
the filtration $L(J)_m$
gives us its subvariation of $\bQ$-mixed Hodge structure
for all $m$.
\end{cor}
\begin{proof}
By the identifications
\begin{align}
&E_1^{a,b}(R(f_I)_{\ast}sA_I(f),L(K))
\simeq
R^{a+b}(f_I)_{\ast}\gr_{-a}^{L(K)}sA_I(f) \\
&E_1^{a,b}(R(f_I)_{\ast}sB(f),L(K))
\simeq
R^{a+b}(f_I)_{\ast}\gr_{-a}^{L(K)}sB_I(f)
\label{identification of E_1 for sBI:eq}
\end{align}
we obtain the conclusion by Lemma \ref{lemma on grL(K)sB},
Lemma \ref{Lemma of grLsA}, Lemma \ref{lemma for closed immersion},
Proposition \ref{proposition for alpha being fqis}
and Theorem \ref{theorem on VMHS for I}.
\end{proof}

The case of $I=\ski$ in the theorem below
is equivalent to the main result Theorem (4.1)
in \cite{FujisawaDWSS}.
The proof below
is independent of the argument in \cite{FujisawaDWSS}
and much simpler than that of
Theorem (4.1) in \cite{FujisawaDWSS}.

\begin{thm}
\label{theorem for E2-degeneracy}
Let $X$ be a K\"ahler manifold
and $f:X \longrightarrow Y=\Delta^k \times S$
a proper semistable morphism.
For any $K \subset I \subset \ski$,
the spectral sequence
$E_r^{a,b}(R(f_I)_{\ast}sB_I(f),L(K))$
degenerates at $E_2$-terms.
In other words,
we have the canonical isomorphism
\begin{equation}
\label{E2-degeneration for L(K):eq}
\gr_m^{L(K)}R^n(f_I)_{\ast}sB_I(f)
\overset{\simeq}{\longrightarrow}
E_2^{-m,n+m}(R(f_I)_{\ast}sB_I(f),L(K))
\end{equation}
for all $m,n$.
\end{thm}
\begin{proof}
We first treat the case of $I=K$.
As in the proof of Theorem \ref{theorem on VMHS for I},
the morphisms of $E_r$-terms
\begin{equation*}
E_r^{a,b}(R(f_K)_{\ast}sB_K(f),L(K))
\longrightarrow
E_r^{a+r,b-r+1}(R(f_K)_{\ast}sB_K(f),L(K))
\end{equation*}
are the zero morphisms over $Y[K]^{\ast}$ for all $a,b$
and for all $r \ge 2$.
Therefore the morphisms of $E_r$-terms are
the zero morphisms over the whole $Y[K]$
for $r \ge 2$
because $E_r^{a,b}(R(f_K)_{\ast}sB_K(f),L(K))$
are locally free over $Y[K]$
by Proposition \ref{local freeness of the spectral sequence}.
Thus we obtain the desired $E_2$-degeneracy.

For the general case,
we have the isomorphism
\eqref{theta and ss:eq} for $J=K$
\begin{equation}
\cO_{Y[I]}
\otimes
\iota_{K|I}^{-1}E_r^{a,b}(R(f_K)_{\ast}sB_K(f),L(K))
\overset{\simeq}{\longrightarrow}
E_r^{a,b}(R(f_I)_{\ast}sB_I(f),L(K))
\end{equation}
for all $a,b,r$,
which is compatible with the morphisms $d_r$.
Then the $E_2$-degeneracy
of $E_r^{a,b}(R(f_K)_{\ast}sB_K(f),L(K))$
implies that of $E_r^{a,b}(R(f_I)_{\ast}sB_I(f),L(K))$.
\end{proof}

\begin{para}
Let $X$ be a K\"ahler manifold
and $f: X \longrightarrow Y=\Delta^k \times S$
a proper semistable morphism as before.
For an integer $b$ and for $I \subset \ski$,
we have the complex of $\cO_{Y[I]}$-modules
\begin{equation}
(E_1^{\bullet,b}(R(f_I)_{\ast}sB_I(f),L(K)), d_1) ,
\end{equation}
where
\begin{equation}
\label{the morphism of E1-terms:eq}
d_1:
E_1^{a,b}(R(f_I)_{\ast}sB_I(f),L(K))
\longrightarrow
E_1^{a+1,b}(R(f_I)_{\ast}sB_I(f),L(K))
\end{equation}
is the morphism of $E_1$-terms for all $a$.
For any $J \subset I$,
the filtration $L(J)_{\rec}=L(J)_d=L(J)_{d^\ast}$
on $E_1^{a,b}(R(f_I)_{\ast}sB_I(f),L(K))$
is simply denoted by $L(J)$,
which forms an increasing filtration of the complex
$E_1^{\bullet,b}(R(f_I)_{\ast}sB_I(f),L(K))$.
We also have the filtration $F_{\rec}=F_{d}=F_{d^{\ast}}$
on $E_1^{\bullet,b}(R(f_I)_{\ast}sB_I(f),L(K))$,
which is simply denoted by $F$.
\end{para}

\begin{lem}
\label{lemma on the isom induced by theta on E1}
For $K, J_1, J_2, \dots, J_l \subset J \subset I \subset \ski$,
the morphism $\theta_{J|I}(f)$ induces an isomorphism of complexes
\begin{equation}
\label{iso grL...grLtheta:eq}
\begin{split}
\cO_{Y[I]} \otimes
\gr_{m_l}^{L(J_l)}
&\cdots
\gr_{m_2}^{L(J_2)}\gr_{m_1}^{L(J_1)}
\iota_{J|I}^{-1}E_1^{\bullet,b}(R(f_J)_{\ast}sB_J(f),L(K)) \\
&\overset{\simeq}{\longrightarrow}
\gr_{m_l}^{L(J_l)} \cdots
\gr_{m_2}^{L(J_2)}\gr_{m_1}^{L(J_1)}
E_1^{\bullet,b}(R(f_I)_{\ast}sB_I(f),L(K))
\end{split}
\end{equation}
for all $m_1, m_2, \dots, m_l$.
\end{lem}
\begin{proof}
From the identification \eqref{identification of E_1 for sBI:eq},
we easily obtain the conclusion
by Lemma \ref{lemma on grL(K)sB}
and Corollary \ref{corollary on grL's and theta}.
\end{proof}

\begin{lem}
\label{strictness for L(J setminus K) on ss}
For $K, J_1, J_2, \dots, J_l \subset J \subset I \subset \ski$,
the filtration $L(J \setminus K)$ on the complex
\begin{equation}
\gr_{m_l}^{L(J_l)} \cdots
\gr_{m_2}^{L(J_2)}\gr_{m_1}^{L(J_1)}
E_1^{\bullet,b}(R(f_I)_{\ast}sB_I(f),L(K))
\end{equation}
is strongly strict
for all $m_1, m_2, \dots, m_l$.
\end{lem}
\begin{proof}
By Lemma \ref{filtered identification between ss of A and B},
we have the isomorphism
\begin{equation}
\begin{split}
\cO_{Y[I]}
\otimes
\coh^a(&\gr_m^{L(J \setminus K)}
\gr_{m_l}^{L(J_l)} \cdots
\gr_{m_2}^{L(J_2)}\gr_{m_1}^{L(J_1)}
E_1^{\bullet,b}(R(f_I)_{\ast}sA_{I|I'}(f),L(K)))
\bigl|_{Y[I']^{\ast}} \\
&\simeq
\coh^a(\gr_m^{L(J \setminus K)}
\gr_{m_l}^{L(J_l)} \cdots
\gr_{m_2}^{L(J_2)}\gr_{m_1}^{L(J_1)}
E_1^{\bullet,b}(R(f_I)_{\ast}sB_I(f),L(K)))
\bigl|_{Y[I']^{\ast}}
\end{split}
\end{equation}
for all $a, m, m_1, m_2, \dots, m_l$
and for any $I' \supset I$.
Therefore the coherent $\cO_{Y[I]}$-module
\begin{equation}
\coh^a(\gr_m^{L(J \setminus K)}
\gr_{m_l}^{L(J_l)} \cdots
\gr_{m_2}^{L(J_2)}\gr_{m_1}^{L(J_1)}
E_1^{\bullet,b}(R(f_I)_{\ast}sB_I(f),L(K)))
\end{equation}
turns out to be locally free for all $a, m, m_1, m_2, \dots, m_l$
as in Proposition \ref{local freeness of the spectral sequence}.
Hence it suffices to prove the strictness
of $L(J \setminus K)$.

First we treat the case of $J=I$.
On the open subset $Y[I]^{\ast}$,
\begin{equation}
(\gr_{m_l}^{L(J_l)} \cdots
\gr_{m_2}^{L(J_2)}\gr_{m_1}^{L(J_1)}
E_1^{a,b}(R(f_I)_{\ast}sB_I(f),L(K)), L(I \setminus K)[b], F)
\left|_{Y[I]^{\ast}}\right.
\end{equation}
underlies a variation of $\bQ$-mixed Hodge structures for all $a$
and the morphism $d_1$ underlies a morphism of variations
by Corollary \ref{MHS on E1-terms}.
Therefore the filtration $L(I \setminus K)$
on the complex
\begin{equation}
\gr_{m_l}^{L(J_l)} \cdots
\gr_{m_2}^{L(J_2)}\gr_{m_1}^{L(J_1)}
E_1^{\bullet,b}(R(f_I)_{\ast}sB_I(f),L(K))
\left|_{Y[I]^{\ast}}\right.
\end{equation}
is strict,
that is,
the canonical morphism
\begin{equation}
\label{the canonical morphism for strictness of L:eq}
\begin{split}
\coh^a(L(I \setminus K)_m
&\gr_{m_l}^{L(J_l)} \cdots
\gr_{m_2}^{L(J_2)}\gr_{m_1}^{L(J_1)}
E_1^{\bullet,b}(R(f_I)_{\ast}sB_I(f),L(K))) \\
&\longrightarrow
\coh^a(
\gr_{m_l}^{L(J_l)} \cdots
\gr_{m_2}^{L(J_2)}\gr_{m_1}^{L(J_1)}
E_1^{\bullet,b}(R(f_I)_{\ast}sB_I(f),L(K)))
\end{split}
\end{equation}
is injective over $Y[I]^{\ast}$.
Since
\begin{equation}
\gr_{m_l}^{L(J_l)} \cdots
\gr_{m_2}^{L(J_2)}\gr_{m_1}^{L(J_1)}
E_1^{a,b}(R(f_I)_{\ast}sB_I(f),L(K)))
\end{equation}
and
\begin{equation}
L(I \setminus K)_m
\gr_{m_l}^{L(J_l)} \cdots
\gr_{m_2}^{L(J_2)}\gr_{m_1}^{L(J_1)}
E_1^{a,b}(R(f_I)_{\ast}sB_I(f),L(K)))
\end{equation}
are the canonical extensions of their restrictions on $Y[I]^{\ast}$
by Lemma \ref{lemma for ss being the canonical extension}
(see also Remark \ref{remark on the canonical extension of ss})
for all $a$,
the morphism \eqref{the canonical morphism for strictness of L:eq}
is injective over the whole $Y[I]$.
Thus we obtain the strictness of $L(I \setminus K)$ over $Y[I]$.

From the case of $J=I$ above,
the filtration $L(J \setminus K)$ on the complex
\begin{equation}
\gr_{m_l}^{L(J_l)}
\cdots
\gr_{m_2}^{L(J_2)}
\gr_{m_1}^{L(J_1)}
E_1^{\bullet,b}(R(f_J)_{\ast}sB_J(f),L(K)))
\end{equation}
is strongly strict.
On the other hand,
the morphism $\theta_{J|I}(f)$ induces the isomorphism
\begin{equation}
\begin{split}
(\cO_{Y[I]} \otimes^L
&\gr_{m_l}^{L(J_l)} \cdots
\gr_{m_2}^{L(J_2)}\gr_{m_1}^{L(J_1)}
\iota_{J|I}^{-1}E_1^{\bullet,b}(R(f_J)_{\ast}sB_J(f),L(K)), L(J \setminus K)) \\
&\simeq
(\gr_{m_l}^{L(J_l)} \cdots
\gr_{m_2}^{L(J_2)}\gr_{m_1}^{L(J_1)}
E_1^{\bullet,b}(R(f_I)_{\ast}sB_I(f),L(K)), L(J \setminus K))
\end{split}
\end{equation}
in the filtered derived category
by Lemma \ref{lemma on the isom induced by theta on E1}.
Therefore we obtain the conclusion
by Lemma \ref{lemma on strong strictness and filtered tensor product}.
\end{proof}

\begin{lem}
\label{lemma on theta and E2}
For $K \subset J \subset I \subset \ski$,
the morphism $\theta_{J|I}(f)$ induces an isomorphism
\begin{equation}
\begin{split}
\cO_{Y[I]} \otimes
\gr_m^{L(J \setminus K)_{\rec}}
&\iota_{J|I}^{-1}E_2^{a,b}(R(f_J)_{\ast}sB_J(f),L(K)) \\
&\overset{\simeq}{\longrightarrow}
\gr_m^{L(J \setminus K)_{\rec}}
E_2^{a,b}(R(f_I)_{\ast}sB_I(f),L(K))
\end{split}
\end{equation}
for all $a, b, m$.
\end{lem}
\begin{proof}
From Lemma \ref{lemma on the isom induced by theta on E1},
we have the isomorphism
\begin{equation}
\begin{split}
(\cO_{Y[I]}
\otimes
E_1^{a,b}(R(f_J)_{\ast}sB_J(f),&L(K)), L(J \setminus K)) \\
&\simeq
(E_1^{a,b}(R(f_I)_{\ast}sB_I(f),L(K)), L(J \setminus K))
\end{split}
\end{equation}
in the filtered derived category.
Because $L(J \setminus K)$
is strongly strict
on the complex
$E_1^{\bullet,b}(R(f_J)_{\ast}sB_J(f),L(K))$,
we obtain the conclusion
by Lemma \ref{lemma on strong strictness and filtered tensor product}.
\end{proof}

\begin{prop}
\label{compatiblity of L(J) and L(J setminus K) on grLK}
Under the identification
\eqref{E2-degeneration for L(K):eq},
the filtration $L(J)[-m]$ on the left hand side
coincides with the filtration $L(J \setminus K)_{\rec}$ on the right
for all $K \subset J \subset I$.
\end{prop}
\begin{proof}
First we treat the case $J=I$.
It suffices to prove the desired coincidence over $Y[I]^{\ast}$
because 
$\gr_l^{L(I)}\gr_m^{L(K)}R^n(f_I)_{\ast}sB_I(f)$ and 
$\gr_l^{L(I \setminus K)_{\rec}}E_2^{-m,n+m}(R(f_I)_{\ast}sB_I(f),L(K))$
are locally free for all $l$.
Therefore we may assume that $J=I=\ski$.
Thus what we have to prove is
the coincidence of $L[-m]$ and $L(\overline{K})_{\rec}$
under the identification
\begin{equation}
\label{identification for f0:eq}
\gr_m^{L(K)}R^nf_{0\ast}sB(f)
\overset{\simeq}{\longrightarrow}
E_2^{-m,n+m}(Rf_{0\ast}sB(f),L(K))
\end{equation}
in Theorem \ref{theorem for E2-degeneracy},
where $\overline{K}=\ski \setminus K$.
The following argument is similar to that of El Zein
\cite[6.1.12 Th\'eor\`em]{ElZeinbook}.
So we only give a sketch of the proof here.

From the $E_2$-degeneracy in Theorem \ref{theorem for E2-degeneracy},
the identification \eqref{identification for f0:eq}
fits in the commutative diagram
\begin{equation}
\label{commutative diagram from E2-degeneracy:eq}
\begin{CD}
R^nf_{0\ast}L(K)_msB(f)
@>>>
\kernel(d_1) \\
@VVV @VVV \\
\gr_m^{L(K)}R^nf_{0\ast}sB(f)
@>{\simeq}>>
E_2^{-m,n+m}(Rf_{0\ast}sB(f),L(K))
\end{CD}
\end{equation}
where $d_1$ in the right hand side of the top horizontal arrow
is the morphism of $E_1$-terms
\begin{equation}
d_1:
E_1^{-m,n+m}(Rf_{0\ast}sB(f),L(K))
\longrightarrow
E_1^{-m+1,n+m}(Rf_{0\ast}sB(f),L(K))
\end{equation}
and the top horizontal arrow
is induced from the projection
\begin{equation}
\label{projection from LKsB to grLKsB:eq}
L(K)_msB(f) \longrightarrow \gr_m^{L(K)}sB(f)
\end{equation}
by the identification
$R^nf_{0\ast}\gr_m^{L(K)}sB(f) \simeq E_1^{-m,n+m}(Rf_{0\ast}sB(f),L(K))$.
It is easy to see that
the projection \eqref{projection from LKsB to grLKsB:eq}
sends $L_l$ to $L(\overline{K})_{l-m}=L(\overline{K})[m]_l$ for all $l$.
Lemma \ref{lemma on VMHS for I},
Theorem \ref{theorem on VMHS for I}
and Lemma \ref{MHS on E1-terms} imply that
\begin{equation}
\begin{split}
&(R^nf_{0\ast}L(K)_msB(f),L[n],F), \\
&(\gr_m^{L(K)}R^nf_{0\ast}sB(f),L[n],F), \\
&(\kernel(d_1),L(\overline{K})[n+m],F), \\
&(E_2^{-m,n+m}(Rf_{0\ast}sB(f),L(K)),L(\overline{K})_{\rec}[n+m],F_{\rec})
\end{split}
\end{equation}
underlie variations of $\bQ$-mixed Hodge structure
on $Y[\ski]=S$.
Moreover the top horizontal arrow and two vertical arrows
in the commutative diagram
\eqref{commutative diagram from E2-degeneracy:eq}
underlie morphisms of variations of $\bQ$-mixed Hodge structure
by definition.
Thus we can easily see that
the identification \eqref{identification for f0:eq},
the bottom horizontal arrow
in the commutative diagram \eqref{commutative diagram from E2-degeneracy:eq},
underlies a morphism of variations of $\bQ$-mixed Hodge structure
by using the fact that
the left vertical arrow is surjective and strictly compatible
with the filtrations $L[n]$ and $F$.
Therefore we obtain the desired coincidence.

For the case $J \subsetneq I$,
we have the commutative diagram
\begin{equation}
\begin{CD}
\cO_{Y[I]}
\otimes
\iota_{J|}^{-1}\gr_m^{L(K)}R^q(f_J)_{\ast}sB_J(f)
@>>>
\gr_m^{L(K)}R^q(f_I)_{\ast}sB_I(f) \\
@VVV @VVV \\
\cO_{Y[I]}
\otimes
\iota_{J|}^{-1}E_2^{-m,q+m}(R(f_J)_{\ast}sB_J(f),L(K))
@>>>
E_2^{-m,q+m}(R(f_I)_{\ast}sB_I(f),L(K)),
\end{CD}
\end{equation}
where the vertical arrows are the identifications
\eqref{E2-degeneration for L(K):eq}
for $J$ and $I$,
and the horizontal arrows
are the morphisms induced by $\theta_{J|I}(f)$,
which are filtered isomorphisms
with respect to $L(J)$ and $L(J \setminus K)_{\rec}$
by Corollary \ref{corollary on grL's and theta}
and Lemma \ref{lemma on theta and E2}.
Then the coincidence of $L(J)[m]$ and $L(J \setminus K)_{\rec}$
under the left vertical arrow
implies the desired coincidence under the right vertical arrow.
\end{proof}

The remainder of this section is devoted to proving the following theorem.

\begin{thm}
\label{local freeness}
Let $X$ be a K\"ahler manifold
and $f:X \longrightarrow Y=\Delta^k \times S$
a proper semistable morphism.
Then the coherent $\cO_{Y[I]}$-module
\begin{equation}
\gr_F^p\gr_{m_l}^{L(J_l)} \cdots
\gr_{m_2}^{L(J_2)}\gr_{m_1}^{L(J_1)}
R^n(f_I)_{\ast}sB_I(f)
\end{equation}
is locally free of finite rank
for all $m_1,m_2, \dots, m_l,n,p$
and for all
$J_1 \subset J_2 \subset \dots \subset J_l \subset I \subset \ski$.
\end{thm}

\begin{para}
In order to prove the theorem above,
we consider the following two assertions:
\begin{itemize}
\item[$(A_l)$]
For any proper semistable morphism
$f:X \longrightarrow Y=\Delta^k \times S$
from a K\"ahler manifold $X$,
the $\cO_{Y[I]}$-module
\begin{equation}
\gr_F^p\gr_{m_l}^{L(J_l)} \cdots
\gr_{m_2}^{L(J_2)}\gr_{m_1}^{L(J_1)}
R^n(f_I)_{\ast}sB_I(f)
\end{equation}
is locally free of finite rank
for all $m_1,m_2, \dots, m_l,n,p$
and for all
$J_1 \subset J_2 \subset \dots \subset J_l \subset I \subset \ski$.
\item[$(B_l)$]
For any proper semistable morphism
$f:X \longrightarrow Y=\Delta^k \times S$
from a K\"ahler manifold $X$,
the filtration $F$ on the complex
\begin{equation}
\gr_{m_l}^{L(J_l \setminus K)}
\cdots
\gr_{m_2}^{L(J_2 \setminus K)}
\gr_{m_1}^{L(J_1 \setminus K)}
E_1^{\bullet,b}(R(f_I)_{\ast}sB_I(f),L(K))
\end{equation}
is strongly strict
for all $b, m_1, m_2, \dots, m_l$
and for all
$K \subset J_1 \subset J_2 \subset \dots \subset J_l \subset I \subset \ski$.
\end{itemize}
What we have to prove is that $(A_l)$ holds true for all $l \in \bnnZ$.
\end{para}

\begin{lem}
\label{lemma on strong strictness for F}
The filtered complex $(R(f_I)_{\ast}sB_I(f),F)$ is strongly strict
for all $I \subset \ski$.
In particular,
the $\cO_{Y[I]}$-module
$\gr_F^pR^n(f_I)_{\ast}sB_I(f)$
is locally free of finite rank
for all $n,p$,
that is, the assertion $(A_0)$ is true.
\end{lem}
\begin{proof}
For the case of $I=\emptyset$,
the strong strictness of $(Rf_{\ast}\omega_{X/Y},F)$ is proved
in \cite[Theorem (6.10)]{FujisawaLHSSV}.
The morphism $\theta_{|I}(f)$ induces an isomorphism
\begin{equation}
(\cO_{Y[I]} \otimes^L \iota_{I|}^{-1}Rf_{\ast}\omega_{X/Y},F)
\simeq
(R(f_I)_{\ast}sB_I(f),F)
\end{equation}
in the filtered derived category
by Lemma \ref{lemma on theta for direct image complexes}.
Therefore we obtain the conclusion
by Lemma \ref{lemma on strong strictness and filtered tensor product}.
\end{proof}

\begin{prop}
\label{iso from FRf*sBJ to FRf*sBI:eq}
The morphism $\theta_{J|I}(f)$ induces an isomorphism
\begin{equation}
\cO_{Y[I]} \otimes \iota_{J|I}^{-1}R^n(f_J)_{\ast}sB_J(f)
\simeq
R^n(f_I)_{\ast}sB_I(f)
\end{equation}
for every $n$,
under which we have
\begin{equation}
\cO_{Y[I]} \otimes \iota_{J|I}^{-1}F^pR^n(f_J)_{\ast}sB_J(f)
\simeq
F^pR^n(f_I)_{\ast}sB_I(f)
\end{equation}
for all $p$.
\end{prop}
\begin{proof}
Lemma \ref{lemma on theta for direct image complexes}
and the strong strictness of $(R(f_J)_{\ast}sB_J(f),F)$
implies the conclusion
as in Lemma \ref{lemma on strong strictness and filtered tensor product}.
\end{proof}

%\begin{lem}
%\label{filtered isom for E1 of J and I}
%The morphism
%\eqref{iso grL...grLtheta:eq}
%induces an isomorphism
%\begin{equation}
%\begin{split}
%F^p(\cO_{Y[I]} \otimes_{\cO_{Y[J]}}
%\gr_{m_l}^{L(J_l)}
%&\cdots
%\gr_{m_2}^{L(J_2)}\gr_{m_1}^{L(J_1)}
%E_1^{\bullet,b}(R(f_J)_{\ast}sB_J(f),L(K))) \\
%&\overset{\simeq}{\longrightarrow}
%F^p(\gr_{m_l}^{L(J_l)} \cdots
%\gr_{m_2}^{L(J_2)}\gr_{m_1}^{L(J_1)}
%E_1^{\bullet,b}(R(f_I)_{\ast}sB_I(f),L(K)))
%\end{split}
%\end{equation}
%for all $m_1, m_2, \dots, m_l, p$
%and for all $K, J_1, J_2, \dots, J_l \subset J \subset I$.
%\end{lem}
%\begin{proof}
%From the identification
%\eqref{identification of E_1 for sBI:eq},
%we can easily obtained the conclusion
%by Lemma \ref{local freeness and the commutattivity of gr}
%and Proposition \ref{iso from FRf*sBJ to FRf*sBI:eq}.
%\end{proof}

\begin{lem}
The assertion $(A_l)$ implies $(B_l)$.
\end{lem}
\begin{proof}
From the identification
\eqref{identification of E_1 for sBI:eq},
the assumption $(A_l)$ implies that
\begin{equation}
\gr_F^p\gr_{m_l}^{L(J_l \setminus K)}
\cdots
\gr_{m_2}^{L(J_2 \setminus K)}
\gr_{m_1}^{L(J_1 \setminus K)}
E_1^{a,b}(R(f_I)_{\ast}sB_I(f),L(K))
\end{equation}
is locally free of finite rank.
Thus the filtered complex
\begin{equation}
(\gr_{m_l}^{L(J_l \setminus K)}
\cdots
\gr_{m_2}^{L(J_2 \setminus K)}
\gr_{m_1}^{L(J_1 \setminus K)}
E_1^{\bullet,b}(R(f_I)_{\ast}sB_I(f),L(K)), F)
\end{equation}
is filtered perfect.
Since
\begin{equation}
\coh^a(
\gr_{m_l}^{L(J_l \setminus K)}
\cdots
\gr_{m_2}^{L(J_2 \setminus K)}
\gr_{m_1}^{L(J_1 \setminus K)}
E_1^{\bullet,b}(R(f_I)_{\ast}sB_I(f),L(K)))
\end{equation}
is locally free as in the proof
of Lemma \ref{strictness for L(J setminus K) on ss},
it is sufficient to prove that
the filtered complex
\begin{equation}
\label{the complex in question:eq}
(\bC(x) \otimes_{\cO_{Y[I]}}^L
\gr_{m_l}^{L(J_l \setminus K)}
\cdots
\gr_{m_2}^{L(J_2 \setminus K)}
\gr_{m_1}^{L(J_1 \setminus K)}
E_1^{\bullet,b}(R(f_I)_{\ast}sB_I(f),L(K)),F)
\end{equation}
is strict for all $x \in Y[I]$
by Lemma \ref{strong strictness criteria by pointwise strictness}.
Because the filtration $F$ on the complex
\begin{equation}
\gr_{m_l}^{L(J_l \setminus K)}
\cdots
\gr_{m_2}^{L(J_2 \setminus K)}
\gr_{m_1}^{L(J_1 \setminus K)}
E_1^{\bullet,b}(R(f_I)_{\ast}sB_I(f),L(K))
\left|_{Y[I]^{\ast}}\right.
\end{equation}
is strongly strict
by Corollary \ref{MHS on E1-terms},
the filtered complex
\eqref{the complex in question:eq}
is strict for $x \in Y[I]^{\ast}$
by Lemma \ref{lemma on strong strictness and filtered tensor product}.
For a point $x \in Y[I] \setminus Y[I]^{\ast}$,
we can find $I' \supset I$ such that $x \in Y[I']^{\ast}$.
Since the $\cO_{Y[I]}$-module
\begin{equation}
\gr_F^p\gr_{m_l}^{L(J_l \setminus K)}
\cdots
\gr_{m_2}^{L(J_2 \setminus K)}
\gr_{m_1}^{L(J_1 \setminus K)}
E_1^{a,b}(R(f_I)_{\ast}sB_I(f),L(K))
\end{equation}
is locally free of finite rank as above,
the morphism $\theta_{I|I'}(f)$
induces the isomorphism
\begin{equation}
\begin{split}
(\cO_{Y[I']} \otimes^L
&\gr_{m_l}^{L(J_l \setminus K)}
\cdots
\gr_{m_2}^{L(J_2 \setminus K)}
\gr_{m_1}^{L(J_1 \setminus K)}
\iota_{I|I'}^{-1}E_1^{\bullet,b}(R(f_I)_{\ast}sB_I(f),L(K)),F) \\
&\simeq
(\gr_{m_l}^{L(J_l \setminus K)}
\cdots
\gr_{m_2}^{L(J_2 \setminus K)}
\gr_{m_1}^{L(J_1 \setminus K)}
E_1^{\bullet,b}(R(f_{I'})_{\ast}sB_{I'}(f),L(K)),F)
\end{split}
\end{equation}
in the filtered derived category
by Lemma \ref{local freeness and the commutattivity of gr}.
Thus we have an isomorphism
\begin{equation}
\begin{split}
(\bC(x) &\otimes_{\cO_{Y[I]}}^L
\gr_{m_l}^{L(J_l \setminus K)}
\cdots
\gr_{m_2}^{L(J_2 \setminus K)}
\gr_{m_1}^{L(J_1 \setminus K)}
E_1^{\bullet,b}(R(f_I)_{\ast}sB_I(f),L(K)),F) \\
&\simeq
(\bC(x) \otimes_{\cO_{Y[I']}}^L
\gr_{m_l}^{L(J_l \setminus K)}
\cdots
\gr_{m_2}^{L(J_2 \setminus K)}
\gr_{m_1}^{L(J_1 \setminus K)}
E_1^{\bullet,b}(R(f_{I'})_{\ast}sB_{I'}(f),L(K)),F)
\end{split}
\end{equation}
in the filtered derived category.
Therefore the complex \eqref{the complex in question:eq}
is strict as desired.
\end{proof}

\begin{cor}
\label{condition (B0)}
The assertion $(B_0)$ holds true.
\end{cor}

\begin{lem}
\label{coincidence for F via E2-degeneracy}
Under the identification
\eqref{E2-degeneration for L(K):eq},
the filtration $F$ on the left hand side
coincides with the filtration $F_{\rec}$ on the right.
\end{lem}
\begin{proof}
From Lemma \ref{lemma on strong strictness for F},
the filtration $F$ on the complex
$R(f_I)_{\ast}\gr_m^{L(K)}sB_I(f)$
is strongly strict for all $m$.
Since $(B_0)$ holds true
by Corollary \ref{condition (B0)},
the lemma on two filtrations
\cite[Proposition (7.2.5)]{DeligneIII}
implies the conclusion.
\end{proof}

\begin{proof}[Proof of Theorem $\ref{local freeness}$]
It suffices to prove that
the assertions $(B_0)$ and $(B_l)$ implies $(A_{l+1})$.
We set $K=J_1$.
Then the complex
\begin{equation}
E_1^{\bullet,b}(R(f_I)_{\ast}sB_I(f),L(K))
\end{equation}
equipped with the filtrations
\begin{equation}
L(J_2 \setminus K), \dots, L(J_l \setminus K), L(J_{l+1} \setminus K), F
\end{equation}
satisfies the assumption
in Lemma \ref{lemma on successively strict filtrations}
by Lemma \ref{strictness for L(J setminus K) on ss}
and by the assumptions $(B_0)$ and $(B_l)$.
Therefore we have the isomorphism
\begin{equation}
\begin{split}
\coh^a(
\gr_F^p&\gr_{m_{l+1}}^{L(J_{l+1} \setminus K)}
\gr_{m_l}^{L(J_l \setminus K)}
\cdots
\gr_{m_2}^{L(J_2 \setminus K)}
E_1^{\bullet,b}(R(f_I)_{\ast}sB_I(f),L(K))) \\
&\simeq
\gr_{F_{\rec}}^p
\gr_{m_{l+1}}^{L(J_{l+1} \setminus K)_{\rec}}
\gr_{m_l}^{L(J_l \setminus K)_{\rec}}
\cdots
\gr_{m_2}^{L(J_2 \setminus K)_{\rec}}
E_2^{a,b}(R(f_I)_{\ast}sB_I(f),L(K))
\end{split}
\end{equation}
for all $a, b, m_2, \dots, m_l, m_{l+1}, p$.
Moreover, it is locally free of finite rank
by the strong strictness of $F$ in the assumption $(B_l)$.
Then the identification \eqref{E2-degeneration for L(K):eq},
Lemma \ref{coincidence for F via E2-degeneracy}
and Proposition \ref{compatiblity of L(J) and L(J setminus K) on grLK}
give us $(A_{l+1})$.
\end{proof}

\begin{cor}
For
$J_1 \subset J_2 \subset \dots \subset J_l \subset J \subset I \subset \ski$,
the morphism $\theta_{J|I}(f)$ induces an isomorphism
\begin{equation}
\begin{split}
\cO_{Y[I]} \otimes
\gr_F^p
\gr_{m_l}^{L(J_l)}
&\cdots
\gr_{m_2}^{L(J_2)}
\gr_{m_1}^{L(J_1)}
\iota_{J|I}^{-1}R^n(f_J)_{\ast}sB_J(f) \\
&\overset{\simeq}{\longrightarrow}
\gr_F^p
\gr_{m_l}^{L(J_l)}
\cdots
\gr_{m_2}^{L(J_2)}
\gr_{m_1}^{L(J_1)}
R^n(f_I)_{\ast}sB_I(f)
\end{split}
\end{equation}
for all $m_1, m_2, \dots, m_l, n, p$.
\end{cor}
\begin{proof}
Easy from Lemma \ref{local freeness and the commutattivity of gr}.
\end{proof}

\section{Main result}
\label{section for the main results}

\begin{para}
Let $X$ be a K\"ahler manifold
and $f:X \longrightarrow Y=\Delta^k \times S$
a proper semistable morphism
as in the last section.
For $i=1,2, \dots, k$,
we write $Y[i]=Y[\{i\}](=E_i)$ for short.
We use the similar convention for other objects,
$X[i]=X[\{i\}]$, $sB_i(f)=sB_{\{i\}}(f)$, $L(i)=L(\{i\})$ and so on.
\end{para}

\begin{lem}
\label{lemma for the main result for the case of |J|=1}
For any $I \subset \ski$,
the morphism $N_{j|I}(f)^l$
defined in Definition \ref{definition of N}
induces an isomorphism
\begin{equation}
\label{isomorphism for Nj:eq}
\gr_l^{L(j)}R^n(f_I)_{\ast}sB_I(f)
\overset{\simeq}{\longrightarrow}
\gr_{-l}^{L(j)}R^n(f_I)_{\ast}sB_I(f)
\end{equation}
for all $l \ge 0$ and for all $j \in I$.
In other words,
the filtration $L(j)$ coincides with
the monodromy weight filtration $W(N_{j|I}(f))$
on $R^n(f_I)_{\ast}sB_I(f)$ for $j \in I$.
\end{lem}
\begin{proof}
The morphism $\theta_{j|I}(f)$ induces the isomorphism
\begin{equation}
\cO_{Y[I]}
\otimes
\gr_l^{L(j)}\iota_{j|I}^{-1}R^n(f_j)_{\ast}sB_j(f)
\overset{\simeq}{\longrightarrow}
\gr_l^{L(j)}R^n(f_I)_{\ast}sB_I(f)
\end{equation}
for all $l$.
The morphism $\theta_{j|I}(f)$ is compatible with
the morphism $N_{j|I}(f)$ and $N_j(f)$ for all $j \in J$.
Therefore we may assume $I=\{j\}$.
Moreover,
the morphism $N_j(f)$ is compatible with
the logarithmic connection $-\nabla(j)$
on $R^n(f_j)_{\ast}sB_j(f)$,
it is sufficient to prove
that the restriction $N_j(f)|_{Y[j]^{\ast}}$
induces the isomorphism \eqref{isomorphism for Nj:eq}
over $Y[j]^{\ast}$
because $(R^n(f_j)_{\ast}sB_j(f), -\nabla(j))$
is the canonical extension of its restriction over $Y[j]^{\ast}$.
Thus we may assume that $k=1$.
Because $N_j(f)$ is compatible with the base change
as in Remark \ref{nu and base change for S},
we can reduce the problem
to the case that $S$ is a point
by Lemma \ref{lemma on base change for sB}.
Hence we obtain the conclusion
by \cite{SaitoMorihikoMHP}, \cite{Guillen-NavarroAznarCI}
and \cite{UsuiMTTS}.
\end{proof}

The following theorem is the main result of this article.

\begin{thm}
\label{theorem on the monodromy weight filtrations}
Let $f:X \longrightarrow Y=\Delta^k \times S$ be
a proper semistable morphism
and $N_{(J \setminus K)|I}(f;c)$ the morphism
defined in \eqref{definition of NJI(f;c):eq}
for $K \subset J \subset I \subset \ski$
and for $c \in \bpR^{J \setminus K}$.
Moreover, we assume that $X$ is K\"ahler.
For any non-negative integer $l$,
the morphism $N_{(J \setminus K)|I}(f;c)^l$
induces an isomorphism
\begin{equation}
\label{isomorphism NJI(f)^l:eq}
\gr_{l+m}^{L(J)}\gr_m^{L(K)}R^n(f_I)_{\ast}sB_I(f)
\overset{\simeq}{\longrightarrow}
\gr_{-l+m}^{L(J)}\gr_m^{L(K)}R^n(f_I)_{\ast}sB_I(f)
\end{equation}
for all $m,n \in \bZ$.
In particular,
$N_{J|I}(f;c)^l$ induces the isomorphism
\begin{equation}
\gr_l^{L(J)}R^n(f_I)_{\ast}sB_I(f)
\overset{\simeq}{\longrightarrow}
\gr_{-l}^{L(J)}R^n(f_I)_{\ast}sB_I(f)
\end{equation}
for all $J \subset I$
by setting $K=\emptyset$.
\end{thm}
\begin{proof}
The morphism $\theta_{J|I}(f)$ induces the isomorphism
\begin{equation}
\cO_{Y[I]}
\otimes
\gr_{l+m}^{L(J)}\gr_m^{L(K)}\iota_{J|I}^{-1}R^n(f_J)_{\ast}sB_J(f)
\overset{\simeq}{\longrightarrow}
\gr_{l+m}^{L(J)}\gr_m^{L(K)}R^n(f_I)_{\ast}sB_I(f)
\end{equation}
for all $l,m,n$.
Moreover the morphism $\theta_{J|I}(f)$ is compatible with
the morphism $N_{j|I}(f)$ for all $j \in J$.
Therefore we may assume $J=I$.
Moreover, we may assume that $I=\ski$ as before
by using the fact that
$(R^n(f_I)_{\ast}sB_I(f),(-1)^{|I|}\nabla(I))$ is the canonical extension
of its restriction over $Y[I]^{\ast}$.
Hence it suffices to prove that the morphism
$N_{\overline{K}|}(f;c)^l$ induces an isomorphism
\begin{equation}
\gr_{l+m}^L\gr_m^{L(K)}R^nf_{0\ast}sB(f)
\overset{\simeq}{\longrightarrow}
\gr_{-l+m}^L\gr_m^{L(K)}R^nf_{0\ast}sB(f)
\end{equation}
for all $l,m,n$ with $l \ge 0$
and for $c \in \bpR^{\overline{K}}$,
where $\overline{K}=\ski \setminus K$.

Now we proceed by induction on $k$.
The case of $k=1$ is already proved
in Lemma \ref{lemma for the main result for the case of |J|=1}.
So we assume $k \ge 2$.

For a non-empty subset $K \subset \ski$,
we have the identification
\begin{equation}
\label{(grLKsBI,LJ) for nu:eq}
(\gr_m^{L(K)}sB(f),L(\overline{K}))
\simeq
\bigoplus_{(q,\Gamma)}
(sB(f_{\Gamma})[-m-2|q|],L)
\end{equation}
by Lemma \ref{lemma on grL(K)sB},
where the index $(q,\Gamma)$ runs through the set
\begin{equation}
\label{index set for grLKsBI for nu:eq}
\{(q,\Gamma); q \in \bnnZ^K,
\Gamma \in S_{m+2|q|+|K|}^{\ge q+e_K}(\Lambda_K)\}
\end{equation}
as in Lemma \ref{lemma on grL(K)sB}.
Under the identification
\eqref{(grLKsBI,LJ) for nu:eq} above,
the morphism
\begin{equation}
\gr_m^{L(K)}\nu_{j|}(f):
\gr_m^{L(K)}sB(f)
\longrightarrow
\gr_m^{L(K)}sB(f)
\end{equation}
is identified with
\begin{equation}
\bigoplus \nu_{j|}(f_{\Gamma})[-m-2|q|]
\end{equation}
for any $j \in \overline{K}$,
where the direct sum is taken over the same index set
as \eqref{index set for grLKsBI for nu:eq}.
Thus the morphism $\gr_m^{L(K)}N_{\overline{K}|}(f;c)^l$
induces an isomorphism
\begin{equation}
\gr_l^{L(\overline{K})}E_1^{-m,n+m}(Rf_{0\ast}sB(f),L(K))
\overset{\simeq}{\longrightarrow}
\gr_{-l}^{L(\overline{K})}E_1^{-m,n+m}(Rf_{0\ast}sB(f),L(K))
\end{equation}
for all $l,m,n$ with $l \ge 0$
because of the identifications
\eqref{(grLKsBI,LJ) for nu:eq} and
\begin{equation}
E_1^{-m,n+m}(Rf_{0\ast}sB(f),L(K))
\simeq
R^nf_{0\ast}\gr_m^{L(K)}sB(f),
\end{equation}
together with the induction hypothesis.
Then $\gr_m^{L(K)}N_{\overline{K}|}(f;c)^l$ induces an isomorphism
\begin{equation}
\gr_l^{L(\overline{K})_{\rec}}E_2^{-m,n+m}(Rf_{0\ast}sB(f),L(K))
\overset{\simeq}{\longrightarrow}
\gr_{-l}^{L(\overline{K})_{\rec}}E_2^{-m,n+m}(Rf_{0\ast}sB(f),L(K))
\end{equation}
because $L(\overline{K})$ is strictly compatible
with the morphism $d_1$ of the $E_1$-terms
by Lemma \ref{strictness for L(J setminus K) on ss}.
Therefore we obtain the desired isomorphism
by Proposition \ref{compatiblity of L(J) and L(J setminus K) on grLK}.

Now the case $K=\emptyset$ is remained.
The morphism $N_{\overline{\{j\}}|}(f)^l$ induces an isomorphism
\begin{equation}
\gr_{l+m}^L\gr_m^{L(j)}R^nf_{0\ast}sB(f)
\overset{\simeq}{\longrightarrow}
\gr_{-l+m}^L\gr_m^{L(j)}R^nf_{0\ast}sB(f)
\end{equation}
for all $l,m,n$ with $l \ge 0$ as proved above.
Since we know the coincidence $L(j)=W(N_{j|}(f))$
by Lemma \ref{lemma for the main result for the case of |J|=1},
the monodromy weight filtration $W(N(f;c))$
satisfies the same property, that is,
the morphism $N_{\{\overline{j}\}|}(f)^l$ induces an isomorphism
\begin{equation}
\gr_{l+m}^{W(N(f;c))}\gr_m^{L(j)}R^nf_{0\ast}sB(f)
\overset{\simeq}{\longrightarrow}
\gr_{-l+m}^{W(N(f;c))}\gr_m^{L(j)}R^nf_{0\ast}sB(f)
\end{equation}
for all $l,m,n$ with $l \ge 0$
by \cite[(3.3) Theorem]{Cattani-Kaplan}
(see also \cite[(3.12) Theorem]{Steenbrink-Zucker}).
Then the uniqueness of the relative monodromy weight filtrations
implies the coincidence $L=W(N(f;c))$.
Therefore we conclude that the morphism $N(f;c)^l$ induces an isomorphism
\begin{equation}
\gr_l^LR^nf_{0\ast}sB(f)
\overset{\simeq}{\longrightarrow}
\gr_{-l}^LR^nf_{0\ast}sB(f)
\end{equation}
for all $l \ge 0$ by the definition of $W(N(f;c))$.
\end{proof}

\begin{thm}
Let $X$ be a K\"ahler manifold
and $f:X \longrightarrow Y=\Delta^k \times S$
a proper semistable morphism.
For any $I \subset \ski$,
the variation of $\bQ$-mixed Hodge structure
\eqref{VMHS for I:eq} on $Y[I]^{\ast}$
is admissible in $Y[I]$.
\end{thm}
\begin{proof}
Let $\varphi: \Delta \longrightarrow Y[I]$ be a holomorphic map
such that $\varphi^{-1}(Y[I]^{\ast})=\pd$.
Take $I' \supsetneq I$
such that $\varphi(0) \in Y[I']^{\ast}$.
For every $j \in I' \setminus I$,
$c_j$ denotes the order of zero of the holomorphic function
$\varphi^{\ast}t_j$
at the origin $0 \in \Delta$.
Then $c_j$ is a positive integer for all $j \in I' \setminus I$.

Now we consider
a locally free $\cO_{\Delta}$-module of finite rank
with an integrable logarithmic connection
$(\varphi^{\ast}R^n(f_I)_{\ast}sB_I(f), (-1)^{|I|}\varphi^{\ast}\nabla(I))$
on $\Delta$.
Since $\gr_F^p\gr_m^{L(I)}R^n(f_I)_{\ast}sB_I(f)$
is locally free of finite rank,
we have
\begin{equation}
\gr_F^p\gr_m^{L(J)}\varphi^{\ast}R^n(f_I)_{\ast}sB_I(f)
\simeq
\varphi^{\ast}\gr_F^p\gr_m^{L(J)}R^n(f_I)_{\ast}sB_I(f) ,
\end{equation}
which is locally free $\cO_{\Delta}$-module of finite rank.
By using the morphism $\theta_{I'|I}(f)$,
we have the identification
\begin{equation}
\bC(0) \otimes \varphi^{\ast}R^n(f_I)_{\ast}sB_I(f)
\simeq
\bC(\varphi(0)) \otimes R^n(f_I)_{\ast}sB_I(f)
\simeq
\bC(\varphi(0)) \otimes R^n(f_{I'})_{\ast}sB_{I'}(f)
\end{equation}
under which 
the equality
\begin{equation}
\res_{0}((-1)^{|I|}\varphi^{\ast}\nabla(I))
=\sum_{j \in I' \setminus I}c_j\res_j((-1)^{|I|}\nabla(I))
=(-1)^{|I|}N_{(I' \setminus I)|I'}(f;c)
\end{equation}
is easily seen.
Thus $\res_{0}((-1)^{|I|}\varphi^{\ast}\nabla(I))$ is nilpotent.
Therefore
\begin{equation}
(\varphi^{\ast}R^n(f_I)_{\ast}sB_I(f),(-1)^{|I|}\varphi^{\ast}\nabla(I))
\end{equation}
is of unipotent monodromy around $0 \in \Delta$
and the canonical extension of its restriction over $\pd$.
Moreover, $L(I')$ gives us
the relative monodromy weight filtration
of $\res_0((-1)^{|I|}\varphi^{\ast}\nabla(I))$
with respect to $L(I)$ on
\begin{equation}
\bC(0) \otimes \varphi^{\ast}R^n(f_I)_{\ast}sB_I(f)
\simeq
\bC(\varphi(0)) \otimes R^n(f_{I'})_{\ast}sB_{I'}(f)
\end{equation}
by Theorem \ref{theorem on the monodromy weight filtrations}.
\end{proof}

\providecommand{\bysame}{\leavevmode\hbox to3em{\hrulefill}\thinspace}
\providecommand{\MR}{\relax\ifhmode\unskip\space\fi MR }
% \MRhref is called by the amsart/book/proc definition of \MR.
\providecommand{\MRhref}[2]{%
  \href{http://www.ams.org/mathscinet-getitem?mr=#1}{#2}
}
\providecommand{\href}[2]{#2}

\end{document}